\documentclass{article}
\usepackage[a4paper,
            left=0.67in,
            right=0.67in,
            top=1in,
            bottom=1in,
            footskip=.25in]{geometry}

\usepackage[utf8]{inputenc}
\usepackage[english]{babel}
\usepackage{pdfpages} 
\usepackage{graphicx}
\usepackage{siunitx}
\usepackage{hyperref}
\usepackage{url}
\usepackage{amsmath}
\usepackage{cleveref}
\usepackage{amssymb}
\usepackage[a4paper]{geometry}
\usepackage{amsthm}
\usepackage{cases}
\usepackage{yfonts}
\usepackage{tikz}
\usepackage{anyfontsize}
\usepackage{appendix}
\usepackage{enumitem}
\usepackage{amsfonts}
\usepackage{dsfont}
\usepackage{comment}
\usepackage{mathtools}
\usepackage{bbm}

\newcommand*\samethanks[1][\value{footnote}]{\footnotemark[#1]}

\usepackage{tikz}

  \newcommand{\M}{{\mathcal M}}
  
\newcommand{\br}{\mathbb{R}}
\newcommand{\bz}{\mathbb{Z}}
\newcommand{\bn}{\mathbb{N}}
\newcommand{\bt}{\mathbb{T}}

\newcommand{\om}{\omega}

\newcommand{\abs}[1]{\left\lvert #1\right\rvert}
\newcommand{\norm}[1]{\left|\left| #1 \right|\right|}
\newcommand{\jp}[1]{\left\langle  #1\right\rangle }

\newtheorem{thm}{Theorem}[section]
\newtheorem{lem}[thm]{Lemma}
\newtheorem{prop}[thm]{Proposition}

\newtheorem{rem}[thm]{Remark}

\numberwithin{equation}{section}

\usepackage{scalerel,stackengine}
\stackMath
\newcommand\reallywidehat[1]{%
\savestack{\tmpbox}{\stretchto{%
  \scaleto{%
    \scalerel*[\widthof{\ensuremath{#1}}]{\kern-.6pt\bigwedge\kern-.6pt}%
    {\rule[-\textheight/2]{1ex}{\textheight}}%WIDTH-LIMITED BIG WEDGE
  }{\textheight}% 
}{0.5ex}}%
\stackon[1pt]{#1}{\tmpbox}%
}

\title{Scattering problem for Vlasov-type equations on the $d$-dimensional  torus with Gevrey data}

\author{
Dario Benedetto
  \thanks{Università di Roma ‘La Sapienza’, 
  Dipartimento di Matematica, Piazzale A. Moro 2, 00185 Roma, Italy.
   \\ Email: \textsf{  benedetto@mat.uniroma1.it},
  \textsf{caglioti@mat.uniroma1.it}}
  \and
 Emanuele Caglioti
  \samethanks
  \and
Antoine Gagnebin
  \thanks{ETH Z\"urich, Department of Mathematics, R\"amistrasse 101, 8092 Z\"urich, Switzerland. \\Email: \textsf{antoine.gagnebin@math.ethz.ch},
  \textsf{mikaela.iacobelli@math.ethz.ch}}
  \and
Mikaela Iacobelli
  \samethanks
\qquad \qquad
Stefano Rossi
  \thanks{University of Z\"urich,
  Institute of Mathematics
  Winterthurerstrasse 190
   8057 Z\"urich, Switzerland.
  \\Email: \textsf{stefano.rossi@math.uzh.ch}}
  \and
}
\date{}

\begin{document}

\maketitle

\begin{abstract}
    
In this article, we consider Vlasov-type equations describing the evolution of single-species type plasmas, such as those composed of electrons (Vlasov-Poisson) or ions (screened Vlasov-Poisson/Vlasov-Poisson with massless electrons). We solve the final data problem on the torus $\mathbb{T}^d$, $d \geq 1$, by considering asymptotic states of regularity Gevrey-$\frac{1}{\gamma}$ with $\gamma>\frac13$, small perturbations of homogeneous equilibria satisfying the Penrose stability condition. This extends to the Gevrey perturbative case, and to higher dimension, the scattering result in analytic regularity obtained by E. Caglioti and C. Maffei in \cite{EC_CM}, and answers an open question raised by J. Bedrossian in \cite{JB_notes}.

\end{abstract}

\section{Introduction}
\label{section Introduction}

In this paper, we investigate the asymptotic behaviour of systems modelling dilute,  collisionless, non-relativistic plasmas in a periodic domain.

In these systems, negatively charged electrons and positively charged ions have significantly different masses - the mass ratio being of order $10^3$, see e.g. \cite{plasma_2003}. This dissimilarity leads to a separation in the relevant timescales of evolution for the different species. Consequently, it is reasonable to describe these species' evolution separately and use different models for their dynamics.  We englobe such models in a unique framework given by a general Vlasov-type equation: 
\begin{eqnarray}
\label{VP}
  \left \{
  \begin{array}{l}
  \partial_t f(t,x,v) + v \cdot \nabla_x f(t,x,v) - E_f(t,x)\cdot \nabla_v f(t,x,v)= 0, \\
  E_f(t,x) = - \nabla V_f(t,x), \qquad -\Delta V_f(t,x) + \beta V_f(t,x) + h(V_f)(t,x) =\int_{\mathbb{R}^d} f(t,x,v) dv - 1.
  \end{array}
  \right.
\end{eqnarray}
Here, we focus on the periodic setting $x \in \bt^d:= \br^d / (2 \pi \bz)^d$. The unknown $f(t,x,v)$ is the distribution function of particles of a given species 
at time $t$, position $x$, and velocity $v$, where $(t,x,v) \in \mathbb{R} \times \mathbb{T}^d\times \mathbb{R}^d$. We denote by $E_f$ the mean-field force generated by the distribution of particles. 

\vspace{0.5cm}     

We will consider different choices of parameters $\beta\ge 0$ and analytic functions $h$.
In this way, one could study several models for plasma physics systems.
For example, the main kinetic equation used to model plasmas is the so-called  Vlasov-Poisson (VP) system, which corresponds to the case $\beta=h=0$.  
It is a well-known model in kinetic theory, which
describes the motion of electrons (i.e.,  $f (t, x, v)$ is the distribution function for the electrons) in plasmas when we neglect collisions and magnetic effects and consider the ions as a stationary background.

The (VP) system has been extensively studied over the past decades, and a vast literature of results is already available on the existence of global classical and weak solutions with various conditions on the initial data.  For global well-posedness of classical solutions on the whole space,  we refer to the work of S. V. Iordanskii \cite{Iordanskii} in dimension one, S. Ukai and T. Okabe \cite{Ukai_Okabe} in dimension two, and P.-L. Lions and B. Perthame \cite{Lions_Perthame}, K. Pfaffelmoser \cite{Pfaffelmoser} in the three-dimensional case, (see also \cite{Bardos_Degond_2, Schaffer}).  For weak solutions on the whole space, we refer to the work of A. A. Arsenev \cite{Arsenev}, (see also \cite{Bardos_Degond_1, Bardos_Degond_Golse, Horst_Hunze}), while J. Batt and G. Rein \cite{Batt_Rein} proved the global well-posedness on the torus  (see also \cite{Chen, Pallard}).  Finally, Loeper \cite{Loeper} made a significant improvement to the uniqueness theory  (see also \cite{Miot}, \cite{HM_18}, \cite{MI_2022}.

While the (VP) system describes the evolution of electrons, on the ion timescale electrons move faster than ions, and it becomes relevant to consider electron-electron collisions in the model.
Thus, it is natural to assume that the electron distribution is already in thermal equilibrium on the ion timescale,
and the electron density is given by a Maxwell-Boltzmann distribution.  
The resulting model under these assumptions is the Vlasov-Poisson system for massless electrons (VPME), sometimes called ionic Vlasov-Poisson. 
This model can be obtained rigorously as the mass ratio between ions and electrons becomes small. We refer to \cite{Bardos} for a mathematical proof of the above result and to \cite{Griffin_Iacobelli_summary} for a more thorough introduction to this model. 
The (VPME) system consists of a Vlasov equation coupled to a nonlinear Poisson equation that describes how the fluctuation between the ion distribution and the Maxwell-Boltzmann electron distribution generates the electric potential. 
In our notation, (VPME) can be obtained by setting $\beta = 1$ and $h(U) = e^U - 1 - U$.

The (VPME) system has received less mathematical attention than the (VP) system due to the additional difficulty in the nonlinear Poisson coupling.  
F. Bouchut \cite{Bouchut} was the first to construct global weak solutions in the whole space in dimension three. In the one-dimensional setting, weak solutions were constructed globally in time by D. Han-Kwan and M. Iacobelli \cite{Han-Kwan_Iacobelli}.  M. Griffin-Pickering and M. Iacobelli have obtained the first result of global well-posedness in the case of the whole space in dimension three \cite{Griffin_Iacobelli_R} and of the torus in dimension two and three \cite{Griffin_Iacobelli_torus}.  Moreover, L. Cesbron and M. Iacobelli showed the well-posedness of (VPME) in bounded domains \cite{Cesbron_Iacobelli}. We refer to \cite{Griffin_Iacobelli_summary} for an exhaustive review of the global well-posedness theory of (VPME).

In the setting of our system \eqref{VP}, another important example of plasma physics model is the screened (VP) system.  
This model can be obtained by choosing $\beta = 1$, and $h(U) = 0$ and corresponds to a first-order approximation of the exponential term $e^U$ by $1+U$ in the Poisson coupling of the (VPME) system. Physically speaking, it can be showed that this approximation holds true when the electric energy of the electrons is much smaller than the kinetic energy. 
The weak global existence of this system is explored in \cite[Theorem 2.1]{HKD}. Readers can refer to \cite[ Section 1.1 and 1.2]{HKD_HDR} for a more in-depth understanding of this system. 

The analysis we conduct here is focused on the study of the system \eqref{VP} in a general framework that covers the three plasma models mentioned earlier. To be more specific, we assume $\beta \geq 0$, and we consider $h: (-R,R) \rightarrow \mathbb{R}$ to be analytic, with an analyticity radius denoted by $R$, and it satisfies the condition $h(x) = O(x^2)$ as $x$ approaches zero. It is worth noting that in the (VP),  (VPME), or screened (VP) system, we have $R = \infty$.

\vspace{0.5cm}

In this work, we are interested in the long-time behaviour of solutions of Vlasov-type equations on the periodic domain $\mathbb{T}^d$.
%Generally, characterizing the long-time dynamics of solutions to nonlinear first-order PDEs of transport-type is a challenging problem, and the state of research in this regard currently contains only partial results and is far from a complete and satisfactory description of the asymptotic dynamics.
Already in $1946$,  L. Landau in \cite{Landau} observed that
the Vlasov-Poisson equation (i.e.,  \eqref{VP} with $\beta = 0$, and $h=0$) linearized around a Maxwellian equilibrium is exactly solvable and that, for analytic initial data, 
the electric field decays exponentially fast to zero, so the flow governed by the mean-field force is asymptotically free.
This electric field decay has been experimentally observed in plasmas only eighteen years later by J. H. Malmberg and C. B. Wharton in \cite{MW_64}.

Today, what is now called Landau damping, is a well-known collisionless relaxation phenomenon studied in plasma physics literature, where several works analyse and discuss the pioneering result of L. Landau \cite{Landau}. 
We can mention the work of O. Penrose \cite{Penrose}, who improved the result of L. Landau for general analytic spatially homogeneous equilibria.
Nevertheless, the extension from the linear to the true nonlinear case has proved to be particularly difficult for the mathematical theory. 

The first one-dimensional nonlinear result was the one of E. Caglioti and C. Maffei in \cite{EC_CM}. They solved the final data problem in analytic regularity, proving the existence of solutions for large times in a non-perturbative regime, using a Lagrangian approach and fixed point techniques. Subsequently, H. J. Hwang and J. L. L.  Velázquez \cite{Hw_Ve} gave a proof with less restrictive hypotheses.  
Concerning this approach, see also the work of A. Gagnebin \cite{G_23} who extended the work of E. Caglioti and C. Maffei \cite{EC_CM} to the one-dimensional ions dynamic (VPME) - i.e.,  \eqref{VP} with $\beta = 1$, and $h(U)=e^U - 1 - U$.

The solution to the Landau damping for the Vlasov-Poisson system in arbitrary dimension
was finally solved in $2011$ in the work of C. Mouhot and C. Villani \cite{CM_CV}. 
They treated the perturbative regime around homogeneous Penrose-stable equilibria considering analytic and Gevrey initial data, with a Gevrey index close to one. 
The proof, which uses a Newton scheme 
combining both an Eulerian and a Lagrangian approach, 
was subsequently simplified and extended to Gevrey $\gamma>\frac13$ regular initial data by J. Bedrossian, N. Masmoudi, and C. Mouhot in \cite{BMM_13},
using paraproduct decomposition techniques instead of the Newton scheme (it is worth mentioning that their results also includes the screened (VP) case - i.e.,  \eqref{VP} with $\beta = 1$, and $h=0$). 
Recently, E. Grenier, T. Nguyen, and I. Rodianski 
 further simplified the proof in \cite{GNR}
 using a different functional setting, better characterizing the
 invertibility of the linearized term 
 and simplifying the non-linear analysis.
Concerning the ions dynamics, we mention the recent work of A. Gagnebin and M. Iacobelli \cite{GI_23} where they solve the Cauchy problem, proving Landau damping for the (VPME) system on the torus $\mathbb{T}^d$, also for $\gamma>\frac13$.

%The critical regularity Gevrey $\gamma=\frac13$ is conjectured to be critical, this is due to the nonlinear resonance phenomenon called plasma echoes, already observed in \cite{Malmberg_and_co} and mathematically studied in \cite{CM_CV} and in the following works.
J. Bedrossian in \cite{JB} proved that the proof of Mouhot-Villani cannot be extended to high Sobolev spaces in the case of gravitational interactions, by showing inflation of Sobolev norms for solutions that exhibit arbitrarily many isolated plasma echoes.
In addition, Z. Lin and C. Zeng in \cite{LZ_2011} and \cite{LZ_2012} showed the existence of periodic travelling BGK waves within any small neighborhood in $H^{\sigma}$ with $\sigma<\frac{3}{2}$ of a general homogeneous equilibrium, thus proving that Landau damping cannot hold for small Sobolev regularity.
The problem is different in other related equations, such as Vlasov-HMF equation, where this phenomenon is absent (for the Vlasov-HMF case, see the works of E. Faou and F. Rousset in \cite{FR_2016} for the Landau damping in Sobolev spaces and
the recent one of D. Benedetto, E. Caglioti, and S. Rossi \cite{BCR} where the final data and initial data Landau damping problems with analytic regularity are compared).

While in the periodic case on $\mathbb{T}^d$, the phase mixing effect holds, allowing to trade regularity for decay, in the whole space $\mathbb{R}^d$ the problem is radically different. 
In particular, due to the dispersive mechanism the decay of the electric field is only algebraic,
moreover in general the Penrose stability condition does not hold.
In this regard, see the work of A. D. Ionescu, B. Pausader, X. Wang and K. Widmayer in \cite{IPWW_2022} and the one of P. Flynn, Z. Ouyang, B. Pausader and K. Widmayer in \cite{FOPW_2023} for the dynamics around vacuum for (VP) system, the one of B. Pausader and K. Widmayer in \cite{PW_2021} for the stability of a charged particle, and the ones of A. D. Ionescu, B. Pausader, X. Wang and K. Widmayer in \cite{IPWW_2023} and \cite{IPWW_2024} for the stability of the Poisson equilibrium. 
In the case of the screened (VP) system on the whole space, Penrose stability condition holds and this allows one to prove Landau damping even with low Sobolev regularity, 
see the work of J. Bedrossian, N. Masmoudi and C. Mouhot in \cite{BMM_16} and of D. Han-Kwan, T. Nguyen and F. Rousset in \cite{HKNR} for the $d=3$ case, and the works of L. Huang, Q.-H. Nguyen, and Y. Xu in \cite{huang_2d} and \cite{huang_sharp_2022}.

\vspace{0.5cm}

Let us consider solutions of \eqref{VP} of the form 
\begin{equation}
\label{solution_f}
	f(t,x,v) = \mu(v) + r(t,x,v),
\end{equation}
where $r$ is the perturbation and
$\mu $ an analytic homogeneous equilibrium 
satisfying a suitable stability condition (see \eqref{Penrose}).
The goal of this work is to study the final data problem where an asymptotic datum $g_\infty(x,v)$ with Gevrey regularity is given (see (\ref{norm_g_infty}) for more details). 
By inserting \eqref{solution_f} into our system \eqref{VP}, we find that the perturbation $r(t,x,v)$ verifies the equation 
\begin{eqnarray}
\label{nlpvpme}
  \left \{
  \begin{array}{l}
  \partial_t r(t,x,v) + v \cdot \nabla_x r(t,x,v) -q E_r(t,x) \cdot \nabla_v (\mu(v) + r(t,x,v)) = 0, \\
  E_r(t,x)= - \nabla U_r(t,x), \qquad - \Delta U_r(t,x) + \beta U_r(t,x) + h(U_r(t,x))= \int_{\mathbb{R}^d} r(t,x,v) \ dv.
  \end{array}
  \right.
\end{eqnarray}
We prove the existence of perturbative solutions $r(t,x,v)$ such that the electric field asymptotically vanishes and it holds
\begin{equation}
\label{freescatt}
	\|  r(t,x+vt,v) - g_\infty(x, v) \|_{L^{\infty} (\bt^d \times \br^d)} \longrightarrow 0, \qquad \mbox{as } t\rightarrow + \infty,
\end{equation}
so that,
\begin{equation*}
	\|  f(t,x+vt,v) - f_\infty(x,v) \|_{L^{\infty} (\bt^d \times \br^d)} \longrightarrow 0, \qquad \mbox{as } t\rightarrow + \infty.
\end{equation*}
with $f_\infty(x,v):=\mu (v) + g_\infty(x, v)$.
{This answers the question about the scattering map addressed by J. Bedrossian in his review on Landau damping \cite{JB_notes}.}

\vspace{0.5cm}

While completing this work, 
we learned that A. D. Ionescu, B. Pausader, X. Wang, and K. Widmayer were independently working on the solution of the
Landau damping and the final data problem in the
case of the (VP) system with critical regularity $\gamma = \frac{1}{3}$,
thus constructing the scattering map for the problem and obtaining results similar to ours. These results were published in the very recent preprint \cite{IPWW_2024_arxvi}. Compared with them, in our work we do not treat the case $\gamma = \frac{1}{3}$, but we deal with a more general Vlasov-type equation allowing a nonlinearity in the Poisson coupling and therefore including the (VPME) system.

\subsection{Notation}
\label{subsection Notation}

Let $(k,\eta) \in \mathbb{Z}^d\times \mathbb{R}^d$, we denote the Fourier coefficient of a general function $\rho: \mathbb{R}_t \times \mathbb{T}_x^d \to \mathbb{R}^+$ by 
\begin{equation*}
  \widehat\rho_t(k)\equiv\widehat{\rho}(t,k) :=\int_{\mathbb{T}^d} \rho (t,x) e^{-ik \cdot x} \ dx,
\end{equation*}
and we have the usual formula
\begin{equation*}
\label{reconstruction formula 1}
	\rho(t,x) = \frac{1}{(2 \pi)^d}\sum_{k \in \mathbb{Z}^d} \widehat{\rho}_t(k) e^{ik \cdot x}.
\end{equation*}
Analogously, we write the Fourier transform of a general function $f: \mathbb{R}_t \times \mathbb{T}_x^d \times \mathbb{R}_v^d \to \mathbb{R}^+$ by 
\begin{equation*}
  \widehat{f}_{t}(k, \eta)\equiv\widehat{f}(t,k,\eta) := \int_{\mathbb{T}^d \times \mathbb{R}^d} f (t,x,v) e^{-ik \cdot x} e^{-i \eta \cdot v} \ dx \ dv,
\end{equation*}
with the reconstruction formula
\begin{equation*}
\label{reconstruction formula 2}
	f(t,x,v) = \frac{1}{(2 \pi)^{d} } \sum_{k \in \mathbb{Z}^d} \int_{\mathbb{R}^d} \widehat{f}_{t}(k,\eta)  e^{ik \cdot x} e^{i \eta \cdot v} \ d\eta.
\end{equation*}
The Japanese bracket is written as follows: $\langle k, \eta \rangle = \sqrt{1 + \vert k \vert^2 + \vert \eta \vert^2}$. For a function which depends on variables $x$ and $v$, we define the Fourier multiplier operator $A_t(\nabla_x, \nabla_v)$ as:
\begin{align}
\label{mult1}
    \reallywidehat{A_t(\nabla_x, \nabla_v) f} (t,k,\eta)\equiv \widehat{Af}(t,k,\eta) := A_t(k,\eta) \widehat{f} (t,k,\eta),
\end{align}
where 
$    A_t(k,\eta):= e^{\lambda (t) \jp{k, \eta}^\gamma} \jp{k, \eta}^\sigma$,
with Gevrey index $\gamma \in \left( \frac{1}{3} , 1 \right]$, Sobolev correction $\sigma > d + 10$, and Gevrey regularity radius $\lambda(t)$ defined as follows
\begin{align}
\label{def_lambda}
    \lambda(t) := \lambda_\infty  - C \jp{t}^{-\delta},
\end{align}
where $\lambda_\infty>0$ and $\delta \ll 1$, $C>0$ such that $\lambda(0)>0$.

For a function which depends only on the variable $x$, we define the analogous multiplier by setting $\eta = kt$ in \eqref{mult1}:
\begin{align*}
    \reallywidehat{A_t(\nabla_x, t\nabla_x) \rho} (t,k) \equiv \widehat{A\rho} (t,k) := A_t(k,kt) \widehat{\rho} (t,k).
\end{align*}
Moreover, we define the following Fourier multiplier:
\begin{align*}
\label{Fourier_multiplier_B}
    \reallywidehat{B_t(\nabla_x, \nabla_v) f} (t,k,\eta) \equiv \widehat{Bf}(t,k,\eta) := \jp{k,\eta} A_t (k,\eta) \widehat{g}(t,k,\eta).
\end{align*}
To quantify the regularity of the distribution function and the decay of the electric field, we use generalised Gevrey norms for our main unknowns. 
More precisely, we will consider the following norms: given a Gevrey index $ \gamma \in \left(\frac{1}{3}, 1\right]$, $\sigma >  10 + d$, where $d$ is the dimension, and a function 
$\rho: \mathbb{R}_t \times \mathbb{T}_x^d \to \mathbb{R}^+ $, we introduce
\begin{equation*}
\label{GenF}
	\left\Vert  A_t(\nabla_x,t \nabla_x) 
   \rho_t
   \right\Vert_{ L_x^2}^2 = \sum_{k \in \mathbb{Z}^d} e^{2\lambda(t) \langle k, kt \rangle^{\gamma}} \vert \widehat{\rho}_t (k) \vert^2 \langle k, kt \rangle^{2\sigma},
\end{equation*}
and, for $f: \mathbb{R}_t\times\mathbb{T}_x^d \times \mathbb{R}_v^d \to \mathbb{R}^+$, 
$j \in \mathbb{N}^d$ a multi-index and $M>\frac{d}{2}$, 
\begin{equation}
\label{sec1:normg}
	\left\Vert \jp{v}^M B_t(\nabla_x,\nabla_v) 
   f_t
   \right\Vert_{ L_{x,v}^2}^2 \approx \sum_{\vert j \vert \leq M}
   \norm{B_t(k,\eta) \partial^j_\eta \widehat{f}}_{L^2_{k,\eta}}
   =\sum_{\vert j \vert \leq M}
   \sum_{k \in \mathbb{Z}^d} \int_{\mathbb{R}^d} e^{2\lambda(t) \langle k, \eta \rangle^{\gamma} } \vert \partial_{\eta}^j \widehat{f}_t (k,\eta) \vert^2 \langle k, \eta \rangle ^{2 \sigma+2} \ d\eta.
\end{equation}
See computations \eqref{sec3:equi_norm_1} and \eqref{sec3:equi_norm_2} for a justification of the equivalence of the norms. With a little abuse of notation, we also define
\begin{align}
\label{sec1:A_z}
    A_z(k,\eta):=e^{z \langle k, \eta \rangle^{\gamma}}  \langle k, \eta \rangle^{\sigma}.
\end{align}
In this paper, $C$ denotes a generic constant that may change from line to line and $\bz_*^d := \bz^d \setminus \{0\}$.

\subsection{Main result}
Before stating our main Theorem \ref{Main_thm}, we introduce the technical assumptions satisfied by the homogeneous equilibrium $\mu(v):$ 
\begin{itemize}
\item[(H1)] $\mu$ is real analytic and satisfies for $M>\frac{d}{2}$, $\lambda>\lambda_\infty>0$, and for all multi-index $j \in \bn^d$
\begin{equation}
\label{norm_mu}
	\sum_{|j|\le M}
    \norm{
    e^{\lambda \langle \eta \rangle} 
    \vert \partial_\eta^j\widehat{\mu}(\eta)\vert}
    _{L^\infty_\eta}  < \infty,
\end{equation}
where $\lambda_\infty$ is introduced in \eqref{def_lambda};

\item[(H2)] $\mu $ satisfies the following Penrose stability condition: there exists a small positive constant $\kappa_0$ such that
\begin{equation}
\label{Penrose}
  \inf_{k \in \mathbb{Z}_*^d;
  \Re \tau \geq 0 }{\left| 1 + \frac{\vert k \vert ^2}{\beta + \vert k \vert^2} \int_0^{+\infty}  t\widehat{\mu}(kt) e^{-\tau t} \ dt \right|} \geq \kappa_0 > 0,
\end{equation}
where $\tau \in \mathbb{C}$ and $\Re \tau$ is the real part of $\tau$;
\item[(H3)] $\int_{\mathbb{R}^d} \mu (v) \ dv = 1$.
\end{itemize}

\begin{thm}\label{Main_thm}
Let us consider the system (\ref{nlpvpme}) with $\beta \geq 0$ a non-negative constant and $h: \br \rightarrow \br$ an analytic function such that $h(x)= \mathcal{O} (x^2)$ when $x$ goes to zero.
Let $\mu$ be a homogeneous equilibrium that satisfies the hypotheses (H1)-(H3) with $\lambda>0$ as in \eqref{norm_mu} and let $g_\infty$ be a Gevrey function of mean zero such that for $\sigma > 10+d$, $M > \frac{d}{2}$, $b>10$ and $\gamma \in \left(\frac{1}{3} , 1\right)$, 
\begin{equation}
\label{norm_g_infty}
	\left \Vert
   \langle v \rangle^M e^{\lambda \langle \nabla_x, \nabla_v  \rangle^\gamma} 
   \langle \nabla_x,\nabla_v \rangle^{\sigma+b} g_\infty
   \right \Vert_{L_{x,v}^2}  < \varepsilon.
\end{equation}
Then, for sufficiently small $\varepsilon>0$, there exists a unique Gevrey solution $r$ of (\ref{nlpvpme}) such that the electric field $E[r]$ decays exponentially fast to zero as time goes to $+\infty$, in the sense that
\begin{equation*}
	\left\Vert  A_{\bar{\lambda}}(\nabla_x,t \nabla_x) 
    E_r(t)
   \right\Vert_{ L_x^2}^2 \le C e^{-C\jp{t}^\gamma}
\end{equation*}
where $C>0$ is a constant and $\bar{\lambda}< \lambda(0)$ in \eqref{def_lambda}.
Moreover, we have
\begin{equation*}
\label{asymptotic_cond_r}
	\lim_{t \rightarrow + \infty}  \|  r(t,x+vt,v) - g_\infty(x, v) \|_{L^{\infty} (\bt^d \times \br^d)} = 0.
\end{equation*} 
\end{thm}

\textbf{Remarks and novelty}
We present here some remarks and discuss the novelties of our approach.
\begin{rem}
    The result in Theorem \ref{Main_thm} gives affirmative answer to the problem posed by J. Bedrossian in the recent review \cite[Section 6]{JB_notes} about the existence and injectivity
    of the wave operator map $\mathcal{W}(g_\infty)=g_0$, which associates the initial data to the asymptotic data.
\end{rem}

\begin{rem}
Combining Theorem \ref{Main_thm} with the Landau damping for the forward problem proved in \cite[Theorem 1.1]{GI_23}, guarantees the construction of the scattering map $\mathcal{S}(g_{-\infty})=g_{+\infty}$ that maps the state at $-\infty$ to the asymptotic state at $+\infty$.
\end{rem}
Let us introduce the new unknown $g(t,x,v):= r(t,x+vt,v)$. Since
\begin{equation*}
    \partial_t g = \partial_t r + v \cdot \nabla_x r, \quad \nabla_x g = \nabla_x r.
\end{equation*}
it follows that the profile $g_t(x,v)$ verifies the equation
\begin{equation}
\label{nlequation_g}
\begin{cases}
    \partial_t g(t,x,v) + E(t,x+vt) \cdot \nabla_v \mu(v) = -  E(t,x+vt) \cdot(\nabla_v - t \nabla_x)g(t,x,v),\\
    E(t,x) = - \nabla U(t,x), \quad \quad - \Delta U(t,x) + \beta U(t,x) + h( U(t,x)) = \int_{\mathbb{R}^d} g(t,x-vt,v) dv.
    \end{cases}
\end{equation}
Moreover, we can rewrite (\ref{freescatt}) as:
\begin{equation}
\label{asymptotic_cond}
	\| g(t,x,v) - g_\infty(x, v) \|_{L^{\infty} (\bt^d \times \br^d)} \longrightarrow 0 \qquad \mbox{as } t\rightarrow + \infty,
\end{equation}
where the asymptotic datum $g_\infty$ will be assumed to be a mean-zero function with Gevrey regularity as in \ref{norm_g_infty}.

Observe that the system \eqref{nlequation_g} is composed of two nonlinear equations, one nonlinear transport-type PDE for $g$ and a nonlinear elliptic equation for $U$. 
We treat the nonlinearity for the Poisson equation by introducing the modified density
\begin{equation}
\label{sec1:screen}
\varrho_t(x):=(\beta- \Delta)U_t(x).
\end{equation}
It holds that $\varrho_t(x)$ satisfies a closed equation, see \eqref{U_hat_1}. We use this fact to study the coupled system given by $(g, \varrho)$, deriving global in time a priori estimates for it.
We will see that assumptions (H1) and (H2) on the homogeneous equilibrium $\mu$ are necessary to invert the linear term in equation \eqref{U_hat_1}.
Moreover, the convolution operator for the linearized term is different from the one in the Cauchy problem and requires the introduction of the \emph{two-sided Laplace transform} instead of the one-sided one.
Regarding the a priori estimates for the nonlinearities, we study the system $(g, \varrho)$ by following
the approach introduced in the recent review by J. Bedrossian in \cite{JB_notes}, recovering
, also in this scattering problem, the Gevrey threshold for $\gamma>\frac13$. 
Technically, the analysis of the final data problem requires solving the equation \eqref{nlequation_g} backwards in time,
from infinity to time zero.
This involves
the use of a regularity parameter in \eqref{def_lambda}, which is increasing in time instead of a decreasing one as in the forward problem.
Moreover, dealing with general Vlasov-type equations, we also need to study further nonlinearity in the elliptic equation, which is treated by means of elliptic PDE techniques. In Proposition \ref{sec4:PoissonGevrey}, we also prove a result of independent interest on the existence of solutions for the elliptic equation in \eqref{nlequation_g} with small source terms in Gevrey spaces.

The structure of the paper is the following:
 in Section \ref{section  Landau damping} we focus on the equation verified by the modified density $\varrho_t(x)$ defined in \eqref{sec1:screen}. After writing down the Volterra equation verified by $\varrho$, in Subsection \ref{sec2:Inversibility of the linear term}  we invert the linear term by using the Penrose stability condition \eqref{Penrose} and we study the decay property of the related resolvent kernel. This allows us to give a priori estimates on $\varrho_t(x)$ by focusing only on the nonlinear term. We do this in Subsection \ref{subsec_nonlin}.
In Section \ref{section_distrg}
we perform the a priori estimates on the distribution function $g$.
Finally, in Section \ref{iterative scheme}, we state a general well-posedness result with Gevrey data for the nonlinear Poisson equation in \eqref{nlequation_g} and construct solutions to the final data problem.

\section{A priori estimates on the modified density \texorpdfstring{$\varrho$}{Lg}}
\label{section  Landau damping}

As mentioned in the Introduction, we consider the system $(g, \varrho)$ where $\varrho$ is the modified density defined as follows $\varrho_t(x):=-\Delta_x U_t(x) + \beta U_t(x)$
so that 
$\int_{\br^d} g(t,x-vt,v)dv=\varrho_t(x) + h(U_t(x))$. 

In this Section, we obtain $L^2$ in time estimates in Gevrey regularity
for the modified density $\varrho$. 
We first write down the equation verified by the Fourier coefficients of $\varrho_t(x)$ to carry out this analysis.

\begin{lem}
Let $\mu$ be an analytic homogeneous equilibrium. 
Let $g$ be the unique solution to the problem (\ref{nlequation_g}) with asymptotic condition (\ref{asymptotic_cond}). Then we have the following equation for $\widehat{\varrho}_t (k)$:
\begin{equation}
\label{U_hat_1}
    \widehat{\varrho_t}(k) 
	  + \frac{\vert k \vert^2}{\beta + \vert k \vert^2} \int_t^{+\infty} \widehat{\varrho_s}(k) (s-t) \widehat{ \mu } (-k (s-t)) \ ds =  \widehat{S}_t(k),\quad k \in \mathbb{Z}^d
\end{equation}
where the source term $\widehat{S}_t(k)$ is given by
\begin{align}
\label{S_hat}
	\widehat{S}_t(k) &:= \widehat{g_\infty}(k,kt) 
    -   \widehat{h(U_t)}(k)
    -   \sum_{\ell \in \mathbb{Z}_*^d}
    \int_t^{+\infty}  (s-t) \frac{k \cdot \ell}{\beta + \vert \ell \vert^2} \widehat{\varrho}_s(\ell)  \widehat{g}_s(k - \ell ,kt  - \ell s) \ ds.
\end{align}

\end{lem}

\begin{proof}

We start by taking the Fourier transform with respect to 
$k \in \mathbb{Z}^d$ and $\eta \in \mathbb{R}^d$
of (\ref{nlequation_g}). 
So using the properties of the Fourier transform we have,
\begin{align*}
	\partial_t \widehat{g}_t(k,\eta) & =
    - \widehat{E}_t (k)\cdot \widehat{\nabla_v \mu} (\eta - kt) 
    - i  \sum_{\ell \in \mathbb{Z}_*^d} (\eta -kt) \cdot \widehat{E}_{t} (\ell) \widehat{g}_t(k - \ell ,\eta  - \ell t) \\
	&= -i (\eta - kt)\cdot \widehat{E}_t (k) \widehat{\mu} (\eta - kt) 
    - i  \sum_{\ell \in \mathbb{Z}_*^d} (\eta -kt) \cdot\widehat{E}_{t} (\ell) \widehat{g}_t(k - \ell ,\eta  - \ell t).
\end{align*}
Since $\widehat{E}_t (k) = -ik \widehat{U}_t (k)$ we get 
\begin{equation}
\label{FT_g}
	\partial_t \widehat{g}_t(k ,\eta) 
  = -[(\eta - kt)\cdot k]  \widehat{U}_t (k) \widehat{\mu} (\eta -kt)
  -  \sum_{\ell \in \mathbb{Z}_*^d} [(\eta -kt)\cdot \ell] \widehat{U}_t(\ell) \widehat{g}_t(k - \ell ,\eta  - \ell t).
\end{equation}
Integrating (\ref{FT_g}) in time, by the asymptotic condition \eqref{asymptotic_cond} we obtain
\begin{equation*}
\label{eqsec2:g}
	\widehat{g}_t(k ,\eta)  
    = \widehat{g}_\infty(k ,\eta) 
    + \int_t^{+\infty}  [(\eta - ks)\cdot k] \widehat{U}_s (k) \widehat{\mu} (\eta - ks) \ ds  
    +   \sum_{\ell \in \mathbb{Z}_*^d}  \int_t^{+\infty}   [(\eta -ks)\cdot \ell] \widehat{U}_s(\ell) \widehat{g}_s(k - \ell ,\eta  -\ell s) \ ds.
\end{equation*}
We now consider the Fourier modes of $g$ with $\eta = kt$.
\begin{align}
\label{rho_hat_1}
\widehat{g}_t(k,kt)=  \widehat{g_\infty}(k,kt)  
- \int_t^{+\infty}  (s-t) \vert k \vert^2  \widehat{U}_s (k) \widehat{\mu} (-k (s-t)) \ ds  
- \sum_{\ell \in \mathbb{Z}_*^d}\int_t^{+\infty} (s-t)k\cdot \ell \widehat{U}_s(\ell) \widehat{g}_s(k - \ell ,\eta  -\ell s) \ ds.
\end{align}
Taking the Fourier transform of the Poisson coupling in (\ref{nlequation_g}), and noting that
\begin{align*}
    \reallywidehat{\int_{\br^d} g(t,x-vt,v)dv}(k) = \widehat{g}_t(k,kt),
\end{align*}
we obtain
\begin{equation}
\label{rho_hat_2}
	 \widehat{g}_t(k,kt) = - \widehat{\Delta_x U_t}(k) +\beta  \widehat{U_t}(k) 
      +  \widehat{h( U)}(k)  
      = \widehat{\varrho}_t(k) + \widehat{h( U)}(k).
\end{equation}
Then, plugging equations (\ref{rho_hat_2}) into (\ref{rho_hat_1}),
we obtain the closed equation in \eqref{U_hat_1}.
\end{proof}
Note that, for any $k \in \mathbb{Z}^d$ and considering a given source term $\widehat{S}_t(k)$, equation \eqref{U_hat_1} is an integral Volterra-type equation for $\widehat{\varrho}_t(k)$. In the following subsection, we focus on how to invert this linear term.
\subsection{Invertibility of the linear term}
\label{sec2:Inversibility of the linear term}

We start by introducing the one-sided and two-sided Laplace transforms.
Let $\phi: \mathbb{R} \to \mathbb{R}$ be of exponential type with parameter $c>0$, \textit{i.e.}, there exists $c>0$ such that $|\phi(t)|\le C e^{-c |t|}$, for a constant $C>0$. Then, the one-sided Laplace transform of $\phi$ is given by
\begin{equation*}
\mathcal{L}[\phi_t](\tau):=\int_0^{+\infty} e^{-\tau t} \phi(t) \, dt, \quad \tau \in \mathbb{C},
\end{equation*}
while the two-sided Laplace transform by
\begin{equation*}
    \mathcal{B}[\phi_t](\tau):=\int_{-\infty}^{+\infty} e^{-\tau t}\phi(t) \, dt,
\end{equation*}
both well-defined for $|\Re{\tau}|<c$.

To invert the linear term of \eqref{U_hat_1}, we state the following Lemma.
\begin{lem}
Let $\psi$ and $\phi$ be two real-valued functions of exponential type with parameter $c>0$. Then for $|\Re{\tau}|<c$, we have
\begin{equation}
\label{Property-Laplace}
	\mathcal{B}\left[\int_{t}^{+\infty} \psi (s) \phi (s-t) \ ds \right](\tau) = \mathcal{B}[\psi_t](\tau) \mathcal{L}[\phi_t]( - \tau).
\end{equation}
\end{lem}
\begin{proof}
We have
\begin{align*}
	\mathcal{B}\left[\int_{t}^{+\infty} \psi (s) \phi (s-t) \ ds \right](\tau) 
  & =  \int_{-\infty}^{+\infty} \int_{t}^{+\infty} \psi (s) \phi (s-t)    e^{-\tau t} \ ds  \  dt \\
  & =  \int_{-\infty}^{+\infty} \int_{-\infty}^{s} \psi (s) \phi (s-t)    e^{-\tau s} e^{\tau (s-t)} \ dt  \  ds \\
  & =   \int_{-\infty}^{+\infty}  \psi (s)  e^{-\tau s}  \  ds  \int_0^{+\infty} \phi (y)    e^{\tau y} \ dy  \\ 
	& = \mathcal{B}[\psi_t](\tau) \mathcal{L}[\phi_t]( - \tau),
\end{align*}
where we used Fubini’s Theorem and the change of variable $y=s-t$.
\end{proof}

\subsubsection{Resolvent estimates}
\label{sec2:Resolvent estimates}

Using the one-sided and two-sided Laplace transforms, we will convert the integral Voltera-type equation \eqref{U_hat_1} for $\widehat{\varrho}_t(k)$ into an algebraic equation.
By taking the two-sided Laplace transform of (\ref{U_hat_1}), thanks to \eqref{Property-Laplace}, we get
\begin{equation*}
	\mathcal{B}[\widehat{\varrho}_t(k)](\tau) + \frac{\vert k \vert^2}{\beta +\vert k \vert^2}  \mathcal{L}[t\widehat{\mu}(-kt) ](-\tau) \mathcal{B}[\widehat{\varrho}_t(k)](\tau) = \mathcal{B}[\widehat{S}_t(k)](\tau),
\end{equation*}
which gives the solution
\begin{equation}
\label{LU1}
	\mathcal{B}[\widehat{\varrho}_t(k)](\tau)  
   =  \frac{\mathcal{B}[\widehat{S}_t(k)](\tau)}{1 + 
   \frac{\vert k \vert^2}{\beta +\vert k \vert^2} 
   \mathcal{L}[t\widehat{\mu}(-kt) ](-\tau) }.
\end{equation}
By the Penrose stability condition (\ref{Penrose}),  the denominator of (\ref{LU1}) never vanishes.  
Indeed, note that the requirement
\begin{equation*}
\inf_{k \in \mathbb{Z}_*^d;
  \Re \tau \leq 0 }{\left| 1 + \frac{\vert k \vert ^2}{\beta + \vert k \vert^2} \int_0^{+\infty}  t\widehat{\mu}(-kt) e^{\tau t} \ dt \right|} \geq \kappa_0 > 0
\end{equation*}
needed in \eqref{LU1} is implied by \eqref{Penrose}.
Next, we define the  kernel $\widetilde{K}_k(\tau)$, as
\begin{equation}
\label{resolvent}
	\widetilde{K}_k(\tau) := - 
  \frac{ \frac{\vert k \vert^2}{\beta +\vert k \vert^2}  \mathcal{L}[t\widehat{\mu}(-kt) ](\tau)}{1+\frac{\vert k \vert^2}{\beta +\vert k \vert^2}  \mathcal{L}[t\widehat{\mu}(-kt) ](\tau)}.
\end{equation}
Therefore, (\ref{LU1}) becomes
\begin{equation}
\label{LU2}
	\mathcal{B}[\widehat{\varrho}_t(k)](\tau)  = \mathcal{B}[\widehat{S}_t(k)](\tau) + \widetilde{K}_k(- \tau) \mathcal{B}[\widehat{S}_t(k)](\tau).
\end{equation}
Thanks to \cite[Lemma 3.2] {GNR}, we have the following property of the resolvent kernel.
\begin{lem}
\label{lemmaresolvent}
Let $\mu$ be an analytic homogeneous equilibrium satisfying assumptions (H1)-(H3) with $\lambda>0$ as in \eqref{norm_mu}. Then there is a positive constant $\lambda_1 < \frac{1}{2}\lambda$ so that the function $\widetilde{K}_k(\tau)$ is an analytic function in $ \lbrace \Re \tau \geq -  \lambda_1 \vert k \vert \rbrace $. Moreover, there exists a constant $C$ such that
\begin{equation}
\label{ine_resolvent}
	\vert \widetilde{K}_k(\tau) \vert \leq \frac{C}{1 + \vert k \vert^2 +\vert \Im \tau \vert^2},
\end{equation}
for $k \neq 0$ and $\Re \tau = -  \lambda_1 \vert k \vert$.
\end{lem}
    Note, by sending $k\mapsto-k$ in \eqref{resolvent}, our resolvent kernel $\widetilde{K}$ is equal, up to the constant $|k|^2(\beta+|k|^2)^{-1}$ to the one in \cite[equation (3.6)]{GNR}. Since the estimates on this kernel depend only on the modulus of $k$, the proof of Lemma \ref{lemmaresolvent} is analogous to the one in \cite[Lemma 3.2] {GNR} and is therefore omitted.

\subsubsection{Inverse of the resolvent kernel}
\label{Pointwise estimates}
First, we recall the definition of the inverse Laplace transform for the function $\widetilde{K}_k(\tau)$ defined in the region $ \lbrace \Re \tau \geq -  \lambda_1 \vert k \vert \rbrace $,
\begin{equation}
\label{K_hat}
	\widehat{K}_k(t)
    \equiv \mathcal{L}^{-1} [\widetilde{K}_k(\tau)] (t) 
    := \frac{1}{2 \pi i} \int_{\lbrace \Re \tau =  \gamma_0  \rbrace} \widetilde{K}_k(\tau) e^{\tau t} \ d\tau,
\end{equation}
for some constant $\gamma_0 > 0$.
Then, we have the following property:
\begin{lem}\cite[Proposition 3.3]{GNR}
	The function $\widehat{K}_k(t)$ satisfies 
\begin{equation}
\label{inverse_resolvent}
	\vert \widehat{K}_k (t) \vert \leq C e^{- \lambda_1 \vert k \vert t },
\end{equation}
where the constant $\lambda_1$ comes from Lemma \ref{lemmaresolvent}.
\end{lem}

\begin{proof}
By Lemma \ref{lemmaresolvent},  we know that $\widetilde{K}_k (\tau)$ is analytic in the region $ \lbrace \Re \tau \geq -  \lambda_1 \vert k \vert \rbrace $, therefore we have by the Cauchy theory on complex contour integral,
\begin{equation*}
    \widehat{K}_k(t) =
	 \frac{1}{2 \pi i} \int_{\lbrace \Re \tau=  \gamma_0  \rbrace} \widetilde{K}_k (\tau) e^{\tau t} \ d\tau = \frac{1}{2 \pi i} \int_{\lbrace \Re \tau= -  \lambda_1 \vert k \vert \rbrace} \widetilde{K}_k (\tau) e^{\tau t} \ d\tau.
\end{equation*}
That is, we can deform the complex contour of integration from $\lbrace \Re \tau=  \gamma_0  \rbrace$ into $\lbrace \Re \tau= -  \lambda_1 \vert k \vert \rbrace$.
Then using \eqref{ine_resolvent} we get,
\begin{align*}
	\vert \widehat{K}_k (t) \vert & \leq  \frac{1}{2 \pi } \int_{\lbrace \Re \tau= - \lambda_1 \vert k \vert \rbrace} \vert \widetilde{K}_k (\tau) e^{ \tau t} \vert \ d\tau\\
	& \leq  \frac{1}{2 \pi } \int_{\lbrace \Re \tau= - \lambda_1 \vert k \vert \rbrace} \frac{C}{1 + \vert k \vert^2 +\vert \Im \tau \vert^2} e^{- \lambda_1 \vert k \vert t}  \ d\tau\\
	& \leq  \frac{C}{2 \pi } e^{- \lambda_1 \vert k \vert t} \int_{-\infty}^{\infty} \frac{1}{1 + \vert k \vert^2 +\vert y\vert^2}  \ dy \leq  C e^{- \lambda_1 \vert k \vert t}.
\end{align*}

\end{proof}

Next, taking the inverse two-sided Laplace transform of \eqref{LU2}, we have the following Proposition.

\begin{prop} 
Let $\mu$ be an analytic homogeneous equilibrium satisfying assumptions (H1)-(H3). 
Let $\widehat{\varrho}_t (k)$ be a solution of \eqref{U_hat_1} and let $\widehat{K}_k(t)$ be given by \eqref{K_hat}. Then we can express $\widehat{\varrho}_t(k)$ as 
\begin{equation}
\label{U_hat_2}
    \widehat{\varrho}_t (k) = \widehat{S}_t(k) + \int_{t}^{+\infty} \widehat{K}_k(s-t) \widehat{S}_s(k) \ ds,
\end{equation}
where $\widehat{S}_t(k)$ is given by \eqref{S_hat} and 
$\widehat{K}_t (k)$ satisfies \eqref{inverse_resolvent}.
\end{prop}

\begin{proof}
With \eqref{K_hat}, we can rewrite \eqref{LU2} as
\begin{align*}
	\mathcal{B}[\widehat{\varrho}_t(k)](\tau) 
   = \mathcal{B}[\widehat{S}_t(k)](\tau)
   +  \widetilde{K}_k( - \tau) \mathcal{B}[\widehat{S}_t(k)](\tau) = \mathcal{B}[\widehat{S}_t(k)](\tau)
   +  \mathcal{L}[\widehat{K}_k(t)](- \tau) \mathcal{B}[\widehat{S}_t(k)](\tau).
\end{align*}
Then, we take the inverse two-sided Laplace transform, and we recall the convolution property  \eqref{Property-Laplace}.
\end{proof}

\subsection{Estimates on the modified density \texorpdfstring{$\varrho$}{Lg}}
\label{subsec_nonlin}
We recall the notation of the weight:
$   A_t(k,\eta):= e^{\lambda (t) \langle k, \eta \rangle^\gamma} \langle k, \eta \rangle^\sigma.$

\begin{lem}
\label{lem_varrho}
 Assume $\widehat{\varrho}_t (k)$ satisfies (\ref{U_hat_2}). Then, we have
    \begin{align*}
   \left\Vert \langle 
   t \rangle^b A_t(\nabla_x,t \nabla_x) 
   \varrho_t
   \right\Vert_{L_t^2 L_x^2}  
   \le
   C 
   \left \Vert \jp{t}^b A_t(\nabla_x,t \nabla_x) S_t
   \right \Vert_{L_t^2 L_x^2},  
\end{align*}
for some constant $C>0$ and where $S_t$ is given in \eqref{S_hat}.
\end{lem}

\begin{proof}
Starting from \eqref{U_hat_2}, we multiply by
$\jp{t}^b A_t(k,kt)$ on both side and take the $L_t^2 L_k^2$ norm. Therefore,
we get
\begin{equation*}
     \left\Vert \langle 
   t \rangle^b A_t(k,kt) 
   \widehat{\varrho}_t(k)
   \right\Vert_{L_t^2 L_k^2}  
   \le  \left\Vert \langle 
   t \rangle^b A_t(k,kt) 
   S_t(k)
   \right\Vert_{L_t^2 L_k^2}  
   +
    \left\Vert \langle 
   t \rangle^b A_t(k,kt) 
   \int_{t}^{+\infty} \widehat{K}_k(s-t) \widehat{S}_s(k) \ ds
   \right\Vert_{L_t^2 L_k^2}.
\end{equation*}
So, we must control only the last norm on the right-hand side.
By using $\lambda(t)\le \lambda(s)$ and $ A_t(k,kt) \le  A_s(k,ks)$  if $t \le s$, we obtain
$$
\left\Vert \jp{t}^b A_t(k,kt) 
   \int_{t}^{+\infty} \widehat{K}_k(s-t) \widehat{S}_s(k) \ ds
   \right\Vert_{L_t^2 L_k^2} \le 
   \left\Vert  
   \int_{t}^{+\infty} 
   \widehat{K}_k(s-t) \jp{s}^b A_s(k,ks)
   \widehat{S}_s(k) \ ds
   \right\Vert_{L_t^2 L_k^2}.
$$
Let us define, for a fixed $k \in \mathbb{Z}^d \setminus \{ 0 \}$, the operator
\begin{align*}
    \mathcal{P}_k (\phi) = \int_{t}^{+\infty} 
   \widehat{K}_k(s-t) \phi_k(s) \ ds,
\end{align*}
where $\widehat{K}_k(t)$ is given by \eqref{K_hat}. We show that $\norm{\mathcal{P}_k (\phi)}_{L_t^2 (\mathbb{R}^+) \rightarrow L_t^2 (\mathbb{R}^+)} \leq C$ by Riesz-Thorin interpolation Theorem. 
For the $L_t^\infty (\mathbb{R}^+)$ norm, we have
\begin{align*}
    \norm{\mathcal{P}_k (\phi)}_{L_t^\infty (\mathbb{R}^+)} \leq \frac{C}{\lambda_1 \vert k \vert} \norm{\phi}_{L_t^\infty (\mathbb{R}^+)},
\end{align*}
where we used the decay \eqref{inverse_resolvent} to bound the integral in $s$.
For the $L_t^1 (\mathbb{R}^+)$ norm, we have by Fubini’s Theorem
\begin{align*}
    \norm{\mathcal{P}_k (\phi)}_{L_t^1 (\mathbb{R}^+)} \leq \frac{C}{\lambda_1 \vert k \vert} \norm{\phi}_{L_t^1 (\mathbb{R}^+)}.
\end{align*}
Therefore, by interpolation, we get that $\mathcal{P}_k$ is bounded in $L_t^2$. Hence, by taking the $L_k^2$ and using Fubini’s Theorem, we obtain
$$
\left\Vert \langle 
   t \rangle^b A_t(k,kt) 
   \int_{t}^{+\infty} \widehat{K}_k(s-t) \widehat{S}_s(k) \ ds
   \right\Vert_{L_t^2 L_k^2} \le 
   \frac{C}{\lambda_1 \vert k \vert}
\left \Vert \jp{t}^b A_t(k,kt) \widehat{S}_t (k)
   \right \Vert_{L_t^2 L_k^2}.
$$
Finally, coming back to the physical side, we get the result.
\end{proof}
From Lemma \ref{lem_varrho}, we see that to obtain an estimate on $\varrho$, we only need to estimate the source term $S$. We carry this analysis in the following Proposition.
\begin{prop}
\label{sec2:prop_estimate_S}
Let $S_t$ be the inverse Fourier transform of \eqref{S_hat}. Let $g_\infty$ be a Gevrey function of mean zero satisfying \eqref{norm_g_infty}. Then for $\sigma > 10+d$, $M > \frac{d}{2}$, $b>10$,  we have
    \begin{align*}
        \left \Vert \langle
        t \rangle^b A_t(\nabla_x,t \nabla_x) S_t
        \right \Vert_{L_t^2 L_x^2} 
   & \le C
   \left \Vert
   \langle v \rangle^m e^{\lambda \langle \nabla_x, \nabla_v  \rangle^\gamma} 
   \langle \nabla_x,\nabla_v \rangle^{\sigma+b} g_\infty
   \right \Vert_{L_{x,v}^2} + 
   \left \Vert \jp{t}^b A_t(\nabla_x,t \nabla_x)  h( U)
   \right \Vert_{L_t^2 L_x^2} \\
   & \quad +C \left \Vert \jp{t}^b A_t(\nabla_x,t \nabla_x) \varrho_t
   \right \Vert_{L_t^2 L_x^2} \sup_{t \in \br_+} \left \Vert \jp{v}^M  B_t(\nabla_x,\nabla_v) g_t
   \right \Vert_{L_{x,v}^2},
    \end{align*}
for some constant $C>0$.
\end{prop}
 
To prove this result, we need the following Lemma \cite[Lemma 3.2]{BMM_13} regarding some Gevrey inequalities we use several times. 
\begin{lem}
Let $x,y \geq 0$ and $\gamma \in (0,1)$.
\begin{itemize}
    \item[i)] Then we have the following triangle inequality,
    \begin{align}
        \label{triangle_inequality_Gevrey}
        C \left( \jp{x}^\gamma + \jp{y}^\gamma \right) \leq \jp{x+y}^\gamma,
    \end{align}
    for some $C=C(\gamma) >0$ depending only on $\gamma$.
    \item[ii)] Then there exists a constant $C=C(\gamma)$ such that
    \begin{align}
    \label{gevrey_inequality_2}
        \abs{\jp{x}^\gamma - \jp{y}^\gamma} \leq C \frac{\jp{x-y}}{\jp{x}^{1-\gamma} + \jp{y}^{1-\gamma}}.
    \end{align}
    \item[iii)]If $\abs{x-y} \leq \frac{x}{K}$ for some $K >1$, then we have the following inequality:
    \begin{equation}
    \label{gevrey_inequality}
        \abs{\jp{x}^\gamma - \jp{y}^\gamma} \le \frac{\gamma}{(K-1)^{1-\gamma}} \jp{x -y}^\gamma.
    \end{equation}
    \item[iv)] Then there exists a constant $c=c(\gamma) \in (0,1) $ such that
    \begin{align}
    \label{gevrey_inequality_4}
        \jp{x + y}^\gamma \leq c \left( \jp{x}^\gamma + \jp{y}^\gamma \right).
    \end{align}
\end{itemize}
\end{lem}
In particular, the improvements to the triangle inequality for $\gamma < 1$ given in \eqref{gevrey_inequality_2}, \eqref{gevrey_inequality}, and \eqref{gevrey_inequality_4} are important for getting useful estimates. Indeed, note that in \eqref{gevrey_inequality}, we need $\gamma < 1$ to get $c<1$. This implies a gain of decay in the estimates that is absent if $\gamma=1$.
\begin{proof}[Proof of Proposition \ref{sec2:prop_estimate_S}]
    Recall that 
    \begin{align*}
        \widehat{S}_t(k) &:= \widehat{g_\infty}(k,kt) 
    - \widehat{h( U_t)}(k)
    + \sum_{\ell \in \mathbb{Z}_*^d}
    \int_t^{+\infty}  (s-t) \frac{k \cdot \ell}{\beta + \vert \ell \vert^2} \widehat{\varrho}_s(\ell)  \widehat{g}_s(k - \ell ,kt  - \ell s) \ ds.
    \end{align*}
By triangle inequality, we treat each term separately. 

\textbf{Estimate on $g_\infty$: linear term.} First, we show the bound on the asymptotic datum $g_\infty$. 
Using that $\int g_\infty(x,v) dx dv=0$, we get
\begin{align}
\label{sec2:estimate_ginf}
    \norm{ \langle
    t \rangle^b A_t(\nabla_x,t \nabla_x) g_{\infty}
    }_{L_t^2 L_x^2}  
    & = \int_0^{+\infty} \sum_{k\in \mathbb{Z}^d} \jp{t}^{2b} A^2_t(k,kt) |\widehat{g_\infty}(k,kt)|^2 dt \nonumber\\
    & \le
    \sum_{k\in \mathbb{Z}_*^d}  \int_0^{+\infty} \langle k,kt \rangle^{2b} A^2_t(k,kt) |\widehat{g_\infty}(k,kt)|^2 dt \nonumber\\
    &\le 
    \sum_{k\in \mathbb{Z}_*^d} \int_0^{+\infty} \left\langle k,t\frac{k}{|k|} \right \rangle^{2b}
    e^{2\lambda\left\langle k,\frac{k}{|k|}t\right\rangle^\gamma }
    \left\langle k,\frac{k}{|k|}t\right\rangle^{2\sigma}
    \left|\widehat{g_\infty}\left(k,\frac{k}{|k|}t\right)\right|^2 dt \nonumber\\
    & \le \sum_{k \in \mathbb{Z}^d_*} \left \| e^{\lambda \langle k,\eta\rangle^\gamma}\langle k, \eta \rangle^{\sigma+b} \widehat{g_\infty}(k, \eta) \right \|^2_{L^2_\eta(\mathbb{R} \frac{k}{|k|})}
    \nonumber \\
    &\le C \sum_{k \in \mathbb{Z}_*^d}
    \sum_{0 \le j \le m}
    \left \| D^j_\eta e^{\lambda \langle k,\eta\rangle^\gamma}\langle k, \eta \rangle^{\sigma+b} \widehat{g_\infty}(k, \eta) \right \|^2_{L^2_\eta(\mathbb{R}^d)}
    \nonumber\\
    &\le C \|\langle v \rangle^m e^{\lambda \langle \nabla_x,\nabla_v \rangle^\gamma}
    \langle \nabla_x, \nabla_v \rangle^{\sigma+b} g_\infty \|^2_{L^2_{x,v}},
\end{align}
where in the fourth inequality, we used the Sobolev trace lemma to bound the $L^2$ integral on the line.

\textbf{Estimate on $h( U_t)$.}
The $L_t^2 L_x^2$ norm of the second term is exactly 
\begin{align}
\label{sec2:estimate_h(U)}
    \left \Vert \jp{t}^b A_t(\nabla_x,t \nabla_x)
    h( U)
   \right \Vert_{L_t^2 L_x^2},
\end{align}
therefore, the challenging part is to treat the nonlinear contribution
\begin{align*}
    T_{NL} := \sum_{\ell \in \mathbb{Z}_*^d}
    \int_t^{+\infty}  (s-t) \frac{k \cdot \ell}{\beta + \vert \ell \vert^2} \widehat{\varrho}_s(\ell)  \widehat{g}_s(k - \ell ,kt  - \ell s) \ ds.
\end{align*}

\textbf{Estimate on $T_{NL}$: nonlinear term.}
The idea is to subdivide this term into different regions of frequency. We use the following splitting: 
\begin{align}
\label{slpitting}
    \mathbbm{1} & = 
    \mathbbm{1}_{\abs{\ell, \ell s} \le \frac{1}{2} \abs{k-\ell,kt-\ell s}} 
    + \mathbbm{1}_{\abs{k-\ell,kt-\ell s} \le \frac{1}{2} \abs{\ell, \ell s}}
    + \mathbbm{1}_{\abs{\ell, \ell s} \approx \abs{k-\ell,kt-\ell s}}.
\end{align}
Therefore, we can rewrite $T_{NL}$ as
\begin{align} 
\label{sec2:splitting_T_NL}
    \jp{t}^b A_t(k, kt) T_{NL} 
    & =
    \jp{t}^b A_t(k, kt) \sum_{\ell \in \bz_*^d} \int_t^{+\infty} \mathbbm{1}_{\abs{\ell, \ell s} \le \frac{1}{2} \abs{k-\ell,kt-\ell s}} \frac{\ell \cdot k}{\beta + \vert\ell\vert^2}  (s-t) \widehat{\varrho}_s (\ell) \widehat{g}_s (k-\ell,kt- \ell s) \ d s \nonumber \\
    &\quad +\jp{t}^b A_t(k, kt) \sum_{\ell \in \bz_*^d} \int_t^{+\infty} \mathbbm{1}_{\abs{k-\ell,kt-\ell s} 
    \le \frac{1}{2} \abs{\ell, \ell s}}  \frac{\ell \cdot k}{\beta +\vert\ell\vert^2}  (s-t) \widehat{\varrho}_s (\ell) \widehat{g}_s (k-\ell,kt- \ell s) \ d s \nonumber \\
    &\quad + \jp{t}^b A_t(k, kt) \sum_{\ell \in \bz_*^d} \int_t^{+\infty} \mathbbm{1}_{\abs{\ell, \ell s} \approx \abs{k-\ell,kt-\ell s}} \frac{\ell \cdot k}{\beta +\vert\ell\vert^2}  (s-t) \widehat{\varrho}_s (\ell) \widehat{g}_s (k-\ell,kt- \ell s) \ d s \nonumber \\
    & =:T_{LH}+  T_{HL}  + T_{HH}.
\end{align}

\emph{Nonlinear estimate on $T_{LH}$:}
We first estimate $T_{LH}$. In this region $\abs{\ell, \ell s} \le \frac{1}{2} \abs{k-\ell,kt-\ell s}$, we have 
\begin{align*}
    \jp{k,kt} \leq \jp{\ell, \ell s} + \jp{k-\ell, kt - \ell s} \leq 2 \jp{k-\ell, kt - \ell s}.
\end{align*}
Therefore,
\begin{align}
\label{ineq_A_LH}
    A_t(k, kt) & \leq 2^\sigma e^{-(\lambda (s) - \lambda(t)) \jp{k, kt}^\gamma} e^{ \lambda (s) \jp{k, kt}^\gamma} \jp{k-\ell, kt - \ell s}^\sigma \nonumber \\
    & \leq 2^\sigma e^{-(\lambda (s) - \lambda(t)) \jp{k, kt}^\gamma} e^{c \lambda (s) \jp{\ell , \ell s}^\gamma} e^{\lambda(s) \jp{k-\ell, kt - \ell s}^\gamma} \jp{k-\ell, kt - \ell s}^\sigma,
\end{align}
where we used \eqref{gevrey_inequality} to get
\begin{align*}
    \abs{\jp{k, kt}^\gamma - \jp{k-\ell, kt - \ell s}^{\gamma}} \leq c \jp{\ell , \ell s}^\gamma,
\end{align*}
with $c < 1$.
Thus, using \eqref{ineq_A_LH}, we get
\begin{align*}
     \abs{T_{LH}} & \leq \jp{t}^b A_t(k, kt) \sum_{\ell \in \bz_*^d} \int_t^{+\infty} \mathbbm{1}_{\abs{\ell, \ell s} \le \frac{1}{2} \abs{k-\ell,kt-\ell s}}  \frac{\abs{\ell \cdot k}}{ \beta + \vert\ell\vert^2}  (s-t) \abs{\widehat{\varrho}_s (\ell)} \abs{\widehat{g}_s (k-\ell,kt- \ell s)} \ d s \\
    & \leq C \sum_{\ell \in \bz_*^d} \int_t^{+\infty} \mathbbm{1}_{\abs{\ell, \ell s} \le \frac{1}{2} \abs{k-\ell,kt-\ell s}} \jp{t}^b e^{-(\lambda (s) - \lambda(t)) \jp{k, kt}^\gamma} e^{c \lambda (s) \jp{\ell , \ell s}^\gamma}  \\
    & \qquad \times \frac{\abs{\ell \cdot k}}{ \beta + \vert\ell\vert^2}  (s-t) \abs{\widehat{\varrho}_s (\ell)} \abs{\widehat{Ag}_s (k-\ell,kt- \ell s)} \ d s.
\end{align*}
In this case, we can bound $e^{-(\lambda (s) - \lambda(t)) \jp{k, kt}^\gamma} \leq 1$, therefore
\begin{align*}
     \abs{T_{LH}} 
    & \leq C \sum_{\ell \in \bz_*^d} \int_t^{+\infty} \mathbbm{1}_{\abs{\ell, \ell s} \le \frac{1}{2} \abs{k-\ell,kt-\ell s}} \jp{t}^b e^{c \lambda (s) \jp{\ell , \ell s}^\gamma}  \\
    & \qquad \times \frac{\abs{\ell \cdot k}}{ \beta + \vert\ell\vert^2}  (s-t) \abs{\widehat{\varrho}_s (\ell)} \abs{\widehat{Ag}_s (k-\ell,kt- \ell s)} \ d s.
\end{align*}
Thus, adding the weight $A$ to $\varrho$ and using that
\begin{align}
\label{sec2:ineq_s-t}
    \abs{k} (s-t) = \abs{ks - \ell s + \ell s - kt} \leq  \jp{s} \abs{k - \ell} + \abs{\ell s - kt} \leq \jp{s} \jp{k-\ell, kt-\ell s},
\end{align}
we have
\begin{align*}
     \abs{T_{LH}} 
    & \leq C \sum_{\ell \in \bz_*^d} \int_t^{+\infty} \mathbbm{1}_{\abs{\ell, \ell s} \le \frac{1}{2} \abs{k-\ell,kt-\ell s}}  \frac{\jp{t}^b}{\jp{\ell, \ell s}^\sigma} e^{-(1-c) \lambda (s) \jp{\ell , \ell s}^\gamma}  \\
    & \qquad \times \frac{\jp{s}}{\vert\ell\vert}  \abs{\widehat{A\varrho}_s (\ell)} \jp{k-\ell, kt-\ell s} \abs{\widehat{Ag}_s (k-\ell,kt- \ell s)} \ d s \\
    & \leq C \sum_{\ell \in \bz_*^d} \int_t^{+\infty} \mathbbm{1}_{\abs{\ell, \ell s} \le \frac{1}{2} \abs{k-\ell,kt-\ell s}}  \frac{\jp{s}^{b+1}}{\jp{\ell, \ell s}^\sigma} e^{-(1-c) \lambda (s) \jp{\ell , \ell s}^\gamma}  \abs{\widehat{A\varrho}_s (\ell)} \abs{\widehat{Bg}_s (k-\ell,kt- \ell s)} \ d s.
\end{align*}
Therefore, taking the $L_t^2 L_x^2$ norm and using Cauchy-Schwarz in $\ell$ and $s$, we get
\begin{align*}
    & \norm{T_{LH}}_{L_t^2 L_x^2}^2 \\
    & \leq C \int_0^{+\infty} \sum_{k \in \bz^d} \left(  \sum_{\ell \in \bz_*^d} \int_t^{+\infty} \mathbbm{1}_{\abs{\ell, \ell s} \le \frac{1}{2} \abs{k-\ell,kt-\ell s}}  \frac{\jp{s}^{b+1}}{\jp{\ell, \ell s}^\sigma} e^{-(1-c) \lambda (s) \jp{\ell , \ell s}^\gamma}  \abs{\widehat{A\varrho}_s (\ell)} \abs{\widehat{Bg}_s (k-\ell,kt- \ell s)} \ d s \right)^2 \ dt \\
    & \leq C \int_0^{+\infty} \sum_{k \in \bz^d} \left(  \sum_{\ell \in \bz_*^d} \int_t^{+\infty} \mathbbm{1}_{\abs{\ell, \ell s} \le \frac{1}{2} \abs{k-\ell,kt-\ell s}}  \frac{\jp{s}^{b+1}}{\jp{\ell, \ell s}^\sigma} e^{-(1-c) \lambda (s) \jp{\ell , \ell s}^\gamma}  \abs{\widehat{A\varrho}_s (\ell)}  \ d s \right) \\
    & \qquad \times \left(\sum_{\ell \in \bz_*^d} \int_t^{+\infty} \mathbbm{1}_{\abs{\ell, \ell s} \le \frac{1}{2} \abs{k-\ell,kt-\ell s}}  \frac{\jp{s}^{b+1}}{\jp{\ell, \ell s}^\sigma} e^{-(1-c) \lambda (s) \jp{\ell , \ell s}^\gamma}  \abs{\widehat{A\varrho}_s (\ell)} \abs{\widehat{Bg}_s (k-\ell,kt- \ell s)}^2 \ d s \right) \ dt.
\end{align*}
We treat the first parenthesis using Cauchy-Schwarz again
\begin{align*}
    & \sum_{\ell \in \bz_*^d} \int_t^{+\infty} \mathbbm{1}_{\abs{\ell, \ell s} \le \frac{1}{2} \abs{k-\ell,kt-\ell s}}  \frac{\jp{s}^{b+1}}{\jp{\ell, \ell s}^\sigma} e^{-(1-c) \lambda (s) \jp{\ell , \ell s}^\gamma}  \abs{\widehat{A\varrho}_s (\ell)}  \ d s \\
    & \leq \left(\sum_{\ell \in \bz_*^d} \int_t^{+\infty} \jp{s}^{2b} \abs{\widehat{A\varrho}_s (\ell)}^2  \ d s \right)^\frac{1}{2}
    \left(  \sum_{\ell \in \bz_*^d} \int_t^{+\infty}   \frac{\jp{s}^{2}}{\jp{\ell, \ell s}^{2\sigma}} e^{-2(1-c) \lambda (s) \jp{\ell , \ell s}^\gamma}  \ d s  \right)^\frac{1}{2} \\
    & \leq \left( \int_0^{+\infty} \sum_{\ell \in \bz_*^d} \jp{s}^{2b}  \abs{\widehat{A\varrho}_s (\ell)}^2  \ d s \right)^\frac{1}{2}
    \left(  \sum_{\ell \in \bz_*^d} \frac{1}{\jp{\ell}^{2\sigma}} \int_t^{+\infty}  \jp{s}^{2}  e^{-2(1-c) \lambda(0) \jp{s}^\gamma}  \ d s  \right)^\frac{1}{2} \\
    & \leq C \norm{\jp{t}^{b} A_t (\nabla_x, t \nabla_x) \varrho_t}_{L_t^2 L_x^2},
\end{align*}
where we used that the second term is bounded by a constant if $\sigma>d$ and recall $\lambda(s)>\lambda(0)$.Thus, by Fubini’s Theorem 
\begin{align*}
     \norm{T_{LH}}_{L_t^2 L_x^2}^2
    & \leq C \norm{\jp{t}^{b} A_t (\nabla_x, t \nabla_x) \varrho_t}_{L_t^2 L_x^2} \\
    &\quad \times \int_0^{+\infty} \int_t^{+\infty} \sum_{\ell \in \bz_*^d} \sum_{k \in \bz^d}     \frac{\jp{s}^{b+1}}{\jp{\ell, \ell s}^\sigma} e^{-(1-c) \lambda (s) \jp{\ell , \ell s}^\gamma}  \abs{\widehat{A\varrho}_s (\ell)} \abs{\widehat{Bg}_s (k-\ell,kt- \ell s)}^2 \ ds dt \\
    & \leq C \norm{\jp{t}^{b} A_t (\nabla_x, t \nabla_x) \varrho_t}_{L_t^2 L_x^2} \\
    &\quad \times \int_0^{+\infty} \int_t^{\infty} \sum_{\ell \in \bz_*^d} \abs{\widehat{A\varrho}_s (\ell)}    \frac{\jp{s}^{b+1}}{\jp{\ell, \ell s}^\sigma} e^{-(1-c) \lambda (s) \jp{\ell , \ell s}^\gamma}  \sum_{k \in \bz^d} \sup_{\eta}\abs{\widehat{Bg}_s (k,\eta)}^2 \ ds dt \\
    & \leq C \norm{\jp{t}^{b} A_t (\nabla_x, t \nabla_x) \varrho_t}_{L_t^2 L_x^2} \sup_{t \in \br_+} \norm{\jp{v}^M B_t (\nabla_x, \nabla_v)g_t}_{L_{x,v}^2}^2 \\
    &\quad \times \int_0^{+\infty} \int_t^{\infty} \sum_{\ell \in \bz_*^d} \abs{\widehat{A\varrho}_s (\ell)}    \frac{\jp{s}^{b+1}}{\jp{\ell, \ell s}^\sigma} e^{-(1-c) \lambda (s) \jp{\ell , \ell s}^\gamma}  \ ds dt.
\end{align*}
The bound, in the penultimate inequality on $\widehat{Bg}$, is obtained using the $L^\infty$ Sobolev embedding with $M>d$,
\begin{align}
\label{Bg_Sobolev_embedding}
    \sum_{k \in \bz^d} \sup_{\eta \in \br^d} \abs{\widehat{Bg}_s(k, \eta)}^2 
    &  \leq \sum_{k \in \bz^d}  \sum_{\vert j \vert \leq M} \left\Vert \partial_\eta^j \big(  \widehat{Bg}_s (k,\eta) \big)\right\Vert_{L_\eta^2}^2 
    = \left \Vert \jp{v}^M  B_s(\nabla_x,\nabla_v) g_s
   \right \Vert_{L_{x,v}^2}^2.
\end{align}
By using that $\lambda(s)>\lambda(0)$ and $s>t$, we get
\begin{align*}
     \norm{T_{LH}}_{L_t^2 L_x^2}^2 
    & \leq C \norm{\jp{t}^b A_t (\nabla_x, t \nabla_x) \varrho_t}_{L_t^2 L_x^2} \sup_{t \in \br_+} \norm{\jp{v}^M B_t (\nabla_x, \nabla_v)g_t}_{L_{x,v}^2}^2 \\
    &\quad \times \int_0^{+\infty} e^{-\frac{(1-c)}{2} \lambda(0) \jp{t}^\gamma} \int_t^{\infty} \sum_{\ell \in \bz_*^d} \abs{\widehat{A\varrho}_s (\ell)}    \frac{\jp{s}^{b+1}}{\jp{\ell, \ell s}^\sigma} e^{-\frac{(1-c)}{2} \lambda (s) \jp{\ell , \ell s}^\gamma}  \ ds dt.
\end{align*}
Then, by Cauchy-Schwarz in $s$ and $\ell$ we have
\begin{align}
\label{sec2:estimate_T_LH}
     \norm{T_{LH}}_{L_t^2 L_x^2}^2 
    & \leq C \norm{\jp{t}^b A_t (\nabla_x, t \nabla_x) \varrho_t}_{L_t^2 L_x^2} \sup_{t \in \br_+} \norm{\jp{v}^M B_t (\nabla_x, \nabla_v)g_t}_{L_{x,v}^2}^2 \nonumber \\
    &\quad \times \int_0^{+\infty} e^{-\frac{(1-c)}{2} \lambda(0) \jp{t}^\gamma} \left( \int_t^{\infty} \sum_{\ell \in \bz_*^d} \jp{s}^{2b} \abs{\widehat{A\varrho}_s (\ell)}^2 \ ds \right)^{\frac{1}{2}} \left(  \int_t^{\infty} \sum_{\ell \in \bz_*^d}   \frac{\jp{s}^{2}}{\jp{\ell, \ell s}^{2\sigma}} e^{-(1-c) \lambda (s) \jp{\ell , \ell s}^\gamma}  \ ds \right)^\frac{1}{2} dt \nonumber \\
    & \leq C \norm{\jp{t}^b A_t (\nabla_x, t \nabla_x) \varrho_t}_{L_t^2 L_x^2}^2 \sup_{t \in \br_+} \norm{\jp{v}^M B_t (\nabla_x, \nabla_v)g_t}_{L_{x,v}^2}^2,
\end{align}
where we used that the integral in $s$ and sum over $\ell$ is finite for $\sigma>d$, and the integral in time is bounded thanks to the exponential decay in $t$.

\emph{Nonlinear estimate on $T_{HL}$:} We now turn to the analysis of the term $T_{HL}$ in \eqref{sec2:splitting_T_NL}.
\begin{align*}
    T_{HL} := \jp{t}^b A_t(k, kt) \sum_{\ell \in \bz_*^d} \int_t^{+\infty} \mathbbm{1}_{\abs{k-\ell,kt-\ell s} 
    \le \frac{1}{2} \abs{\ell, \ell s}}  \frac{\ell \cdot k}{\beta + \vert\ell\vert^2}  (s-t) \widehat{\varrho}_s (\ell) \widehat{g}_s (k-\ell,kt- \ell s) \ d s.
\end{align*}
In this region $\abs{k-\ell,kt-\ell s} \le \frac{1}{2} \abs{\ell, \ell s}$, we have 
\begin{align*}
    \jp{k,kt} \leq \jp{\ell, \ell s} + \jp{k-\ell, kt - \ell s} \leq 2 \jp{\ell, \ell s}.
\end{align*}
Therefore,
\begin{align}
\label{ineq_A_HL}
    A_t(k, kt) & \leq 2^\sigma e^{-(\lambda (s) - \lambda(t)) \jp{k, kt}^\gamma} e^{ \lambda (s) \jp{k, kt}^\gamma} \jp{k-\ell, kt - \ell s}^\sigma \nonumber \\
    & \leq 2^\sigma e^{-(\lambda (s) - \lambda(t)) \jp{k, kt}^\gamma} e^{c \lambda (s) \jp{ k-\ell , kt-\ell s }^\gamma} e^{\lambda(s) \jp{\ell, \ell s}^\gamma} \jp{\ell,  \ell s}^\sigma,
\end{align}
where we used \eqref{gevrey_inequality} to get
\begin{align*}
    \abs{\jp{k, kt}^\gamma - \jp{\ell , \ell s}^\gamma} \leq c \jp{k-\ell, kt - \ell s}^{\gamma} ,
\end{align*}
with $c < 1$. Hence, by \eqref{ineq_A_HL}
\begin{align*}
    \abs{T_{HL}} & \leq  C \sum_{\ell \in \bz_*^d} \int_t^{+\infty} \mathbbm{1}_{\abs{k-\ell,kt-\ell s} 
    \le \frac{1}{2} \abs{\ell, \ell s}} \jp{t}^b e^{-(\lambda (s) - \lambda(t)) \jp{k, kt}^\gamma} e^{c \lambda (s) \jp{ k-\ell , kt-\ell s }^\gamma}  \\
    & \quad \times\frac{\abs{\ell \cdot k}}{ \beta + \vert\ell\vert^2}  (s-t) \abs{\widehat{A \varrho}_s (\ell)} \abs{\widehat{g}_s (k-\ell,kt- \ell s)} \ d s \\
    & \leq  C \sum_{\ell \in \bz_*^d} \int_t^{+\infty} \mathbbm{1}_{\abs{k-\ell,kt-\ell s} 
    \le \frac{1}{2} \abs{\ell, \ell s}} \jp{t}^b e^{-(\lambda (s) - \lambda(t)) \jp{k, kt}^\gamma} \frac{e^{-(1-c) \lambda (s) \jp{ k-\ell , kt-\ell s }^\gamma}}{\jp{k-\ell,kt-\ell s}^{\sigma+1}}  \\
    & \quad \times\frac{\abs{\ell \cdot k}}{ \beta + \vert\ell\vert^2}  (s-t) \abs{\widehat{A \varrho}_s (\ell)} \abs{\widehat{Bg}_s (k-\ell,kt- \ell s)} \ ds.
\end{align*}
Next, similarly to \eqref{Bg_Sobolev_embedding}, we have
\begin{align*}
     \abs{\widehat{Bg}_s (k-\ell,kt- \ell s)} \leq \sum_{\ell \in \bz_*^d} \sup_{\eta \in \br^d} \abs{\widehat{Bg}_s(k - \ell, \eta)}
 & \leq \left \Vert \jp{v}^M  B_s(\nabla_x,\nabla_v) g_s
   \right \Vert_{L_{x,v}^2}.
\end{align*}
Therefore,
\begin{align*}
    \abs{T_{HL}} 
    & \leq C \sup_{t \in \br^+}\left \Vert \jp{v}^M  B_t(\nabla_x,\nabla_v) g_t
   \right \Vert_{L_{x,v}^2}   \sum_{\ell \in \bz_*^d} \int_t^{+\infty}  \mathcal{K}(t,s,k,\ell)  \jp{s}^b \abs{\widehat{A \varrho}_s (\ell)}  \ ds,
\end{align*}
where
\begin{align*}
    \mathcal{K}(t,s,k,\ell) := \frac{\jp{t}^b}{\jp{s}^b} \frac{\abs{\ell \cdot k}}{ \beta + \vert\ell\vert^2}  (s-t) e^{-(\lambda (s) - \lambda(t))\jp{k, kt}^\gamma } \frac{e^{-(1-c) \lambda (s) \jp{ k-\ell , kt-\ell s }^\gamma}}{\jp{k-\ell,kt-\ell s}^{\sigma+1}}\mathbbm{1}_{\abs{k-\ell,kt-\ell s} 
    \le \frac{1}{2} \abs{\ell, \ell s}}.
\end{align*}
By taking the $L_t^2L_x^2$ norm, we get by Cauchy-Schwarz in $s$ and $\ell$
\begin{align*}
    \norm{T_{HL}}_{L_t^2L_x^2}^2 
    & \leq C \sup_{t \in \br^+}\left \Vert \jp{v}^M  B_t(\nabla_x,\nabla_v) g_t
   \right \Vert_{L_{x,v}^2}^2 \int_0^{+\infty} \sum_{k \in \bz^d} \left( \sum_{\ell \in \bz_*^d} \int_t^{+\infty}  \mathcal{K}(t,s,k,\ell) \jp{s}^b \abs{\widehat{A \varrho}_s (\ell)}  \ ds \right)^2 dt \\
   & \leq C \sup_{t \in \br^+}\left \Vert \jp{v}^M  B_t(\nabla_x,\nabla_v) g_t
   \right \Vert_{L_{x,v}^2}^2 \\
   & \quad \times \int_0^{+\infty} \sum_{k \in \bz^d} \left( \sum_{\ell \in \bz_*^d} \int_t^{+\infty}  \mathcal{K}(t,s,k,\ell) \jp{s}^{2b} \abs{\widehat{A \varrho}_s (\ell)}^2  \ ds \right) \left( \sum_{\ell \in \bz_*^d} \int_t^{+\infty}  \mathcal{K}(t,s,k,\ell)   \ ds \right) dt.
\end{align*}
Then, by Fubini’s Theorem
\begin{align}
\label{sec2:estimate_T_HL}
    \norm{T_{HL}}_{L_t^2L_x^2}^2 
   & \leq C \sup_{t \in \br^+}\left \Vert \jp{v}^M  B_t(\nabla_x,\nabla_v) g_t
   \right \Vert_{L_{x,v}^2}^2 \sup_{k,t} \left( \sum_{\ell \in \bz_*^d} \int_t^{+\infty}  \mathcal{K}(t,s,k,\ell)   \ ds \right) \nonumber \\
   & \quad \times \int_0^{+\infty} \sum_{k \in \bz^d} \left( \sum_{\ell \in \bz_*^d} \int_t^{+\infty}  \mathcal{K}(t,s,k,\ell) \jp{s}^b \abs{\widehat{A \varrho}_s (\ell)}^2  \ ds \right)  dt \nonumber \\
   & \leq C \sup_{t \in \br^+}\left \Vert \jp{v}^M  B_t(\nabla_x,\nabla_v) g_t
   \right \Vert_{L_{x,v}^2}^2 \sup_{k,t} \left( \sum_{\ell \in \bz_*^d} \int_t^{+\infty}  \mathcal{K}(t,s,k,\ell)   \ ds \right) \nonumber \\
   & \quad \times \int_0^{+\infty} \sum_{\ell \in \bz_*^d} \jp{s}^{2b} \abs{\widehat{A \varrho}_s (\ell)}^2 \left( \sum_{k \in \bz^d} \int_0^{s}  \mathcal{K}(t,s,k,\ell)    \ dt \right)  ds \nonumber \\
   & \leq C \norm{\jp{t}^b A_t  (\nabla_x, t\nabla_x) \varrho_t}_{L_t^2L_x^2}^2 \sup_{t \in \br^+}\left \Vert \jp{v}^M  B_t(\nabla_x,\nabla_v) g_t
   \right \Vert_{L_{x,v}^2}^2 \nonumber \\
   & \quad \times \sup_{k,t} \left( \sum_{\ell \in \bz_*^d} \int_t^{+\infty}  \mathcal{K}(t,s,k,\ell)   \ ds \right)
   \sup_{s,\ell} \left( \sum_{k \in \bz_*^d} \int_0^{s}  \mathcal{K}(t,s,k,\ell)    \ dt \right).
\end{align}
Therefore, the analysis of the term $T_{HL}$ reduces to estimate the two norms of this kernel $\mathcal{K}$. Let us start with the first one. Note that we have for all time $t>s$,
\begin{align}
\label{ineq_lambda}
    \lambda(s) - \lambda(t) = \frac{C}{\jp{t}^\delta} - \frac{C}{\jp{s}^\delta}  \geq C \frac{\abs{s-t}}{\jp{s} \jp{t}^\delta}.
\end{align}
This implies we get enough exponential decay when $s \geq 2t$. Hence, we split the time integral into two regions: $s \leq 2t$ and $s \geq 2t$. Therefore, for 
$s \geq 2t$ we get from \eqref{ineq_lambda}
\begin{align}
\label{ineq_lambda_improved}
    \lambda(s) - \lambda(t) \geq \lambda(2t) - \lambda(t) \geq C \frac{\abs{t}}{\jp{2t} \jp{t}^\delta} \geq C \jp{s}^{-\delta}.
\end{align}
Therefore,
\begin{align*}
    \mathcal{I}_1 := & \sup_{k,t} \left( \sum_{\ell \in \bz_*^d} \int_{2t}^{+\infty}  \mathcal{K}(t,s,k,\ell)   \ ds \right) \\
    &\leq \sup_{k,t} \left( \sum_{\ell \in \bz_*^d} \int_{2t}^{+\infty}  \frac{\jp{t}^b}{\jp{s}^b} \frac{\abs{\ell \cdot k}}{ \beta + \vert\ell\vert^2}  (s-t) e^{- \jp{s}^{-\delta} \jp{k, kt}^\gamma } \frac{e^{-(1-c) \lambda (s) \jp{ k-\ell , kt-\ell s }^\gamma}}{\jp{k-\ell,kt-\ell s}^{\sigma+1}}\mathbbm{1}_{\abs{k-\ell,kt-\ell s} 
    \le \frac{1}{2} \abs{\ell, \ell s}}   \ ds \right).
\end{align*}
In the region $\abs{k-\ell,kt-\ell s} 
    \le \frac{1}{2} \abs{\ell, \ell s}$, we have
\begin{align}
\label{inequality_l_and_k}
    \jp{\ell,\ell s} \leq \jp{k,kt} + \jp{k-\ell,kt-\ell s} \leq  \jp{k,kt} + \frac{1}{2} \jp{\ell,\ell s} \leq 2 \jp{k,kt}.
\end{align}
Hence, by \eqref{sec2:ineq_s-t} and \eqref{inequality_l_and_k}, we get
\begin{align*}
    \mathcal{I}_1
    &\leq  \sup_{k,t} \left( \sum_{\ell \in \bz_*^d} \int_{2t}^{+\infty}  \frac{\jp{t}^b}{\jp{s}^b} \jp{s}  e^{- C \jp{s}^{-\delta} \jp{\ell, \ell s}^\gamma } \frac{e^{-(1-c) \lambda (s) \jp{ k-\ell , kt-\ell s }^\gamma}}{\jp{k-\ell,kt-\ell s}^{\sigma}} \ ds \right) \\
    & \leq \sup_{k,t} \left( \sum_{\ell \in \bz_*^d} \frac{1}{\jp{k-\ell}^{\sigma}} \int_{2t}^{+\infty}   \jp{s}  e^{- C \jp{s}^{\gamma -\delta} }  \ ds \right) \\
    & \leq C \sup_{k,t} \left( \sum_{\ell \in \bz_*^d} \frac{1}{\jp{k-\ell}^{\sigma}}  \right) \leq C.
\end{align*}
Now, let us treat the case $s \leq 2t$. 
\begin{align*}
    \mathcal{I}_2 := & \sup_{k,t} \left( \sum_{\ell \in \bz_*^d} \int_{t}^{2t} \mathcal{K}(t,s,k,\ell)   \ ds \right).
\end{align*}
If $\ell=k$, we have 
\begin{align*}
   \sup_{k,t} \left( \int_{t}^{2t} \mathcal{K}(t,s,k,k)   \ ds \right) 
   & \leq \sup_{k,t} \left( \int_{t}^{2t}  (s-t) e^{-(\lambda (s) - \lambda(t))\jp{k, kt}^\gamma } \frac{e^{-(1-c) \lambda (s) \jp{ k(t-s) }^\gamma}}{\jp{k(t-s)}^\sigma}  \ ds \right) \\
   & \leq \sup_{k,t} \left( \int_{t}^{2t}  e^{-(1-c) \lambda (s) \jp{ k(t-s) }^\gamma}  \ ds \right) \leq C.
\end{align*}
Therefore, we only have to look at the sum over $\ell$ when $\ell \neq k$. Now, we split again the support of the integral between $\abs{kt -\ell s} \leq \frac{1}{100}\abs{kt}$ and $\abs{kt -\ell s} \geq \frac{1}{100}\abs{kt}$.
Let us first treat the easy case $\abs{kt -\ell s} \geq \frac{1}{100}\abs{kt}$, where we can get enough exponential decay. Since we have $s \in (t,2t)$ we get $\abs{kt -\ell s} \geq \frac{1}{200}\abs{ks}$. Therefore, using \eqref{sec2:ineq_s-t} we have
\begin{align*}
    \mathcal{I}_2 & \leq  C \sup_{k,t} \left( \sum_{\ell \in \bz_*^d, \ell \neq k} \int_{t}^{2t} \frac{\jp{t}^b}{\jp{s}^b} \frac{\jp{s}}{\abs{\ell}}  e^{-(\lambda (s) - \lambda(t))\jp{k, kt}^\gamma } \frac{e^{-(1-c) \lambda (s) \jp{ k-\ell , kt-\ell s }^\gamma}}{\jp{k-\ell,kt-\ell s}^{\sigma}}  \ ds \right) \\
     & \leq  C \sup_{k,t} \left( \sum_{\ell \in \bz_*^d, \ell \neq k} \int_{t}^{2t} \frac{\jp{s}}{\abs{\ell}} \frac{e^{-C \abs{kt-\ell s }^\gamma}}{\jp{k-\ell}^{\sigma}}   \ ds \right) \\
     & \leq  C \sup_{k,t} \left( \sum_{\ell \in \bz_*^d, \ell \neq k} \frac{1}{\jp{k - \ell}^{\sigma}}\int_{t}^{2t} \jp{s} e^{-\frac{C}{200} \abs{ks }^\gamma}  \ ds \right) \leq C.
\end{align*}
The second case $\abs{kt -\ell s} \leq \frac{1}{100}\abs{kt}$ is a bit more challenging, and it is the part of the proof where we see the requirement on the Gevrey regularity assumption $\gamma \in \left(\frac{1}{3},1 \right)$. Again using \eqref{sec2:ineq_s-t} we have
\begin{align*}
    \mathcal{I}_2 
     & \leq  C \sup_{k,t} \left( \sum_{\ell \in \bz_*^d, \ell \neq k} \int_{t}^{2t} \frac{\jp{t}^b}{\jp{s}^b} \frac{\jp{s}}{\abs{\ell}} e^{-(\lambda (s) - \lambda(t))\jp{k, kt}^\gamma } \frac{e^{-C \jp{ k-\ell , kt-\ell s }^\gamma}}{\jp{k-\ell,kt-\ell s}^{\sigma}}   \ ds \right),
\end{align*}
where we used that $(1-c) \lambda (s) \geq C$ for some constant $C$. Using Gevrey inequality \eqref{triangle_inequality_Gevrey} and following the argument \cite[page 62]{BMM_13}, we have 
\begin{align*}
    \mathcal{I}_2 
     & \leq  C \sup_{k,t} \left( \sum_{\ell \in \bz_*^d} \sum_{j: \ell_j \neq k_j}  \int_{t}^{2t} \frac{\jp{s}}{\abs{\ell}} e^{-(\lambda (s) - \lambda(t))\jp{k_j, k_jt}^\gamma } \frac{e^{-C \jp{ k_j-\ell_j , k_jt-\ell_j s }^\gamma}}{\jp{k_j-\ell_j,k_jt-\ell_j s}^{\sigma}}  \prod_{i\neq j}^d e^{-C \jp{ k_i-\ell_i }^\gamma}   \ ds \right) \\
     & \leq  C \sup_{k,t} \left( \sum_{j=1}^d  \sum_{\ell_j \in \bz_*} \int_{t}^{2t}  \frac{\jp{s}}{\abs{\ell_j}} e^{-(\lambda (s) - \lambda(t))\jp{k_j, k_jt}^\gamma } \frac{e^{-C \jp{ k_j-\ell_j , k_jt-\ell_j s }^\gamma}}{\jp{k_j-\ell_j,k_jt-\ell_j s}^{\sigma}}    \ ds \right).
\end{align*}
Now that we have reduced the estimates in dimension one and since we are in the case 
$\abs{kt -\ell s} \leq \frac{1}{100}\abs{kt}$ we have $s \approx \frac{kt}{\ell}$, (note since $s>t$ we have $\abs{k}>\abs{\ell}$), then 
\begin{align}
\label{sec2:approx_s-t}
    \abs{s-t} \approx \frac{kt}{\ell} - t = \frac{t}{\ell} (k-\ell) \geq \frac{t}{\ell}.
\end{align}
Therefore, by \eqref{ineq_lambda}, we have
\begin{align}
\label{ineq_lambda_improved_2}
    \lambda(s) - \lambda(t)  \geq C \frac{\abs{t-s}}{\jp{s} \jp{t}^\delta} \geq \frac{\jp{t}^{1-\delta}}{\jp{s}} \frac{1}{\abs{\ell}}.
\end{align}
Hence, for some $\alpha>0$, we get
\begin{align*}
    \mathcal{I}_2 
     & \leq  C \sup_{k,t} \left( \sum_{j=1}^d  \sum_{\ell_j \in \bz_*} \int_{t}^{2t}  \frac{\jp{s}}{\abs{\ell_j}} e^{-\frac{\jp{t}^{1-\delta}}{\jp{s}} \frac{1}{\abs{\ell_j}}\jp{k_j, k_jt}^\gamma } \frac{e^{-C \jp{ k_j-\ell_j , k_jt-\ell_j s }^\gamma}}{\jp{k_j-\ell_j,k_jt-\ell_j s}^{\sigma}}    \ ds \right) \\
     & \leq  C \sup_{k,t} \left( \sum_{j=1}^d  \sum_{\ell_j \in \bz_*} \int_{t}^{2t}  \frac{\jp{s}}{\abs{\ell_j}} \frac{\jp{s}^\alpha  \abs{\ell_j}^{\alpha}}{\jp{t}^{(1-\delta)\alpha} \jp{k_j, k_jt}^{\gamma \alpha}}  \frac{e^{-C \jp{ k_j-\ell_j , k_jt-\ell_j s }^\gamma}}{\jp{k_j-\ell_j,k_jt-\ell_j s}^{\sigma}}    \ ds \right) \\
     & \leq  C \sup_{k,t} \left( \sum_{j=1}^d  \sum_{\ell_j \in \bz_*} \int_{t}^{2t}  \jp{t}^{1+\alpha - (1-\delta)\alpha - \gamma \alpha} \abs{\ell_j}^{\alpha -1 - \gamma \alpha}  \frac{e^{-C \jp{ k_j-\ell_j , k_jt-\ell_j s }^\gamma}}{\jp{k_j-\ell_j,k_jt-\ell_j s}^{\sigma}}    \ ds \right),
\end{align*}
where we used \eqref{inequality_l_and_k}, $\jp{\ell_j, \ell_j s} \geq \abs{ \ell_j s}$, and $s \in (t, 2t)$.
Then we choose $\alpha$ such that $1+\alpha - (1-\delta)\alpha - \gamma \alpha=0$, i.e.,
\begin{align*}
    \alpha = \frac{1}{\gamma - \delta}.
\end{align*}
Next, we do a change of variable $\tau = \ell_j s$. Therefore
\begin{align*}
    \mathcal{I}_2 
     & \leq  C \sup_{k,t} \left( \sum_{j=1}^d  \sum_{\ell_j \in \bz_*} \int_{0}^{+ \infty}   \abs{\ell_j}^{\alpha -1 - \gamma \alpha} \frac{1}{\abs{\ell_j}} \frac{e^{-C \jp{ k_j-\ell_j , k_jt-\tau }^\gamma}}{\jp{k_j-\ell_j,k_jt-\tau}^{\sigma}}    \ d \tau \right),
\end{align*}
In order to be able to sum over $\ell$, we need $\alpha-1 - \gamma \alpha < 1$. Therefore,
\begin{align*}
    \frac{1}{\gamma - \delta} - 1 - \frac{\gamma}{\gamma - \delta} < 1 \quad \implies \frac{1 + 2\delta}{3} < \gamma.
\end{align*}
This inequality gives us the Gevrey regularity threshold that we have to impose. That is $ \gamma \in \left(\frac{1}{3},1 \right)$. Finally,
\begin{align*}
    \mathcal{I}_2 
     & \leq  C \sup_{k,t} \left( \sum_{j=1}^d  \sum_{\ell_j \in \bz_*} \int_{0}^{+ \infty}  \frac{e^{-C \jp{ k_j-\ell_j , k_jt-\tau }^\gamma}}{\jp{k_j-\ell_j,k_jt-\tau}^{\sigma}}    \ d \tau \right) \\
     & \leq  C \sup_{k,t} \left( \sum_{j=1}^d  \sum_{\ell_j \in \bz_*} \jp{k_j-\ell_j}^{\sigma} \int_{0}^{+ \infty}  e^{-C \jp{k_jt-\tau }^\gamma}     \ d \tau \right) \\
     & \leq C.
\end{align*}
Next, we turn to the estimate of the second kernel:
\begin{align*}
    \sup_{s,\ell} \left( \sum_{k \in \bz_*^d} \int_0^{s}  \mathcal{K}(t,s,k,\ell)    \ dt \right).
\end{align*}
Note, in this case, we also have $t<s$, and therefore \eqref{ineq_lambda} holds as well. Next, we split again the time integral into two regions: $t\leq \frac{s}{2}$ and $t \geq \frac{s}{2}$. As before, we get enough exponential decay for $t\leq \frac{s}{2}$, so let us first treat this case.
By \eqref{ineq_lambda_improved}, we get
\begin{align*}
    \mathcal{I}_3 := & \sup_{s,\ell} \left( \sum_{k \in \bz_*^d} \int_0^{\frac{s}{2}}  \mathcal{K}(t,s,k,\ell)    \ dt \right) \\
    &\leq \sup_{s,\ell} \left( \sum_{k \in \bz_*^d} \int_0^{\frac{s}{2}}  \frac{\jp{t}^b}{\jp{s}^b} \frac{\abs{\ell \cdot k}}{ \beta + \vert\ell\vert^2}  (s-t) e^{- \jp{s}^{-\delta} \jp{k, kt}^\gamma } \frac{e^{-(1-c) \lambda (s) \jp{ k-\ell , kt-\ell s }^\gamma}}{\jp{k-\ell,kt-\ell s}^{\sigma+1}}\mathbbm{1}_{\abs{k-\ell,kt-\ell s} 
    \le \frac{1}{2} \abs{\ell, \ell s}}   \ dt \right).
\end{align*}
Then, by \eqref{sec2:ineq_s-t} and \eqref{inequality_l_and_k}, we obtain
\begin{align*}
    \mathcal{I}_3 
    &\leq \sup_{s,\ell} \left( \sum_{k \in \bz_*^d} \int_0^{\frac{s}{2}}  \frac{\jp{t}^b}{\jp{s}^b} \jp{s}  e^{- C \jp{s}^{-\delta} \jp{\ell, \ell s}^\gamma } \frac{e^{-(1-c) \lambda (s) \jp{ k-\ell , kt-\ell s }^\gamma}}{\jp{k-\ell,kt-\ell s}^{\sigma}} \ dt \right) \\
    & \leq \sup_{s,\ell} \left( \sum_{k \in \bz_*^d} \frac{1}{\jp{k-\ell}^{\sigma}} \int_0^{\frac{s}{2}}   \jp{s}  e^{- C \jp{s}^{\gamma -\delta} }  \ dt \right) \\
    & \leq C \sup_{s,\ell} \left( \sum_{k \in \bz_*^d} \frac{1}{\jp{k-\ell}^{\sigma}}  \right) \leq C.
\end{align*}
Then, let us treat the case $t \geq \frac{s}{2}$, i.e.,
\begin{align*}
    \mathcal{I}_4 := & \sup_{s,\ell} \left( \sum_{k \in \bz_*^d} \int_{\frac{s}{2}}^s  \mathcal{K}(t,s,k,\ell)    \ dt \right).
\end{align*}
If $k=\ell$, we have 
\begin{align*}
   \sup_{s,\ell} \left( \int_{\frac{s}{2}}^s \mathcal{K}(t,s,\ell,\ell)   \ dt \right) 
   & \leq \sup_{s,\ell} \left( \int_{\frac{s}{2}}^s  (s-t) e^{-(\lambda (s) - \lambda(t))\jp{\ell, \ell t}^\gamma } \frac{e^{-(1-c) \lambda (s) \jp{ \ell(t-s) }^\gamma}}{\jp{\ell(t-s)}^\sigma}  \ dt \right) \\
   & \leq \sup_{s,\ell} \left( \int_{\frac{s}{2}}^s  e^{-(1-c) \lambda (s) \jp{ \ell (t-s) }^\gamma}  \ dt \right) \leq C.
\end{align*}
Therefore, we only have to look at the sum over $k$ when $k \neq \ell$. Similarly to $\mathcal{I}_2$, we split the frequency between $\abs{kt -\ell s} \leq \frac{1}{100}\abs{\ell s}$ and $\abs{kt -\ell s} \geq \frac{1}{100}\abs{\ell s}$. Therefore, using \eqref{sec2:ineq_s-t} we have
\begin{align*}
    \mathcal{I}_4 & \leq  C \sup_{s,\ell} \left( \sum_{k \in \bz_*^d, k \neq \ell} \int_{\frac{s}{2}}^s \frac{\jp{t}^b}{\jp{s}^b} \frac{\jp{s}}{\abs{\ell}}  e^{-(\lambda (s) - \lambda(t))\jp{k, kt}^\gamma } \frac{e^{-(1-c) \lambda (s) \jp{ k-\ell , kt-\ell s }^\gamma}}{\jp{k-\ell,kt-\ell s}^{\sigma}}  \ dt \right) \\
     & \leq  C \sup_{s,\ell} \left( \sum_{k \in \bz_*^d, k \neq \ell} \int_{\frac{s}{2}}^s \frac{\jp{s}}{\abs{\ell}} \frac{e^{-C \abs{kt-\ell s }^\gamma}}{\jp{k-\ell}^{\sigma}}   \ dt \right) \\
     & \leq  C \sup_{s,\ell} \left( \sum_{k \in \bz_*^d, k \neq \ell} \frac{1}{\jp{k - \ell}^{\sigma}}\int_{\frac{s}{2}}^s \jp{s} e^{-\frac{C}{100} \abs{\ell s }^\gamma}  \ dt \right) \\
    & \leq  C \sup_{s,\ell} \left( \sum_{k \in \bz_*^d, k \neq \ell} \frac{1}{\jp{k - \ell}^{\sigma}} \jp{s}^2 e^{-\frac{C}{100} \abs{\ell s }^\gamma} \right)  \leq C.
\end{align*}
Let us now estimate the case $\abs{kt -\ell s} \leq \frac{1}{100}\abs{\ell s}$. By \eqref{sec2:ineq_s-t}, we obtain
\begin{align*}
    \mathcal{I}_4 
    & \leq  C \sup_{s,\ell} \left( \sum_{k \in \bz_*^d, k \neq \ell} \int_{\frac{s}{2}}^s \frac{\jp{t}^b}{\jp{s}^b} \frac{\jp{s}}{\abs{\ell}}  e^{-(\lambda (s) - \lambda(t))\jp{k, kt}^\gamma } \frac{e^{-(1-c) \lambda (s) \jp{ k-\ell , kt-\ell s }^\gamma}}{\jp{k-\ell,kt-\ell s}^{\sigma}}  \ dt \right) \\
    & \leq  C \sup_{s,\ell} \left( \sum_{k \in \bz_*^d, k \neq \ell} \int_{\frac{s}{2}}^s \frac{\jp{s}}{\abs{\ell}}  e^{-(\lambda (s) - \lambda(t))\jp{k, kt}^\gamma } \frac{e^{-C \jp{ k-\ell , kt-\ell s }^\gamma}}{\jp{k-\ell,kt-\ell s}^{\sigma}}  \ dt \right),
\end{align*}
where we used that $(1-c) \lambda (s) \geq C$ for some constant $C$. Using Gevrey inequality \eqref{triangle_inequality_Gevrey} and following the argument \cite[page 62]{BMM_13}, we have 
\begin{align*}
    \mathcal{I}_4 
     & \leq  C \sup_{s,\ell} \left( \sum_{k \in \bz_*^d} \sum_{j: k_j \neq \ell_j}  \int_{\frac{s}{2}}^s \frac{\jp{s}}{\abs{\ell}} e^{-(\lambda (s) - \lambda(t))\jp{k_j, k_jt}^\gamma } \frac{e^{-C \jp{ k_j-\ell_j , k_jt-\ell_j s }^\gamma}}{\jp{k_j-\ell_j,k_jt-\ell_j s}^{\sigma}}  \prod_{i\neq j}^d e^{-C \jp{ k_i-\ell_i }^\gamma}   \ dt \right) \\
     & \leq  C \sup_{s,\ell} \left( \sum_{j=1}^d  \sum_{k_j \in \bz_*} \int_{\frac{s}{2}}^s  \frac{\jp{s}}{\abs{\ell_j}} e^{-(\lambda (s) - \lambda(t))\jp{k_j, k_jt}^\gamma } \frac{e^{-C \jp{ k_j-\ell_j , k_jt-\ell_j s }^\gamma}}{\jp{k_j-\ell_j,k_jt-\ell_j s}^{\sigma}}    \ dt \right).
\end{align*}
Now that we have reduced the estimates in dimension one and since we are in the case 
$\abs{kt -\ell s} \leq \frac{1}{100}\abs{\ell s}$ we have $s \approx \frac{kt}{\ell}$, (note since $s>t$ we have $\abs{k}>\abs{\ell}$). Then, as before, using \eqref{sec2:approx_s-t}, we get \eqref{ineq_lambda_improved_2}.
Hence, for some $\alpha>0$, we obtain
\begin{align*}
    \mathcal{I}_4 
     & \leq  C \sup_{s,\ell} \left( \sum_{j=1}^d  \sum_{k_j \in \bz_*} \int_{\frac{s}{2}}^s  \frac{\jp{s}}{\abs{\ell_j}} e^{-\jp{t}^{1-\delta} \jp{s}^{-1} \abs{\ell_j}^{-1} \jp{k_j, k_jt}^\gamma } \frac{e^{-C \jp{ k_j-\ell_j , k_jt-\ell_j s }^\gamma}}{\jp{k_j-\ell_j,k_jt-\ell_j s}^{\sigma}}    \ dt \right) \\
     & \leq  C \sup_{s,\ell} \left( \sum_{j=1}^d  \sum_{k_j \in \bz_*} \int_{\frac{s}{2}}^s  \frac{\jp{s}}{\abs{\ell_j}} \frac{ \abs{\ell_j}^\alpha \jp{s}^\alpha}{ \jp{t}^{(1 - \delta) \alpha} \jp{k_j, k_jt}^{\gamma \alpha}}  \frac{e^{-C \jp{ k_j-\ell_j , k_jt-\ell_j s }^\gamma}}{\jp{k_j-\ell_j,k_jt-\ell_j s}^{\sigma}}    \ dt \right) \\
     & \leq  C \sup_{s,\ell} \left( \sum_{j=1}^d  \sum_{k_j \in \bz_*} \int_{\frac{s}{2}}^s \frac{\jp{s}^{1 - (1 -\delta) \alpha - \gamma \alpha + \alpha}}{\abs{\ell_j}^{1-\alpha+\gamma \alpha}}   \frac{e^{-C \jp{ k_j-\ell_j , k_jt-\ell_j s }^\gamma}}{\jp{k_j-\ell_j,k_jt-\ell_j s}^{\sigma}}    \ dt \right),
\end{align*}
where we used  $\jp{k_j, k_j t} \geq \abs{k_j t}$, $t \geq \frac{s}{2}$, and $\abs{k}>\abs{\ell}$.
Then, we choose $\alpha$ such that $1 - (1 -\delta) \alpha - \gamma \alpha + \alpha=0$, i.e.,
\begin{align*}
    \alpha = \frac{1}{\gamma - \delta}.
\end{align*}
Next, we do a change of variable $\tau = k_j t$. Therefore
\begin{align*}
    \mathcal{I}_4 
     & \leq  C \sup_{s,\ell} \left( \sum_{j=1}^d  \sum_{k_j \in \bz_*} \int_{0}^{+ \infty}  \abs{\ell_j}^{\alpha - \gamma \alpha -1} \frac{1}{\abs{k_j}} \frac{e^{-C \jp{ k_j-\ell_j , \tau-\ell_j s }^\gamma}}{\jp{k_j-\ell_j,\tau-\ell_j s}^{\sigma}}    \ d \tau \right) \\
    & \leq  C \sup_{s,\ell} \left( \sum_{j=1}^d  \sum_{k_j \in \bz_*} \int_{0}^{+ \infty}  \abs{\ell_j}^{\alpha - \gamma \alpha -1} \frac{1}{\abs{\ell_j}} \frac{e^{-C \jp{ k_j-\ell_j ,\tau-\ell_j s}^\gamma}}{\jp{k_j-\ell_j,\tau-\ell_j s}^{\sigma}}    \ d \tau \right),
\end{align*}
In order to be able to sum over $k$ uniformly in $\ell$, we need $\alpha-1 - \gamma \alpha < 1$. Therefore,
\begin{align*}
    \frac{1}{\gamma - \delta} - 1 - \frac{\gamma}{\gamma - \delta} < 1 \quad \implies \frac{1 + 2\delta}{3} < \gamma.
\end{align*}
This inequality gives us the Gevrey regularity threshold that we have to impose. That is $ \gamma \in \left(\frac{1}{3},1 \right)$. Finally,
\begin{align*}
    \mathcal{I}_4 
     & \leq  C \sup_{s,\ell} \left( \sum_{j=1}^d  \sum_{k_j \in \bz_*} \int_{0}^{+ \infty}  \frac{e^{-C \jp{ k_j-\ell_j , \tau-\ell_j s }^\gamma}}{\jp{k_j-\ell_j,\tau-\ell_j s}^{\sigma}}    \ d \tau \right) \\
     & \leq  C \sup_{s,\ell} \left( \sum_{j=1}^d  \sum_{k_j \in \bz_*} \jp{k_j-\ell_j}^{\sigma} \int_{0}^{+ \infty}  e^{-C \jp{\tau-\ell_j s }^\gamma}     \ d \tau \right)  \leq C,
\end{align*}
which concludes the analysis of the second kernel.

\emph{Nonlinear estimate on $T_{HH}$:}Finally, we estimate the last term of \eqref{sec2:splitting_T_NL}. 
\begin{align*}
    T_{HH} := \jp{t}^b A_t(k, kt) \sum_{\ell \in \bz_*^d} \int_t^{+\infty} \mathbbm{1}_{\abs{\ell, \ell s} \approx \abs{k-\ell,kt-\ell s}} \frac{\ell \cdot k}{\vert\ell\vert^2}  (s-t) \widehat{\varrho}_s (\ell) \widehat{g}_s (k-\ell,kt- \ell s) \ d s.
\end{align*}
In this region, we have $\frac{1}{2}\abs{\ell, \ell s} \leq \abs{k-\ell,kt-\ell s} \leq 2\abs{\ell, \ell s}$. Therefore, we can apply inequality \eqref{gevrey_inequality_4} 
\begin{align*}
    \jp{k,kt}^\gamma = \jp{k-\ell + \ell, kt-\ell s + \ell s}^\gamma 
    \leq c \left( \jp{k-\ell, kt-\ell s}^\gamma  + \jp{\ell, \ell s}^\gamma \right),
\end{align*}
where $c=c(\gamma) \in (0,1)$.
This allows us to bound 
\begin{align*}
    A_t(k, kt) 
    \leq \jp{k,kt}^\sigma e^{c \lambda(t) \jp{k-\ell, kt-\ell s}^\gamma} e^{c  \lambda(t)\jp{\ell, \ell t}^\gamma} \leq \jp{k-\ell, kt-\ell s}^\sigma \jp{\ell,\ell s}^\sigma e^{c \lambda(s) \jp{k-\ell, kt-\ell s}^\gamma} e^{c  \lambda(s)\jp{\ell, \ell s}^\gamma},
\end{align*}
where we used $s>t$ and $\lambda(t)$ is an increasing function.
Hence,
\begin{align*}
    T_{HH} 
    & \leq   \sum_{\ell \in \bz_*^d} \int_t^{+\infty} \mathbbm{1}_{\abs{\ell, \ell s} \approx \abs{k-\ell,kt-\ell s}} e^{-(1-c)  \lambda(s)\jp{\ell, \ell s}^\gamma} e^{-(1-c) \lambda(s) \jp{k-\ell, \eta-\ell s}^\gamma}  \\
    & \quad \times \jp{s}^b \frac{\abs{\ell \cdot k}}{ \beta + \vert\ell\vert^2}  \abs{s-t} \abs{\widehat{A\varrho}_s (\ell)} \abs{\widehat{Ag}_s (k-\ell,kt- \ell s)} \ d s.
\end{align*}
Using \eqref{sec2:ineq_s-t} we can bound $\frac{\abs{\ell \cdot k}}{ \beta + \vert\ell\vert^2}  \abs{s-t} \leq \jp{s} \jp{k-\ell, \eta-\ell s}$. Therefore,
\begin{align*}
    T_{HH} 
    & \leq \jp{t}^b  \sum_{\ell \in \bz_*^d} \int_t^{+\infty} \mathbbm{1}_{\abs{\ell, \ell s} \approx \abs{k-\ell,kt-\ell s}} e^{-(1-c)  \lambda(s)\jp{\ell, \ell s}^\gamma} e^{-(1-c) \lambda(s) \jp{k-\ell, \eta-\ell s}^\gamma} \\
    & \qquad \times \jp{s}^{b+1} \abs{\widehat{A\varrho}_s (\ell)} \abs{\widehat{Bg}_s (k-\ell,kt- \ell s)} \ d s.
\end{align*}
By taking the $L_t^2L_x^2$ norm and proceeding as for the term $T_{LH}$ we get the final estimate thanks to the two exponential decays in the integral. That is,
\begin{align}
\label{sec2:estimate_T_HH}
     \norm{T_{HH}}_{L_t^2 L_x^2}^2 
      \leq C \norm{\jp{t}^b A_t (\nabla_x, t \nabla_x) \varrho}_{L_t^2 L_x^2}^2 \sup_{t \in \br_+} \norm{\jp{v}^M B_t (\nabla_x, \nabla_v)g_t}_{L_{x,v}^2}^2.
\end{align}
Hence, by collecting \eqref{sec2:estimate_ginf},\eqref{sec2:estimate_h(U)}, \eqref{sec2:estimate_T_LH}, \eqref{sec2:estimate_T_HL}, and \eqref{sec2:estimate_T_HH} we get the result.
\end{proof}

\section{A priori estimates on the distribution function \texorpdfstring{$g$}{Lg}}
\label{section_distrg}

In this Section,
we give a priori estimates for the distribution function $g$. 
Inserting the expression for the modified density 
$\widehat{\varrho}_t(k)=(|k|^2 + \beta)\widehat{U}_t(k)$ 
in \eqref{FT_g}, 
we get the following equation for $\widehat{g}_t(k,\eta)$:
    \begin{equation*}
    \label{sec3:eqg}
	\partial_t \widehat{g}_t(k ,\eta) 
    = -\frac{[(\eta - kt)\cdot k]}{|k|^2+\beta} \widehat{\varrho}_t (k) \widehat{\mu} (\eta -kt)
    - \sum_{\ell \in \mathbb{Z}^d} 
    \frac{[(\eta -kt)\cdot \ell]}{|\ell|^2+\beta} \widehat{\varrho}_t(\ell) \widehat{g}_t(k - \ell ,\eta  - \ell t).
    \end{equation*}
    
We consider energy estimates
by using the $L^2$-based norm introduced in \eqref{sec1:normg}. 
Recalling the notation
    $B_t(k,\eta)=e^{\lambda(t) 
    \langle k, \eta \rangle^{\gamma}}  
    \langle k, \eta \rangle^{\sigma+1},$ 
note that for any multi-index $j \in \bn^d$ 
and $\lambda(t)\in[0,\lambda]$, there exists a constant $C=C(|j|,\lambda)$ such that
    \begin{equation*}
	\left| \partial_\eta^j B_t (k, \eta) \right| 
    \le C 
    \frac{1}{\langle k, \eta \rangle^{\vert j \vert (1-\gamma)} } B_t (k, \eta).
    \end{equation*}
Thanks to the previous inequality, we get the bound
    \begin{align}
    \label{sec3:equi_norm_1}
    \norm{\jp{v}^M B_t(\nabla_x, \nabla_v) g_t}_{L_{x,v}^2}
    & \le C  \sum_{\abs{j} \leq M} 
    \norm{B_t (k,\eta) \partial_\eta^j \widehat{g}_t}_{L_{k,\eta}^2} 
    = C \sum_{\abs{j} \leq M} \norm{B_t (\nabla_x, \nabla_v) (v^j g_t)}_{L_{x,v}^2},
    \end{align}
and 
\begin{align}
\label{sec3:equi_norm_2}
    \sum_{\abs{j} \leq M} \norm{B_t (\nabla_x, \nabla_v) (v^j g_t)}_{L_{x,v}^2}
    \leq \sum_{\abs{j} \leq M} \norm{ \partial_\eta^j \left(B_t (k,\eta)  \widehat{g}_t)\right)}_{L_{k,\eta}^2}
    = \norm{\jp{v}^M B_t(\nabla_x, \nabla_v) g_t}_{L_{x,v}^2}.
    \end{align}
Therefore,
\begin{align}
\label{sec3:equi_norm}
    \norm{\jp{v}^M B_t(\nabla_x, \nabla_v) g_t}_{L_{x,v}^2} \approx \sum_{\abs{j} \leq M} \norm{B_t (\nabla_x, \nabla_v) (v^j g_t)}_{L_{x,v}^2},
\end{align}
and it suffices to get a priori estimates on the second norm.

    \begin{prop}[A priori estimates on the distribution function]
       \label{sec3:prop1}
       Let $g$ be the solution of the Vlasov equation $\eqref{nlequation_g}$. Let $\mu$ be an analytic homogeneous equilibrium satisfying assumptions (H1)-(H3) with $\lambda>0$ as in \eqref{norm_mu} and let $g_\infty$ be a Gevrey function of mean zero satisfying \eqref{norm_g_infty}. Then for $\sigma > 10+d$, $M > \frac{d}{2}$,  it holds that
    \begin{align}
    \label{sec3:prop1_estimate}
        & \frac{1}{2}\frac{d}{dt} \sum_{0 \leq \abs{j} \leq M}\norm{B_t(\nabla_x, \nabla_v) (v^j g_t)}_{L^2_{x,v}}^2 
        + \left( \jp{t}^{-2}\norm{ A_t(\nabla_x, t\nabla_x)\varrho_t}_{L^2_x} 
        - \dot{\lambda}(t) \right) \norm{\jp{\nabla_x , \nabla_v}^{\gamma/2} B_t(\nabla_x, \nabla_v) (v^j g_t)}_{L^2_{x,v}} \nonumber \\
        & \quad \geq - \jp{t}^2 \norm{A_t (\nabla_x, t\nabla_x) \varrho_t}_{L^2_x}\sum_{0 \leq \abs{j} \leq M}\norm{B_t(\nabla_x, \nabla_v) (v^j g_t)}_{L^2_{x,v}}^2 \nonumber \\
        & \qquad - \jp{t}\norm{A_t (\nabla_x, t\nabla_x)  \varrho_t}_{L^2_x}\sum_{0 \leq \abs{j} \leq M}\norm{B_t(\nabla_x, \nabla_v) (v^j g_t)}_{L^2_{x,v}}.
    \end{align}
    \end{prop}
    
The following two Lemmas are used several times in the proof of Proposition \ref{sec3:prop1}. 
\begin{lem}[Lemma 3.1 in \cite{BMM_13}]
\label{sec3:lemma_Bedrossian}
    \begin{itemize}
        \item[a)] 
        Let $\phi, \varphi \in L^2 (\bz^d \times \br^d)$ and $ \jp{k}^\theta \psi \in L^2 (\bz^d)$  for $\theta > \frac{d}{2}$. Then, for any $t \in \br$, we have 
        \begin{align}
        \label{sec3:lemma_Bedrossian_estimate_1}
            \abs{\sum_{k,\ell \in \mathbb{Z}^d} \int_{\br^d} \phi (k,\eta) \psi(\ell) \varphi(k-\ell, \eta-\ell t) \ d\eta} \leq C \norm{\phi}_{L_{k,\eta}^2} 
            \norm{\jp{k}^{\theta} \psi}_{L_{k}^2}
            \norm{\varphi}_{L_{k,\eta}^2}.
        \end{align}
        \item[b)] 
        Let $\phi, \jp{k}^\theta \varphi \in L^2 (\bz^d \times \br^d)$ and $ \psi \in L^2 (\bz^d)$  for $\theta > \frac{d}{2}$. Then, for any $t \in \br$, we have 
    \begin{align}
    \label{sec3:lemma_Bedrossian_estimate_2}
        \abs{\sum_{k,\ell \in \mathbb{Z}^d} \int_{\br^d} \phi (k,\eta) \psi(\ell) \varphi(k-\ell, \eta-\ell t) \ d\eta} \leq C \norm{\phi}_{L_{k,\eta}^2}
        \norm{ \psi}_{L_{k}^2}
        \norm{\jp{k}^{\theta} \varphi}_{L_{k,\eta}^2}.
    \end{align}
    \end{itemize}
\end{lem}
The proof follows by applying the Cauchy-Schwartz and Young's inequality, see e.g. \cite[Lemma 3.1]{BMM_13}.
By \eqref{gevrey_inequality} and \eqref{gevrey_inequality_2}, it also holds the following result.
\begin{lem}
\label{lemma_gevrey_inequality_3}
    Let $x,y \geq 0$ and $\gamma \in (0,1)$ such that $\abs{x-y} < \frac{x}{2}$. Then there exists a $c\in(0,1)$ such that
    \begin{align*}
    %\label{gevrey_inequality_3}
        \abs{e^{\lambda x^\gamma} - e^{\lambda y^\gamma}} 
        \leq \lambda C \frac{\abs{x-y}}{x^{1-\gamma} + y^{1-\gamma}} e^{c \lambda \abs{x-y}^\gamma} e^{\lambda x^\gamma}.
    \end{align*}
\end{lem}

\begin{proof}[Proof of Proposition \ref{sec3:prop1}]
We compute the time derivative of the $L_{x,v}^2$ norm in \eqref{sec3:equi_norm}.
By the product rule, we get,
\begin{align*}
    \frac{1}{2} \frac{d}{dt}
    \norm{B_t(\nabla_x, \nabla_v) (v^j g_t)}_{L^2_{x,v}}^2 
    & = \dot{\lambda}(t)  \norm{\jp{\nabla_x , \nabla_v}^{\gamma/2} B_t(\nabla_x, \nabla_v) (v^j g_t)}_{L^2_{x,v}} \\
    &\quad +  \sum_{k \in \mathbb{Z}^d} \Re
    \int_{\mathbb{R}^d} 
    \overline{B_t(k,\eta)\partial_\eta^j \widehat{g}_t(k,\eta)}
    B_t(k,\eta) \partial_t\partial_\eta^j \widehat{g}_t(k,\eta)
     \ d\eta.
\end{align*}
Recall that the Fourier transform of the Vlasov equation \eqref{nlequation_g} is given by
\begin{align*}
    \partial_t \widehat{g}_t(k,\eta) & =
    - \widehat{E}_t (k)\cdot \widehat{\nabla_v \mu} (\eta - kt) 
    - i \sum_{\ell \in \mathbb{Z}^d} (\eta -kt) \cdot \widehat{E}_{t} (\ell) \widehat{g}_t(k - \ell ,\eta  - \ell t).
\end{align*}
Thus, we get
\begin{align}
\label{time_derivative_L2_norm}
    \frac{1}{2} \frac{d}{dt} & \norm{B_t(\nabla_x, \nabla_v) (v^j g_t)}_{L^2_{x,v}}^2 
    - \dot{\lambda}(t)  \norm{\jp{\nabla_x , \nabla_v}^{\gamma/2} B_t(\nabla_x, \nabla_v) (v^j g_t)}_{L^2_{x,v}}=\mathcal{T}_L + \mathcal{T}_{NL},
\end{align}
where 
\begin{align}
\label{sec3:T_L}
    & \mathcal{T}_L:=- \sum_{k \in \mathbb{Z}^d} \Re
    \int_{\mathbb{R}^d}  
    \overline{B_t(k,\eta)\partial_{\eta}^j \widehat{g}_t (k,\eta)} 
    B_t(k,\eta)
    \widehat{E}_t (k)\cdot \partial_{\eta}^j \widehat{\nabla_v \mu} (\eta - kt)  \ d\eta, \\
\label{sec3:T_NL}
    & \mathcal{T}_{NL}:=-
     \sum_{k \in \mathbb{Z}^d} 
     \Re i \int_{\mathbb{R}^d} 
     \overline{B_t(k,\eta) \partial_{\eta}^j \widehat{g}_t (k,\eta)}
     B_t(k,\eta)\partial_{\eta}^j \left( \sum_{\ell \in \mathbb{Z}^d} (\eta -kt) \cdot \widehat{E}_{t} (\ell) \widehat{g}_t(k - \ell ,\eta  - \ell t)  \right)   \ d\eta.
\end{align}
Using the following bound
\begin{align}
\label{sec3:bound T_L+T_NL}
    \abs{\mathcal{T}_L + \mathcal{T}_{NL}}
    & \leq  C \jp{t}
      \norm{A_t (\nabla_x, t\nabla_x)  \varrho_t}_{L^2_x} \norm{B_t(\nabla_x, \nabla_v) (v^j g_t)}_{L^2_{x,v}} \nonumber \\
    & \quad + C \jp{t}^{-2}  
    \norm{A\varrho}_{L^2_x}
    \norm{\jp{\nabla_x , \nabla_v}^{\gamma/2} B_t(\nabla_x, \nabla_v) (v^j g_t)}_{L^2_{x,v}}^2 \nonumber \\
    & \quad + C \jp{t} \norm{A_t (\nabla_x, t\nabla_x) \varrho_t}_{L^2_x} \norm{B_t(\nabla_x, \nabla_v) (v^j g_t)}_{L^2_{x,v}}^2
\end{align}
in \eqref{time_derivative_L2_norm}, we get
\begin{align*}
     \frac{1}{2} \frac{d}{dt} & \norm{B_t(\nabla_x, \nabla_v) (v^j g_t)}_{L^2_{x,v}}^2 
    - \dot{\lambda}(t)  \norm{\jp{\nabla_x , \nabla_v}^{\gamma/2} B_t(\nabla_x, \nabla_v) (v^j g_t)}_{L^2_{x,v}} \geq - \abs{\mathcal{T}_L + \mathcal{T}_{NL}} \\
    & \geq - C \jp{t}
      \norm{A_t (\nabla_x, t\nabla_x)  \varrho_t}_{L^2_x} \norm{B_t(\nabla_x, \nabla_v) (v^j g_t)}_{L^2_{x,v}} \nonumber \\
    & \quad - C \jp{t}^{-2}  
    \norm{A\varrho}_{L^2_x}
    \norm{\jp{\nabla_x , \nabla_v}^{\gamma/2} B_t(\nabla_x, \nabla_v) (v^j g_t)}_{L^2_{x,v}}^2 \nonumber \\
    & \quad - C \jp{t} \norm{A_t (\nabla_x, t\nabla_x) \varrho_t}_{L^2_x} \norm{B_t(\nabla_x, \nabla_v) (v^j g_t)}_{L^2_{x,v}}^2.
\end{align*}
This shows estimate \eqref{sec3:prop1_estimate} of Proposition \eqref{sec3:prop1} assuming that we have \eqref{sec3:bound T_L+T_NL}. So let us prove that estimate \eqref{sec3:bound T_L+T_NL} holds.

\textbf{Estimate on $\mathcal{T}_L$: linear term.}
 We start studying the linear term \eqref{sec3:T_L} involving the homogeneous equilibrium $\mu$. Using that
$
    B_t(k,\eta) \leq B_t(\eta-kt)B_t(k,kt),
$
where $B_t(\eta)=e^{\lambda(t) \jp{\eta}}$,
and recalling the relation 
$\widehat{E}_t (k) = -ik \widehat{U}_t (k)$ 
and the definition of the modified density
$\widehat{\varrho}_t(k)=(\abs{k}^2 + \beta) \widehat{U}_t(k)$,
we obtain
    \begin{align*}
    &\abs{\sum_{k \in \mathbb{Z}^d} \Re \int_{\mathbb{R}^d} 
     \overline{B_t(k,\eta) \partial_{\eta}^j \widehat{g}_t (k,\eta)}
    B_t(k,\eta)   \widehat{E}_t (k)\cdot \partial_{\eta}^j \widehat{\nabla_v \mu} (\eta - kt)  \ d\eta} \\
    & \quad \leq \sum_{k \in \mathbb{Z}^d} \int_{\mathbb{R}^d} 
    B_t(k,kt)   \frac{\abs{k}}{\abs{k}^2 + \beta} \abs{\widehat{\varrho}_t(k)} B_t(\eta-kt) \abs{\partial_{\eta}^j \widehat{\nabla_v \mu} (\eta - kt)} B_t(k,\eta) \abs{\partial_{\eta}^j \widehat{g}_t (k,\eta)}  \ d \eta \\
    & \quad \leq \jp{t} \sum_{k \in \mathbb{Z}_*^d} \abs{\widehat{A \varrho}_t (k)} \int_{\mathbb{R}^d}    
    B_t(\eta-kt) \abs{\partial_{\eta}^j \widehat{\nabla_v \mu} (\eta - kt)} B_t(k,\eta) \abs{\partial_{\eta}^j \widehat{g}_t (k,\eta)}  \ d \eta,
    \end{align*}
where in the last inequality, we used the fact that
$|k|^{-1} B_t(k,kt) \leq \jp{t} A_t(k,kt)$, for $k \neq 0$.
By applying Cauchy-Schwarz in $\eta$ we obtain,
\begin{align*}
    & \abs{ \sum_{k \in \mathbb{Z}^d} \Re
    \int_{\mathbb{R}^d} 
    \overline{B_t(k,\eta) \partial_{\eta}^j \widehat{g}_t (k,\eta)}
    B_t(k,\eta)
    \widehat{E}_t (k)\cdot \partial_{\eta}^j \widehat{\nabla_v \mu} (\eta - kt)  \ d\eta} \\
    & \quad \leq \jp{t} \sum_{k \in \mathbb{Z}_*^d} \abs{\widehat{A \varrho}_t (k)} \left(\int_{\mathbb{R}^d}    B_t^2(\eta-kt) \abs{\partial_{\eta}^j \widehat{\nabla_v \mu} (\eta - kt)}^2 \ d \eta \right)^\frac{1}{2}  \left(\int_{\mathbb{R}^d} B_t^2(k,\eta) \abs{\partial_{\eta}^j \widehat{g}_t (k,\eta)}^2  \ d \eta \right)^\frac{1}{2}.
\end{align*}
Using the assumption \eqref{norm_mu} on the equilibrium, we can bound the first integral as follows 
\begin{align*}
    \int_{\mathbb{R}^d}    
    &B_t^2(\eta-kt) \abs{\partial_{\eta}^j \widehat{\nabla_v \mu} (\eta - kt)}^2 \ d \eta 
     = \int_{\mathbb{R}^d}    B_t^2(\eta-kt) \abs{\partial_{\eta}^j \left( (\eta-kt) \widehat{\mu} (\eta - kt) \right)}^2 \ d \eta \\
    & \le C \int_{\mathbb{R}^d}    B_t^2(\eta-kt) \jp{\eta-kt}^2
    \left(\vert\partial_{\eta}^j \widehat{\mu} (\eta - kt)\vert^2 
    +\vert\partial_{\eta}^{j-1} \widehat{\mu} (\eta - kt)\vert^2 \right)
    \ d \eta \\
    & \leq C \left(\sum_{|p|\le M}
    \norm{
    e^{\lambda \langle \eta \rangle} 
    \vert \partial_\eta^p\widehat{\mu}(\eta)\vert}^2_{L^\infty_\eta}
    \right)
    \int_{\mathbb{R}^d}   
    B_t^2(\eta-kt) \jp{\eta-kt}^2 e^{-2 \lambda \langle \eta-kt \rangle} \ d \eta \\
    &\le C(\lambda, \sigma) 
    \left(\sum_{|p|\le M}
    \norm{
    e^{\lambda \langle \eta \rangle} 
    \vert \partial_\eta^p\widehat{\mu}(\eta)\vert}^2_{L^\infty_\eta}
    \right),
\end{align*}
where we recall that $\lambda > \lambda(t)$.
Hence, applying Cauchy-Schwarz in $k$ we obtain,
\begin{align}
\label{sec3:estimate_TL}
    & \abs{\Re  \sum_{k \in \mathbb{Z}^d} \int_{\mathbb{R}^d} 
    \overline{B_t(k,\eta) \partial_{\eta}^j \widehat{g}_t (k,\eta)}
    B_t(k,\eta)   \widehat{E}_t (k)\cdot \partial_{\eta}^j \widehat{\nabla_v \mu} (\eta - kt)   \ d\eta} \nonumber \\
    & \quad \leq C \jp{t} 
    \left(\sum_{|p|\le M}
    \norm{
    e^{\lambda \langle \eta \rangle} 
    \vert \partial_\eta^p\widehat{\mu}(\eta)\vert}_{L^\infty_\eta}
    \right)  \left( \sum_{k \in \mathbb{Z}^d} \abs{\widehat{A \varrho}_t (k)}^2 \right)^\frac{1}{2} \left( \sum_{k \in \mathbb{Z}^d} \int_{\mathbb{R}^d} B_t^2(k,\eta) \abs{\partial_{\eta}^j \widehat{g}_t (k,\eta)}^2  \ d \eta \right)^\frac{1}{2} \nonumber \\
    &\quad =C \jp{t}
    \left(\sum_{|p|\le M}
    \norm{
    e^{\lambda \langle \eta \rangle} 
    \vert \partial_\eta^p\widehat{\mu}(\eta)\vert}_{L^\infty_\eta}
    \right)  \norm{A_t (\nabla_x, t\nabla_x)  \varrho_t}_{L^2_x} \norm{B_t(\nabla_x, \nabla_v) (v^j g_t)}_{L^2_{x,v}}.
\end{align}

\textbf{Estimate on $\mathcal{T}_{NL}$: nonlinear term.}
We now treat the nonlinear term $\mathcal{T}_{NL}$ \eqref{sec3:T_NL}. Introducing commutators in the $\eta$-derivatives, we get
$ \mathcal{T}_{NL}=\mathcal{T}^{(1)}_{NL}+ \mathcal{T}^{(2)}_{NL},$ where
\begin{align*}
    \mathcal{T}^{(1)}_{NL}&:= - \sum_{k, \ell \in \mathbb{Z}^d} 
     \Re i \int_{\mathbb{R}^d} 
     \overline{B_t(k,\eta) \partial_{\eta}^j \widehat{g}_t (k,\eta)}
     B_t(k,\eta) (\eta - kt) \cdot \widehat{E}_{t} (\ell) \partial_{\eta}^j \widehat{g}_t(k - \ell ,\eta  - \ell t)     \ d\eta,\\
     \mathcal{T}^{(2)}_{NL} 
    &:= - \sum_{k, \ell \in \mathbb{Z}^d}
    \Re i \int_{\mathbb{R}^d} 
    \overline{B_t(k,\eta) \partial_{\eta}^j \widehat{g}_t (k,\eta)}
    B_t(k,\eta) 
    \partial_\eta^j \left( 
    (\eta-kt)\cdot
    \widehat{E}_{t} (\ell) \widehat{g}_t(k - \ell ,\eta  - \ell t)
    \right)\\
    &\quad  +
    \sum_{k, \ell \in \mathbb{Z}^d}
    \Re i \int_{\mathbb{R}^d} 
    \overline{B_t(k,\eta) \partial_{\eta}^j \widehat{g}_t (k,\eta)}
    B_t(k,\eta) 
     (\eta-kt)\cdot\widehat{E}_{t} (\ell) \partial_{\eta}^j\widehat{g}_t(k - \ell ,\eta  - \ell t) 
     \ d\eta.
\end{align*}

\emph{Nonlinear estimate on $\mathcal{T}_{NL}^{(1)}$:}
we introduce a commutator in the weight $B_t(k,\eta)$ which reflects the Hamiltonian structure of the equation:
\begin{align}
\label{sec3:NL}
    \mathcal{T}^{(1)}_{NL}  & = - \Re i  \sum_{k, \ell \in \mathbb{Z}^d}  \int_{\mathbb{R}^d}
    \overline{B_t(k,\eta) \partial_{\eta}^j \widehat{g}_t (k,\eta)}
    \left( B_t(k,\eta) - B_t(k-\ell,\eta-\ell t)\right) (\eta - kt) \cdot \widehat{E}_{t} (\ell) \partial_{\eta}^j \widehat{g}_t(k - \ell ,\eta  - \ell t)     \ d\eta \nonumber \\
    & - \Re i  \sum_{k, \ell \in \mathbb{Z}^d}  \int_{\mathbb{R}^d}
    \overline{B_t(k,\eta) \partial_{\eta}^j \widehat{g}_t (k,\eta)} 
    B_t(k-\ell,\eta-\ell t) (\eta - kt) \cdot \widehat{E}_{t} (\ell) \partial_{\eta}^j \widehat{g}_t(k - \ell ,\eta  - \ell t)    \ d\eta.
\end{align}
We can show that the last term in \eqref{sec3:NL} is zero: indeed by the change of variables $\ell \rightarrow -\ell$ and $k \rightarrow k-\ell$, we get
\begin{align*}
    & - \Re  \sum_{k, \ell \in \mathbb{Z}^d}  \int_{\mathbb{R}^d} 
    \overline{B_t(k,\eta) \partial_{\eta}^j \widehat{g}_t (k,\eta)} 
    B_t(k-\ell,\eta-\ell t) (\eta - kt) \cdot i\widehat{E}_{t} (\ell) \partial_{\eta}^j \widehat{g}_t(k - \ell ,\eta  - \ell t)    \ d\eta  \\
   & \quad = - \Re  \sum_{k, \ell \in \mathbb{Z}^d}  \int_{\mathbb{R}^d}
   \overline{B_t(k-\ell,\eta) \partial_{\eta}^j \widehat{g}_t (k-\ell,\eta)}  
   B_t(k,\eta + \ell t) (\eta - (k-\ell)t) \cdot i\widehat{E}_{t} (-\ell) \partial_{\eta}^j \widehat{g}_t(k ,\eta  + \ell t)   \ d\eta  \\
   & \quad =  \Re \sum_{k, \ell \in \mathbb{Z}^d}  \int_{\mathbb{R}^d}
   B_t(k-\ell,\eta- \ell t) \partial_{\eta}^j \widehat{g}_t (k-\ell,\eta- \ell t)
    \overline{B_t(k,\eta) (\eta - kt) \cdot i \widehat{E}_{t} (\ell) \partial_{\eta}^j \widehat{g}_t(k ,\eta )}
   \ 
   d\eta
   = 0,
\end{align*}
where we applied the change of variable, $\eta \mapsto \eta - \ell t$ and the fact that $ i \widehat{E}_t (-\ell) = - \overline{i \widehat{E}_t (\ell) }$ for the third equality. We deduce that the expression is zero noting that $\Re (z_1 \overline{z_2}) = \Re (\overline{z_1}z_2)$ for $z_1, z_2 \in \mathbb{C}$.

Next, let us treat the first term of \eqref{sec3:NL}. We use the usual frequency decomposition as in \eqref{slpitting},
\begin{align*}
    & \abs{
    \Re i \sum_{k, \ell \in \mathbb{Z}^d}  \int_{\mathbb{R}^d}
    \overline{B_t(k,\eta) \partial_{\eta}^j \widehat{g}_t (k,\eta)}
    (B_t(k,\eta) - B_t(k-\ell,\eta-\ell t))
    (\eta - kt) \cdot \widehat{E}_{t} (\ell)
    \partial_{\eta}^j \widehat{g}_t(k - \ell ,\eta  - \ell t)     \ d\eta} \\
    & \leq
    \sum_{k, \ell \in \mathbb{Z}^d} 
    \int_{\mathbb{R}^d} 
    \left(\mathbbm{1}_{\abs{\ell, \ell t} \le \frac{1}{2} \abs{k-\ell,\eta-\ell t}} 
    + \mathbbm{1}_{\abs{k-\ell,\eta-\ell t} \le \frac{1}{2} \abs{\ell, \ell t}}
    + \mathbbm{1}_{\abs{\ell, \ell t} \approx \abs{k-\ell,\eta-\ell t}}\right) \\
    & \quad \times \abs{B_t(k,\eta) \partial_{\eta}^j \widehat{g}_t (k,\eta)} \abs{ B_t(k,\eta) - B_t(k-\ell,\eta-\ell t)} \abs{\eta - kt}  \abs{\widehat{E}_{t} (\ell)} \abs{\partial_{\eta}^j \widehat{g}_t(k - \ell ,\eta  - \ell t)}     \ d\eta \\
    & =:\mathcal{T}_{LH}+\mathcal{T}_{HL}+\mathcal{T}_{HH}.
\end{align*}

Term $\mathcal{T}_{LH}$:
since $\abs{\ell, \ell t} \le \frac{1}{2} \abs{k-\ell,\eta-\ell t}$, we get by triangle inequality,
\begin{align*}
    \jp{k,\eta} \leq \jp{\ell, \ell t} + \jp{k-\ell, \eta - \ell t} \leq 2 \jp{k-\ell, \eta - \ell t}.
\end{align*}
Moreover, we have 
\begin{align*}
\abs{\jp{k-\ell, \eta - \ell t} - \jp{k,\eta}} \leq \jp{\ell, \ell t} \leq  \frac{\jp{k-\ell, \eta - \ell t}}{2}.
\end{align*}
As a consequence, we can apply Lemma \ref{lemma_gevrey_inequality_3} with $x = \jp{k-\ell, \eta - \ell t}$ and $y = \jp{k,\eta}$ to bound the difference of the weight as follows
\begin{align*}
    & \abs{ B_t(k,\eta) - B_t(k-\ell,\eta-\ell t)}
     = \abs{e^{\lambda(t) \jp{ k, \eta }^{\gamma} } \jp{k, \eta}^{ \sigma+1}  - e^{\lambda(t) \jp{k-\ell, \eta - \ell t}^{\gamma} } \jp{k-\ell, \eta - \ell t}^{ \sigma+1}} \\
    & \qquad \leq C \jp{k-\ell, \eta - \ell t}^{ \sigma+1} \abs{e^{\lambda(t) \jp{ k, \eta }^{\gamma} }  - e^{\lambda(t) \jp{k-\ell, \eta - \ell t}^{\gamma} } } \\
    & \qquad \leq C \lambda(t) \jp{k-\ell, \eta - \ell t}^{ \sigma+1} \frac{\abs{\jp{k-\ell, \eta - \ell t} - \jp{k,\eta}}}{\jp{k-\ell, \eta - \ell t}^{1-\gamma} + \jp{k,\eta}^{1-\gamma}} e^{c  \lambda(t) \abs{\jp{k-\ell, \eta - \ell t} - \jp{k,\eta}}^\gamma} e^{ \lambda(t)  \jp{k-\ell, \eta - \ell t}^\gamma},
\end{align*}
where $c \in (0,1)$.
Then, using that $\lambda(t) \leq \lambda$ and $\abs{\jp{k-\ell, \eta - \ell t} - \jp{k,\eta}} \leq \jp{\ell, \ell t}$, we obtain
\begin{align*}
    & \abs{ B_t(k,\eta) - B_t(k-\ell,\eta-\ell t)}
     \leq C \lambda \frac{\jp{\ell, \ell t}}{\jp{k-\ell, \eta - \ell t}^{1-\gamma} + \jp{k,\eta}^{1-\gamma}} e^{c  \lambda(t) \jp{\ell,\ell t}^\gamma} B_t (k-\ell, \eta - \ell t).
\end{align*}
Hence,
\begin{align*}
    \mathcal{T}_{LH}
    & \leq C \sum_{k, \ell \in \mathbb{Z}^d}  \int_{\mathbb{R}^d} \mathbbm{1}_{\abs{\ell, \ell t} \le \frac{1}{2} \abs{k-\ell,\eta-\ell t}}  \frac{\jp{\ell, \ell t}}{\jp{k-\ell, \eta - \ell t}^{1-\gamma} + \jp{k,\eta}^{1-\gamma}} e^{c  \lambda(t) \jp{\ell,\ell t}^\gamma} B_t (k-\ell, \eta - \ell t) \\
    & \quad \times  \abs{\eta - kt}  \abs{\widehat{E}_{t} (\ell)} \abs{\partial_{\eta}^j \widehat{g}_t(k - \ell ,\eta  - \ell t)}   \abs{B_t(k,\eta) \partial_{\eta}^j \widehat{g}_t (k,\eta)}  \ d\eta.
\end{align*}
Since we have
$
  \jp{\ell, \ell t} \abs{\widehat{E}_t (\ell)}
  \leq C \jp{t} \abs{\widehat{\varrho}_t(\ell)},
$
we obtain
\begin{align*}
    \mathcal{T}_{LH}
    & \leq C \jp{t} \sum_{k, \ell \in \mathbb{Z}^d}  \int_{\mathbb{R}^d}
    \mathbbm{1}_{\abs{\ell, \ell t} \le \frac{1}{2} \abs{k-\ell,\eta-\ell t}} 
    \frac{|\eta -kt|}{\jp{k-\ell, \eta - \ell t}^{1-\gamma} + \jp{k,\eta}^{1-\gamma}} e^{c  \lambda(t) \jp{\ell,\ell t}^\gamma} 
    \abs{\widehat{\varrho}_t(\ell)} \\
    & \quad \times
    B_t (k-\ell, \eta - \ell t)
     \abs{\partial_{\eta}^j \widehat{g}_t(k - \ell ,\eta  - \ell t)}   \abs{B_t(k,\eta) \partial_{\eta}^j \widehat{g}_t (k,\eta)}  \ d\eta.
\end{align*}
Since 
\begin{equation*}
    \frac{|\eta -kt|}{\jp{k-\ell, \eta - \ell t}^{1-\gamma}+ \jp{k, \eta}^{1-\gamma}}\le \jp{t} \jp{k,\eta}^{\frac{\gamma}{2}}
    \jp{k-\ell, \eta-\ell t}^{\frac{\gamma}{2}},
\end{equation*}
we get
\begin{align*}
    \mathcal{T}_{LH}
    & \leq C \jp{t}^{-2} \sum_{k, \ell \in \mathbb{Z}^d}  \int_{\mathbb{R}^d}
    \mathbbm{1}_{\abs{\ell, \ell t} \le \frac{1}{2} \abs{k-\ell,\eta-\ell t}} 
     e^{-(1-c)  \lambda(t) \jp{\ell,\ell t}^\gamma} 
     \jp{t}^4 \jp{\ell,\ell t}^{-\sigma}
    \abs{ \widehat{A \varrho}_t(\ell)} \\
    & \times
    \jp{k-\ell, \eta-\ell t}^{\frac{\gamma}{2}}
     \abs{B_t(k,\eta) \partial_{\eta}^j \widehat{g}_t(k - \ell ,\eta  - \ell t)}
     \jp{k,\eta}^{\frac{\gamma}{2}}
      \abs{B_t(k,\eta) \partial_{\eta}^j \widehat{g}_t (k,\eta)}  \ d\eta.
\end{align*}
Hence, by \eqref{sec3:lemma_Bedrossian_estimate_1} with $\theta=\frac{d}{2}+1$, we get
\begin{align}
\label{sec3:estimate_T_LH}
    \mathcal{T}_{LH}
    & \leq C \jp{t}^{-2}  
    \norm{\jp{\nabla_x}^{-\sigma+\frac{d}{2}+5} A\varrho}_{L^2_x}
    \norm{\jp{\nabla_x , \nabla_v}^{\gamma/2} B_t(\nabla_x, \nabla_v) (v^j g_t)}_{L^2_{x,v}}^2 \nonumber \\
    & \leq C \jp{t}^{-2}  
    \norm{A\varrho}_{L^2_x}
    \norm{\jp{\nabla_x , \nabla_v}^{\gamma/2} B_t(\nabla_x, \nabla_v) (v^j g_t)}_{L^2_{x,v}}^2,
\end{align}
thanks to the assumption  $\sigma>10 + d$.

Term $\mathcal{T}_{HL}$: in this case, we do not need the difference inside the commutator, and we take the absolute values.
\begin{align*}
    \mathcal{T}_{HL} 
    & \leq   \sum_{k, \ell \in \mathbb{Z}^d} 
    \int_{\mathbb{R}^d}  
    \mathbbm{1}_{\abs{k-\ell,\eta-\ell t} \le \frac{1}{2} \abs{\ell, \ell t}}
     \left(  B_t(k,\eta) + B_t(k-\ell,\eta-\ell t) \right)\\
    & \quad \times \abs{B_t(k,\eta) \partial_{\eta}^j \widehat{g}_t (k,\eta)}
    \abs{\eta - kt}  \abs{\widehat{E}_{t} (\ell)} \abs{\partial_{\eta}^j \widehat{g}_t(k - \ell ,\eta  - \ell t)}     \ d\eta =: \mathcal{T}_{HL}^{(1)} + \mathcal{T}_{HL}^{(2)}.
\end{align*}
We estimate the first term as follows.
\begin{align*}
    \mathcal{T}_{HL}^{(1)}
    & :=  \sum_{k, \ell \in \mathbb{Z}^d} 
    \int_{\mathbb{R}^d}  
    \mathbbm{1}_{\abs{k-\ell,\eta-\ell t} \le \frac{1}{2} \abs{\ell, \ell t}}
    B_t(k,\eta)  \abs{B_t(k,\eta) \partial_{\eta}^j \widehat{g}_t (k,\eta)}
    \abs{\eta - kt}  \abs{\widehat{E}_{t} (\ell)} \abs{\partial_{\eta}^j \widehat{g}_t(k - \ell ,\eta  - \ell t)}     \ d\eta \\
    & \leq C \sum_{k, \ell \in \mathbb{Z}^d} 
    \int_{\mathbb{R}^d}  
    \mathbbm{1}_{\abs{k-\ell,\eta-\ell t} \le \frac{1}{2} \abs{\ell, \ell t}}
    \left( \jp{\ell, \ell t}^{\sigma+1} + \jp{k-\ell,\eta-\ell t}^{\sigma+1} \right) e^{\jp{\ell, \ell t}^\gamma} e^{\jp{k-\ell,\eta-\ell t}^\gamma}
    \\
    & \quad \times \abs{B_t(k,\eta) \partial_{\eta}^j \widehat{g}_t (k,\eta)}
    \abs{\eta - kt}  \abs{\widehat{E}_{t} (\ell)} \abs{\partial_{\eta}^j \widehat{g}_t(k - \ell ,\eta  - \ell t)}     \ d\eta,
\end{align*}
where we used the triangle inequality for $B_t(k,\eta)$.
For $\jp{\ell,\ell t}^{\sigma+1}$, we use 
$
  \jp{\ell, \ell t} \abs{\widehat{E}_t (\ell)}
  \leq C \jp{t} \abs{\widehat{\varrho}_t(\ell)}
$
and 
\begin{align}
\label{sec3:eta-kt}
    \abs{\eta - kt} \leq \abs{\eta - \ell t} + \abs{\ell - k} \jp{t} \leq \jp{t} \jp{k - \ell, \eta-\ell t}.
\end{align}
For $\jp{k-\ell,\eta-\ell t}^{\sigma+1}$, we use the same bound for the electric field and $\abs{\eta - kt} \leq  \jp{t} \jp{\ell,\ell t}$ thanks to \eqref{sec3:eta-kt} and the region of the indicator function.
Therefore,
\begin{align*}
    &\mathcal{T}_{HL}^{(1)} \\
    & \leq C \jp{t}^2 \sum_{k, \ell \in \mathbb{Z}^d} 
    \int_{\mathbb{R}^d}  
    \mathbbm{1}_{\abs{k-\ell,\eta-\ell t} \le \frac{1}{2} \abs{\ell, \ell t}}
     \abs{B_t(k,\eta) \partial_{\eta}^j \widehat{g}_t (k,\eta)}
      \abs{\widehat{A \varrho}_{t} (\ell)} \jp{k - \ell, \eta-\ell t} e^{\jp{k-\ell,\eta-\ell t}^\gamma}  \abs{\partial_{\eta}^j \widehat{g}_t(k - \ell ,\eta  - \ell t)}     \ d\eta \\
      & + C \jp{t}^2 \sum_{k, \ell \in \mathbb{Z}^d} 
    \int_{\mathbb{R}^d}  
    \mathbbm{1}_{\abs{k-\ell,\eta-\ell t} \le \frac{1}{2} \abs{\ell, \ell t}} 
    \abs{B_t(k,\eta) \partial_{\eta}^j \widehat{g}_t (k,\eta)}
    \jp{\ell,\ell t}  e^{\jp{\ell, \ell t}^\gamma} \abs{\widehat{\varrho}_{t} (\ell)} \abs{B_t(k-\ell,\eta-\ell t)\partial_{\eta}^j \widehat{g}_t(k - \ell ,\eta  - \ell t)}     \ d\eta.
\end{align*}
By using \eqref{sec3:lemma_Bedrossian_estimate_2} with $\theta=\sigma$, we can bound the first term allowing extra spatial moment on $\abs{\partial_{\eta}^j \widehat{g}_t(k - \ell ,\eta  - \ell t)}$ and for the second term, we use \eqref{sec3:lemma_Bedrossian_estimate_1} with $\theta = \sigma-1$ allowing extra spatial moment on $\abs{\widehat{\varrho}_{t} (\ell)}$. Hence,
\begin{align}
\label{sec3:estimates_T_HL_1}
    \mathcal{T}_{HL}^{(1)} 
    & \leq C \jp{t}^2 \norm{A_t (\nabla_x, t\nabla_x) \varrho_t}_{L^2_x} \norm{B_t(\nabla_x, \nabla_v) (v^j g_t)}_{L^2_{x,v}}^2.
\end{align}
The second term is easier to treat. We have
\begin{align*}
    \mathcal{T}_{HL}^{(2)} 
    & :=
    \sum_{k, \ell \in \mathbb{Z}^d} 
    \int_{\mathbb{R}^d}  
    \mathbbm{1}_{\abs{k-\ell,\eta-\ell t} \le \frac{1}{2} \abs{\ell, \ell t}} \abs{B_t(k,\eta) \partial_{\eta}^j \widehat{g}_t (k,\eta)}
    \abs{\eta - kt}  \abs{\widehat{E}_{t} (\ell)} \abs{B_t(k-\ell,\eta-\ell t) \partial_{\eta}^j \widehat{g}_t(k - \ell ,\eta  - \ell t)}     \ d\eta.
\end{align*}
Note that we have again
$\abs{\eta - kt} \leq \jp{t} \jp{\ell, \ell t}$
thanks to \eqref{sec3:eta-kt} and the region of the indicator function. Therefore,
\begin{align*}
    \abs{\eta - kt} \abs{\widehat{E}_t (\ell)}
    \leq \jp{t} \jp{\ell, \ell t} \abs{\widehat{E}_t (\ell)}
  \leq C \jp{t}^2 \abs{\widehat{\varrho}_t(\ell)}.
\end{align*}
Thus,
\begin{align*}
    \mathcal{T}_{HL}^{(2)} 
    & \leq C \jp{t}^2
    \sum_{k, \ell \in \mathbb{Z}^d} 
    \int_{\mathbb{R}^d}  
    \mathbbm{1}_{\abs{k-\ell,\eta-\ell t} \le \frac{1}{2} \abs{\ell, \ell t}} \abs{B_t(k,\eta) \partial_{\eta}^j \widehat{g}_t (k,\eta)}
     \abs{\widehat{\varrho}_t(\ell)}
     \abs{B_t(k-\ell,\eta-\ell t) \partial_{\eta}^j \widehat{g}_t(k - \ell ,\eta  - \ell t)}     \ d\eta.
\end{align*}
Hence, by applying  \eqref{sec3:lemma_Bedrossian_estimate_1} with $\theta = \sigma$, we get
\begin{align}
\label{sec3:estimates_T_HL_2}
    \mathcal{T}_{HL}^{(2)} 
    & \leq C \jp{t}^2 \norm{A_t (\nabla_x, t\nabla_x) \varrho_t}_{L^2_x} \norm{B_t(\nabla_x, \nabla_v) (v^j g_t)}_{L^2_{x,v}}^2.
\end{align}

Term $\mathcal{T}_{HH}$: this one is estimated as $\mathcal{T}_{HL}$. Note that in this region, we have $\frac{1}{2}\abs{\ell, \ell t} \leq \abs{k-\ell,\eta-\ell t} \leq 2\abs{\ell, \ell t}$. Therefore, \eqref{sec3:eta-kt} also holds with a factor 2. That is
\begin{align*}
    \abs{\eta - kt} \leq \abs{\eta - \ell t} + \abs{\ell - k} \jp{t} \leq  \jp{t} \jp{k - \ell, \eta-\ell t} \leq 2 \jp{\ell, \ell t}.
\end{align*}
Then, we can repeat the same analysis by splitting $\mathcal{T}_{HH}$ with $\mathcal{T}_{HH}^{(1)}$ and $\mathcal{T}_{HH}^{(2)}$ as we did with $\mathcal{T}_{HL}$. Hence,
\begin{align}
\label{sec3:estimates_T_HH}
    \mathcal{T}_{HH} 
    & \leq C \jp{t}^2 \norm{A_t (\nabla_x, t\nabla_x) \varrho_t}_{L^2_x} \norm{B_t(\nabla_x, \nabla_v) (v^j g_t)}_{L^2_{x,v}}^2.
\end{align}

\emph{Nonlinear estimate on $\mathcal{T}^{(2)}_{NL}$:}
with the same techniques, we can treat the commutator term in the $\eta$-derivatives. In this case we have
\begin{align*}
 \abs{\mathcal{T}_{NL}^{(2)}} &\le C \sum_{k, \ell \in \mathbb{Z}^d}
 \sum_{0\le \abs{p} < \abs{j}}
    \int_{\mathbb{R}^d} 
     \abs{
     B_t(k,\eta) \partial_{\eta}^j \widehat{g}_t (k,\eta)
     B_t(k,\eta) 
     \widehat{E}_{t} (\ell) \partial_{\eta}^p
     \widehat{g}_t(k - \ell ,\eta  - \ell t) }\ d\eta\\
      &\le  C  \sum_{k, \ell \in \mathbb{Z}^d}
 \sum_{0\le \abs{p} < \abs{j}}
    \int_{\mathbb{R}^d} 
      \left(\mathbbm{1}_{\abs{\ell, \ell t} \le \frac{1}{2} \abs{k-\ell,\eta-\ell t}} 
    + \mathbbm{1}_{\abs{k-\ell,\eta-\ell t} \le \frac{1}{2} \abs{\ell, \ell t}}
    + \mathbbm{1}_{\abs{\ell, \ell t} \approx \abs{k-\ell,\eta-\ell t}}
    \right)\\
     &\times \abs{B_t(k,\eta) \partial_{\eta}^j \widehat{g}_t (k,\eta)
     B_t(k,\eta) 
     \widehat{E}_{t} (\ell) \partial_{\eta}^p
     \widehat{g}_t(k - \ell ,\eta  - \ell t)}
     \ d\eta. 
\end{align*}
We bound the first term with $\mathbbm{1}_{\abs{\ell, \ell t} \le \frac{1}{2} \abs{k-\ell,\eta-\ell t}}$, by using the triangle inequality 
\begin{align*}
    \jp{k,\eta} \leq \jp{\ell, \ell t} + \jp{k-\ell, \eta - \ell t} \leq 2 \jp{k-\ell, \eta - \ell t},
\end{align*}
and 
\begin{align*}
    B_t(k,\eta) \leq C \jp{k-\ell, \eta - \ell t}^{\sigma+1} e^{\lambda(t) \jp{k-\ell, \eta - \ell t}^{\gamma}} e^{\lambda(t) \jp{\ell, \ell t}^{\gamma}}.
\end{align*}
Therefore, using $\abs{\widehat{E}_{t} (\ell)} \leq \abs{\widehat{\varrho}_{t} (\ell)}$, we get
\begin{align}
\label{sec3:estimates_T_NL_2_a}
    & C  \sum_{k, \ell \in \mathbb{Z}^d}
    \sum_{0\le \abs{p} < \abs{j}}
    \int_{\mathbb{R}^d} 
    \mathbbm{1}_{\abs{\ell, \ell t} \le \frac{1}{2} \abs{k-\ell,\eta-\ell t}} \abs{B_t(k,\eta) \partial_{\eta}^j \widehat{g}_t (k,\eta)
     B_t(k,\eta) 
     \widehat{E}_{t} (\ell) \partial_{\eta}^p
     \widehat{g}_t(k - \ell ,\eta  - \ell t)}
     \ d\eta \nonumber\\
    & \leq  C  \sum_{k, \ell \in \mathbb{Z}^d}
 \sum_{0\le \abs{p} < \abs{j}}
    \int_{\mathbb{R}^d} 
    \mathbbm{1}_{\abs{\ell, \ell t} \le \frac{1}{2} \abs{k-\ell,\eta-\ell t}} \abs{B_t(k,\eta) \partial_{\eta}^j \widehat{g}_t (k,\eta)} \\
    &\qquad \times  e^{\lambda(t) \jp{\ell, \ell t}^{\gamma}} \abs{\widehat{\varrho}_{t} (\ell)}
     B_t (k-\ell,\eta-\ell t)
     \abs{\partial_{\eta}^p
     \widehat{g}_t(k - \ell ,\eta  - \ell t)}
     \ d\eta \nonumber \\
     & \leq C \norm{A_t (\nabla_x, t\nabla_x) \varrho_t}_{L^2_x} \norm{B_t(\nabla_x, \nabla_v) (v^j g_t)}_{L^2_{x,v}}^2,
\end{align}
where we used \eqref{sec3:lemma_Bedrossian_estimate_1} with $\theta = \sigma$ for the last inequality.

The second term with $\mathbbm{1}_{\abs{k-\ell,\eta-\ell t} \le \frac{1}{2} \abs{\ell, \ell t}}$ can be treated exactly as $\mathcal{T}_{HL}^{(1)}$ above. Note, however, that in this case, we do not have the factor $\abs{\eta - kt}$ and therefore only have one power of $\jp{t}$ and not $\jp{t}^2$ as before. That is  
\begin{align}
\label{sec3:estimates_T_NL_2_b}
    & C  \sum_{k, \ell \in \mathbb{Z}^d}
 \sum_{0\le \abs{p} < \abs{j}}
    \int_{\mathbb{R}^d} 
    \mathbbm{1}_{\abs{k-\ell,\eta-\ell t} \le \frac{1}{2} \abs{\ell, \ell t}} \abs{B_t(k,\eta) \partial_{\eta}^j \widehat{g}_t (k,\eta)
     B_t(k,\eta) 
     \widehat{E}_{t} (\ell) \partial_{\eta}^p
     \widehat{g}_t(k - \ell ,\eta  - \ell t)}
     \ d\eta \nonumber \\
     & \leq C \jp{t} \norm{A_t (\nabla_x, t\nabla_x) \varrho_t}_{L^2_x} \norm{B_t(\nabla_x, \nabla_v) (v^j g_t)}_{L^2_{x,v}}^2.
\end{align}
Using the same argument as before, the term with $\mathbbm{1}_{\abs{\ell, \ell t} \approx \abs{k-\ell,\eta-\ell t}}$ is treated similarly.

Hence, by collecting the estimates 
\eqref{sec3:estimate_TL},
\eqref{sec3:estimate_T_LH},
\eqref{sec3:estimates_T_HL_1},
\eqref{sec3:estimates_T_HL_2},
\eqref{sec3:estimates_T_HH},
\eqref{sec3:estimates_T_NL_2_a},
and \eqref{sec3:estimates_T_NL_2_b} we get estimate \eqref{sec3:bound T_L+T_NL}.
\end{proof}

\section{Construction of solutions: proof of Theorem \ref{Main_thm} }
\label{iterative scheme}

In this Section, we conclude our analysis, constructing a solution to the Vlasov-type system \eqref{nlequation_g} with given asymptotic state $g_\infty$ as in \eqref{norm_g_infty}.

Prior to this, we state two useful results to treat the nonlinearity inside the Poisson coupling in \eqref{nlequation_g}.

\begin{lem}
\label{sec3:lemma_h(U)}
Let $h: \br \rightarrow \br$ be an analytic function with analyticity radius $R>0$,
\begin{align*}
    h(y) = \sum_{n \geq 2} a_n y^n, \quad \{a_n\} \subset \mathbb{R},
\end{align*}
and we define $\widetilde{h}(y):= \sum_{n \geq 2} \vert a_n \vert y^n$.
Let $\om: \mathbb{T}^d \to \mathbb{R} $ be a Gevrey function such that 
$\left \Vert A_z(\nabla_x,t \nabla_x)  \om
\right \Vert_{ L_x^2} <R$, 
where $A_z(\nabla_x, t \nabla_x)$ is defined in \eqref{sec1:A_z} with $\sigma>d/2$.
   Then
    \begin{equation*}
        \left \Vert A_z(\nabla_x,t \nabla_x) h(\om)
        \right \Vert_{L_x^2} \leq \tilde{h} \left(C \left \Vert A_z(\nabla_x,t \nabla_x)  \om
        \right \Vert_{ L_x^2} \right),
    \end{equation*}  
    for a constant $C$.
\end{lem}

\begin{proof}
Let us first prove the inequality in the case $h(y)=y^2$:
\begin{align}
\label{Algebra_x2}
     \left \Vert  A_z(\nabla_x,t \nabla_x)   \om^2
   \right \Vert_{ L_x^2} 
   & \leq C \left \Vert   A_z(\nabla_x,t \nabla_x)  \om
   \right \Vert_{ L_x^2}^2.
\end{align}
By direct computation, we have
\begin{align*}
    &\left( 
    \sum_{k \in \bz^d}  A_z^2(k,kt) \vert \widehat{\om^2}(k) \vert^2 \right)^{\frac{1}{2}}
   = \left( \sum_{k \in \bz^d}  A_z^2(k,kt)  \left \vert \sum_{\ell \in \bz^d} \widehat{\om}(k-\ell) \widehat{\om}(\ell) \right \vert^2 \right)^{\frac{1}{2}} \\
   & \quad= \left( \sum_{k \in \bz^d} \left \vert \sum_{\ell \in \bz^d} A_z(k-\ell ,(k-\ell) t) \widehat{\om}(k-\ell)   A_z(\ell ,\ell t) \widehat{\om}(\ell) \frac{ A_z(k,kt)}{A_z(\ell ,\ell t)A_z(k-\ell ,(k-\ell) t)} \right \vert^2 \right)^{\frac{1}{2}}.
\end{align*}
We estimate the sum over $\ell$ by using Cauchy-Schwarz and the following claim
\begin{equation}
\label{sec2:Lemma_weight_A}
S:=\sum_{\ell \in \mathbb{Z}^d} A_z^{2}(k ,k t)A_z^{-2}(\ell ,\ell t)A_z^{-2}(k-\ell ,(k-\ell) t)  \leq C,
\end{equation}
which follows by  inequalities
$$
e^{z \langle k, kt \rangle^{\gamma}} \leq e^{z \langle \ell, \ell t \rangle^{\gamma}} e^{z \langle k-\ell, (k-\ell)t \rangle^{\gamma}},
\quad 
\jp{k,kt}^\sigma \le C \left(
\jp{l,lt}^{\sigma} + \jp{k-l,(k-l)t}^\sigma
\right),
$$
and that, for $\sigma>\frac{d}{2}$,
\begin{align*}
S 
\le \sum_{\ell \in \mathbb{Z}^d }  \frac{\langle k, kt \rangle^{2\sigma} }{\langle \ell, \ell t \rangle^{2\sigma} \langle k-\ell, (k-\ell)t \rangle^{2\sigma} } 
\leq C\sum_{\ell \in \mathbb{Z}^d}
\left(
\frac{1}{\jp{\ell,\ell t}^{2\sigma}} + \frac{1}{\jp{k-\ell, (k-\ell)t}^{2\sigma}}
\right)
< +\infty.
\end{align*}
Therefore, by \eqref{sec2:Lemma_weight_A},
\begin{align*}
\norm{A_z(\nabla_x,t \nabla_x)  \om^2}_{L_x^2}
    & \leq C \left( \sum_{k \in \bz^d}  \sum_{\ell \in \bz^d} A_z^2(k-\ell ,(k-\ell) t)  \abs{ \widehat{\om}(k-\ell)}^2 A_z^2(\ell ,\ell t)  \abs{ \widehat{\om}(\ell)}^2  \right)^{\frac{1}{2}}\\
   & = C \left( \sum_{\ell \in \bz^d} A_z^2(\ell ,\ell t)    \abs{ \widehat{\om}(\ell)}^2  \left( \sum_{k \in \bz^d}  A_z^2(k-\ell ,(k-\ell) t)  \abs{ \widehat{\om}(k-\ell)}^2  \right) \right)^{\frac{1}{2}} \\
   & = C\sum_{k \in \bz^d}   A_z^2(k ,kt)  \abs{ \widehat{\om}(k)}^2   = C \left \Vert  A_z(\nabla_x,t \nabla_x)  \om
   \right \Vert_{ L_x^2}^2.
\end{align*}
Iterating \eqref{Algebra_x2} for $h(y)=y^n$ with $n \in \bn_{\geq 2}$, we get
\begin{align}
\label{Algebra_xn}
      \left \Vert  A_z(\nabla_x,t \nabla_x)  \om^n
   \right \Vert_{ L_x^2} 
   \leq  C^{n-1}\left \Vert  A_z(\nabla_x,t \nabla_x)  \om
   \right \Vert_{ L_x^2}^n.
\end{align}
Using that
$	h(y) = \sum_{n \geq 2} a_n y^n$,
we conclude that
\begin{align*}
    \left \Vert  A_z(\nabla_x,t \nabla_x) 
    h(\om)
   \right \Vert_{L_x^2} 
    \leq   \sum_{n \geq 2} \abs{a_n} \left \Vert A_z(\nabla_x,t \nabla_x) \om^n 
   \right \Vert_{L_x^2}
   \leq  \widetilde{h} \left(C \left \Vert A_z(\nabla_x,t \nabla_x)     \om
   \right \Vert_{L_x^2} \right),
\end{align*}
where in the last inequality, we used \eqref{Algebra_xn}.
Note that the right-hand side in the last formula is well-defined thanks to the assumption $\left \Vert A_z(\nabla_x,t \nabla_x)     \om
   \right \Vert_{L_x^2}<RC^{-1}$.

\end{proof}

To prove Theorem \ref{Main_thm}, we also need a well-posedness result in the Gevrey setting for the nonlinear equation
\begin{align*}
    -\Delta u(x) + \beta u(x) + h( u)
  =q(x),
  \quad
  x \in \mathbb{R}^d, \quad \beta\ge 0,
\end{align*}
with $q(x)$ a given source term with small Gevrey norm.
For our purposes, it is convenient to reformulate the problem in terms of the modified density 
$\varrho:=(\beta-\Delta)u$ and study the equivalent equation
\begin{align}
\label{sec3:varro_phi}
    \varrho(x)=
    q(x)
    -h
    \left (
    (\beta- \Delta)^{-1}
    \varrho \right).
\end{align}
 We denote by $\mathcal{B}_r$ the closed ball of Gevrey functions $\om$ such that $\norm{A_z(\nabla_x,t\nabla_x) \om}_{L_x^2}\le r$.
\begin{prop}
\label{sec4:PoissonGevrey}
    Let $h(y)=\sum_{n \geq 2} a_n y^n$ be an analytic function with analyticity radius $R>0$ and consider a Gevrey source term 
    $q\in \mathcal{B}_{\varepsilon}$,
    for $\varepsilon, z>0$ and $\sigma>\frac{d}{2}$. 
    If $\varepsilon$ is
    sufficiently small, there exists $r(\varepsilon)$ such that
    $r(\varepsilon)=\mathcal{O}(\varepsilon)$ as $\varepsilon \to 0$ and so that
    \eqref{sec3:varro_phi} admits a unique Gevrey solution $\varrho \in \mathcal{B}_{r(\varepsilon)}$.
\end{prop}

\begin{proof}
   For a given $q \in \mathcal{B}_\varepsilon$, we define the map
\begin{equation*}
\M_q: \mathcal{B}_r \to \mathcal{B}_r, \quad \omega \mapsto \M_q\left( \omega \right):=
q(x)
- h\left ( (\beta- \Delta)^{-1} \om \right).
\end{equation*}
This map is well-defined if $\varepsilon$ and $r$ are sufficiently small. Indeed,
\begin{align*}
     \norm{A_z(\nabla_x,t\nabla_x) \M_q ( \omega )}_{L_x^2} 
     & \leq \norm{A_z(\nabla_x,t\nabla_x) q}_{L_x^2} 
    + \norm{A_z(\nabla_x,t\nabla_x) h\left ( (\beta- \Delta)^{-1} \om \right)}_{L_x^2} \\
    & \leq \varepsilon
    + \tilde{h}\left( C\norm{A_z(\nabla_x,t\nabla_x)\left((\beta- \Delta)^{-1} \om\right)}_{L_x^2} \right),
\end{align*}
where we used that $q \in \mathcal{B}_\varepsilon$
and Lemma \ref{sec3:lemma_h(U)} to take $h$ out of the $L_x^2$ norm. 
For the second term, note that on the Fourier side
\begin{align}
    \label{sec4:FT_omega}
    \abs{\reallywidehat{\left((\beta- \Delta)^{-1} \om\right)}(k)} = \abs{\frac{1}{\beta+\abs{k}^2} \widehat{\om}(k)} \leq 
    \abs{\widehat{\om}(k)},
\end{align}
we get,
\begin{align*}
     \norm{A_z(\nabla_x,t\nabla_x) \M_q ( \omega )}_{L_x^2} 
     \le \varepsilon
    + \tilde{h}\left( C\norm{A_z(\nabla_x,t\nabla_x) \om}_{L_x^2} \right)
     \leq \varepsilon + Cr^2 ,
\end{align*}
where we used that $\om \in \mathcal{B}_r$, and $\tilde{h}(x) = \mathcal{O}(x^2)$ for $x$ small. It follows that,
\begin{align*}
     \norm{A_z(\nabla_x,t\nabla_x) \M_q ( \omega )}_{L_x^2} \leq r,
\end{align*}
provided $\varepsilon$ is sufficiently small
and $r \in (r_-(\varepsilon), r_+(\varepsilon))$, 
$r_{\pm}(\varepsilon):=(2C)^{-1}(1 \pm \sqrt{1-4C\varepsilon})$.

We now want to show that $\M_q$ is a contraction with respect to the previously defined Gevrey norm
so that, by a fixed point argument, we get the existence of a unique solution of \eqref{sec3:varro_phi} inside $\mathcal{B}_r$. 

Let $\om_1, \om_2 \in \mathcal{B}_r$,  then
\begin{align*}
    \norm{A_z(\nabla_x, t \nabla_x)\left[\M_q(\om_1)-\M_q(\om_2)\right]}_{L^2_x}
    &\leq \norm{A_z(\nabla_x, t \nabla_x) \left[
    h\left( (\beta- \Delta)^{-1} \om_1\right) 
    - h\left ( (\beta- \Delta)^{-1} \om_2 \right)
    \right]}_{L^2_x} \\
    & \leq \sup_{|\xi| \le r}
    \abs{h'(\xi)}
    \norm{
    A_z(\nabla_x, t \nabla_x)
    \left[
    (\beta- \Delta)^{-1}( \om_1 -  \om_2)
    \right]}_{L^2_x}.
\end{align*}
Then, using \eqref{sec4:FT_omega} and $h'(x) = \mathcal{O}(x)$ for $x$ small, we have
\begin{align*}
    \norm{A_z(\nabla_x, t \nabla_x)\left[\M_q(\om_1)-\M_q(\om_2)\right]}_{L^2_x}
    & \leq C r  \norm{
    A_z(\nabla_x, t \nabla_x)
    ( \om_1 -  \om_2)
    }_{L^2_x}.
\end{align*}
Therefore, for $r(\varepsilon) \in (r_-(\varepsilon), r_+(\varepsilon))$ sufficiently small, there exists a unique $\varrho \in \mathcal{B}_{r(\varepsilon)}$ such that
\begin{align*}
    \varrho = q(x)
- h\left ( (\beta- \Delta)^{-1} \varrho \right).
\end{align*}
\end{proof}

\begin{proof}[Proof of Theorem \ref{Main_thm}]
Let 
$X:=C([0,+\infty);\mathcal{G}^{\bar\lambda})$ be the space of continuous-in-time Gevrey functions in the norm given by \eqref{sec1:normg} with Gevrey parameter $\bar\lambda< \lambda(0)$.

We consider the closed subset $\mathcal{C}_{\varepsilon}\subset X $: the set of Gevrey functions $\phi$ in $X$ such that
$$
N[\phi]:=N_1[\phi]+N_2[q_\phi]\le\varepsilon, \quad 
N_1[\phi]:=\sup_{t \in [0,+\infty)}\left\Vert \jp{v}^M B_t(\nabla_x,\nabla_v) 
   \phi_t
   \right\Vert_{ L_{x,v}^2},
N_2[q_\phi]:=
   \left \Vert \langle
        t \rangle^b A_t(\nabla_x,t \nabla_x) q_\phi
        \right \Vert_{L_t^2 L_x^2},
$$
where
$q_\phi(t,x):=\int \phi(t, x-tv,v) dv$. 
Note that thanks to the Sobolev embedding,
\begin{equation}
\label{sec4:densityest}
    \norm{A_t(\nabla_x,t \nabla_x)\int \phi(t,x-tv,v) dv }_{L_x^2} \leq N_1[\phi] \le \varepsilon.
\end{equation}

We introduce the map 
\begin{equation*}
\mathcal{F}: \mathcal{C}_\varepsilon \to \mathcal{C}_\varepsilon, \quad \{\phi_t\}_{t\in [0,+\infty)} \mapsto \mathcal{F}\left( \{\phi_t\}\right):= \{\psi_t\}_{t \in [0,+\infty)},
\end{equation*}
where $\psi$ is the unique solution to the linear problem
\begin{eqnarray}
\label{sec4:linear_system}
  \left \{
  \begin{array}{l}
  \partial_t \psi(t,x,v) +
  E_\phi(t,x+vt) \cdot(\nabla_v - t \nabla_x)\psi(t,x,v)=
    - E_{\psi}(t,x+vt) \cdot \nabla_v \mu(v),\\
  E_\phi(t,x) = - \nabla u_\phi(t,x), 
  \qquad -\Delta u_\phi(t,x) + \beta u_\phi(t,x) + h( u_\phi)(t,x)
  =q_\phi(t,x),\\
  E_\psi(t,x)=-\nabla u_\psi(t,x), \qquad 
  -\Delta u_\psi(t,x) + \beta u_\psi(t,x) + h( u_\phi)(t,x)
  =q_\psi(t,x),\\
  \lim_{t \to \infty} \|\psi_t - g_\infty\|_{L^\infty(\mathbb{T}^d \times \mathbb{R}^d)}=0,
  \end{array}
  \right.
\end{eqnarray}
where
$g_\infty$ is a Gevrey function with Gevrey parameter $\lambda>\lambda(t)$ such that
$\int_{\mathbb{T}^d\times \mathbb{R}^d} g_\infty(x,v) dx dv=0$
and satisfying \eqref{norm_g_infty}. 

The map is well-defined, provided that $\varepsilon$ is sufficiently small, thanks to Proposition \ref{sec4:PoissonGevrey} and the a priori estimates of Proposition \ref{sec2:prop_estimate_S} and Proposition \ref{sec3:prop1}. Indeed, starting from the a prior estimate \eqref{sec3:prop1_estimate} for the linear system \eqref{sec4:linear_system} we get
\begin{align}
\label{sec4:a prioriestimate}
     & \frac{1}{2}\frac{d}{dt} \sum_{0 \leq \abs{j} \leq M}\norm{B_t(\nabla_x, \nabla_v) (v^j \psi_t)}_{L^2_{x,v}}^2 
        + \left( \jp{t}^{-2}\norm{ A_t(\nabla_x, t\nabla_x)\varrho_\phi}_{L^2_x} 
        - \dot{\lambda}(t) \right) \norm{\jp{\nabla_x , \nabla_v}^{\gamma/2} B_t(\nabla_x, \nabla_v) (v^j \psi_t)}_{L^2_{x,v}} \nonumber \\
        & \quad \geq - \jp{t}^2 \norm{A_t (\nabla_x, t\nabla_x) \varrho_\phi}_{L^2_x}\sum_{0 \leq \abs{j} \leq M}\norm{B_t(\nabla_x, \nabla_v) (v^j \psi_t)}_{L^2_{x,v}}^2 \nonumber \\
        & \qquad - \jp{t}\norm{A_t (\nabla_x, t\nabla_x)  \varrho_\psi}_{L^2_x}\sum_{0 \leq \abs{j} \leq M}\norm{B_t(\nabla_x, \nabla_v) (v^j \psi_t)}_{L^2_{x,v}},
\end{align}
 where $\varrho_\phi$ is solution of the equation
    $\varrho_\phi=
    q_\phi
    -h
    \left (
    (\beta- \Delta)^{-1}
    \varrho_\phi \right)$ and $\varrho_\psi:=q_\psi-
    h
    \left (
    (\beta- \Delta)^{-1}
    \varrho_\phi \right)$.
    Note that, by Proposition \ref{sec4:PoissonGevrey} and thanks to \eqref{sec4:densityest}, 
    $\varrho_\phi$ is well-defined and
    that
    $$
    \norm{A_t (\nabla_x, t\nabla_x) \varrho_\phi}_{L^2_x} \le C\varepsilon.
    $$
Moreover, note that in \eqref{sec4:a prioriestimate} the term $E_{\psi}(t,x+vt) \cdot \nabla_v \mu(v)$ corresponds to the linear term in the a priori estimate \eqref{sec3:prop1_estimate} and gives the term with the $\jp{t}$ in \eqref{sec4:a prioriestimate}. While $E_\phi(t,x+vt) \cdot(\nabla_v - t \nabla_x)\psi(t,x,v)$ corresponds to the nonlinear term in the a priori estimate \eqref{sec3:prop1_estimate} and gives the terms with $\jp{t}^{-2}$ and $\jp{t}^2$ in \eqref{sec4:a prioriestimate}.
Then, by using the equivalence of norm \eqref{sec3:equi_norm}, and the fact that
\begin{align*}
    \left(\jp{t}^{-2} \norm{A_t (\nabla_x, t\nabla_x) \varrho_\phi}_{L^2_x}- \dot\lambda(t)\right) <0,
\end{align*}
for $\delta \ll 1$, we obtain
\begin{align*}
    \frac{d}{dt} \norm{\jp{v}^M B_t(\nabla_x, \nabla_v) \psi_t}_{L^2_{x,v}} & \geq - \jp{t}^2 \norm{A_t (\nabla_x, t\nabla_x) \varrho_\phi}_{L^2_x} \norm{\jp{v}^M B_t(\nabla_x, \nabla_v) v^j \psi_t}_{L^2_{x,v}} \nonumber \\
        & \qquad - \jp{t}\norm{A_t (\nabla_x, t\nabla_x)  \varrho_\psi}_{L^2_x}.
\end{align*}
Integrating from $t$ to $+\infty$ and rearranging terms, we get
    \begin{align}
    \label{sec4:apriori}
     & \norm{\jp{v}^M B_t(\nabla_x, \nabla_v)  \psi_t}_{L^2_{x,v}}
    \le
    \norm{\jp{v}^M B_{\lambda_\infty}(\nabla_x, \nabla_v)  g_\infty}_{L^2_{x,v}} \nonumber\\
    &+C\int_t^{+\infty}\jp{s}^2 \norm{A_s (\nabla_x, s\nabla_x) \varrho_\phi}_{L^2_x}\norm{\jp{v}^M B_s(\nabla_x, \nabla_v)  \psi_s}_{L^2_{x,v}}ds
    \nonumber \\
    &+ C\int_t^{+\infty}\jp{s}\norm{A_s (\nabla_x, s\nabla_x)  \varrho_\psi}_{L^2_x}
    ds.
    \end{align}
Next, by using Cauchy-Schwarz in time with $b>6$, we have 
    \begin{align}
    \label{sec4:stima1}
        N_1[\psi]\le C \|g_\infty\|_{\mathcal{G}}+ CN_2[\varrho_\phi]N_1[\psi]+ CN_2[\varrho_\psi],
    \end{align}
    where 
    $
    \|g_\infty\|_{\mathcal{G}}:=
    \left \Vert
   \langle v \rangle^M e^{\lambda \langle \nabla_x, \nabla_v  \rangle^\gamma} 
   \langle \nabla_x,\nabla_v \rangle^{\sigma+b} g_\infty
   \right \Vert_{L_{x,v}^2}.$
    By Proposition \ref{sec2:prop_estimate_S}, we obtain
    \begin{align}
    \label{sec4:stima2}
        N_2[\varrho_\psi]
    \le
   \|g_\infty\|_{\mathcal{G}}+
   CN_2[h\left(
    (\beta-\Delta)^{-1}\varrho_\phi
    \right)]
    + CN_2[\varrho_\phi]N_1[\psi].
    \end{align}
Note, thanks to Lemma \ref{sec3:lemma_h(U)},
\begin{align*}
    N_2[h\left( (\beta-\Delta)^{-1}\varrho_\phi \right)]
    & = \norm{\jp{t}^b \norm{A_t(\nabla_x,t \nabla_x) 
    h\left( (\beta-\Delta)^{-1}\varrho_\phi
    \right)}_{L_x^2}}_{L_t^2} \\
    & \leq \norm{\jp{t}^b \Tilde{h} \left( \norm{A_t(\nabla_x,t \nabla_x) 
    \left( (\beta-\Delta)^{-1}\varrho_\phi
    \right)}_{L_x^2} \right)}_{L_t^2} \\
    & \leq \norm{\jp{t}^b \Tilde{h} \left( \norm{A_t(\nabla_x,t \nabla_x) 
    \varrho_\phi}_{L_x^2} \right)}_{L_t^2}.
\end{align*}
 Then, by linearization of $\tilde h(x)$, we get
\begin{align}
\label{sec4:N_2_h_varro_phi}
    N_2[h\left( (\beta-\Delta)^{-1}\varrho_\phi \right)] 
     \leq C\varepsilon \norm{\jp{t}^b  \norm{A_t(\nabla_x,t \nabla_x) 
    \varrho_\phi}_{L_x^2} }_{L_t^2}
    = C\varepsilon \norm{\jp{t}^b A_t(\nabla_x,t \nabla_x) 
    \varrho_\phi}_{L_t^2L_x^2} 
    \leq C \varepsilon^2,
\end{align}
since by assumption $\phi \in \mathcal{C}_\varepsilon$. For $N_2[\varrho_\phi]$ we have
\begin{align}
\label{sec4:N_2_varro_phi}
    N_2[\varrho_\phi] \leq N_2[q_\phi] + N_2[h\left( (\beta-\Delta)^{-1}\varrho_\phi \right)] \leq \varepsilon + C \varepsilon^2,
\end{align}
where we used \eqref{sec4:N_2_h_varro_phi} and $\phi \in \mathcal{C}_\varepsilon$. Therefore, by assumption \eqref{norm_g_infty} and inequalities \eqref{sec4:N_2_h_varro_phi} and \eqref{sec4:N_2_varro_phi} in \eqref{sec4:stima2}, we get
\begin{align}
\label{sec4:N_2_varro_psi}
    N_2[\varrho_\psi]
    \le
   C \varepsilon 
    + C\varepsilon N_1[\psi].
\end{align}
Finally, using  \eqref{norm_g_infty}, \eqref{sec4:N_2_varro_phi} and \eqref{sec4:N_2_varro_psi} in \eqref{sec4:stima1}, we obtain
\begin{align}
\label{sec4:N_1_psi}
    N_1[\psi] \leq C \varepsilon + C\varepsilon N_1[\psi].
\end{align}
Collecting the inequalities in \eqref{sec4:N_2_varro_psi} and \eqref{sec4:N_1_psi}, we get that $\psi \in \mathcal{C}_\varepsilon$, provided that $\varepsilon$ is sufficiently small.

To construct solutions to the nonlinear problem with asymptotic state $g_\infty$ , we want to prove that $\mathcal{F}$ is
contractive in the
norm 
$\mathcal{N}[\cdot]$ defined as $N[\cdot]$ but with a different $\Tilde{\lambda}(t)=\Tilde{\lambda}_\infty - C\jp{t}^{-\delta}$, $\Tilde{\lambda}_\infty<\lambda_\infty$.
Note that we have the equation
\begin{align*}
    \partial_t 
    \left(\psi^{(2)}- \psi^{(1)}\right)&
    + \left(
    (E_{\phi^{(2)}}- E_{\phi^{(1)}})(x+vt)\right)\cdot (\nabla_v -t \nabla_x )\psi^{(2)}\\
    &+ E_{\phi^{(1)}}(t,x+vt)\cdot 
    (\nabla_v -t \nabla_x)\left( \psi^{(2)}- \psi^{(1)}
    \right) \nonumber
    =-(E_{\psi^{(2)}}-E_{\psi^{(1)}})\cdot \nabla_v \mu,
\end{align*}
and that
$$\varrho_{\phi^{(2)}}-\varrho_{\phi^{(1)}}=
    \int \left[\phi^{(2)}(t,x-tv,v)
    - \phi^{(1)}(t, x-tv,v) 
    \]dv
    -h
    \left (
    (\beta- \Delta)^{-1}
    \varrho_{\phi^{(2)}}
    \right)
    +h
    \left (
    (\beta- \Delta)^{-1}
    \varrho_{\phi^{(1)}}
    \right).$$

By the a priori estimates in Proposition \ref{sec3:prop1}, we get
    \begin{align*}
        \mathcal{N}_1[\psi^{(2)}-\psi^{(1)}]\le
        C\mathcal{N}_2[\varrho_{\phi^{(2)}}-\varrho_{\phi^{(1)}}]
        N_1[\psi^{(2)}]
        + CN_2[\varrho_{\phi^{(1)}}]
        \mathcal{N}_1[\psi^{(2)}-\psi^{(1)}]+ 
        C\mathcal{N}_2[\varrho_{\psi^{(2)}}-\varrho_{\psi^{(1)}}],
    \end{align*}
   where $\mathcal{N}_1[\cdot]$ and $\mathcal{N}_2[\cdot]$ are defined as $N_1[\cdot]$ and $N_2[\cdot]$ but with $\Tilde{\lambda}(t)$. 
   Since, by Proposition \ref{sec2:prop_estimate_S}, 
    \begin{align*}
        \mathcal{N}_2[\varrho_{\psi^{(2)}}-\varrho_{\psi^{(1)}}]
    &\leq
   C\mathcal{N}_2[h\left(
    (\beta-\Delta)^{-1}\varrho_{\phi^{(2)}}
    \right)-
    h\left(
    (\beta-\Delta)^{-1}\varrho_{\phi^{(1)}}
    \right)
    ]
    +C\mathcal{N}_2[\varrho_{\phi^{(1)}}-\varrho_{\phi^{(2)}}]
      N_1[\psi^{(2)}] \nonumber \\
      &
    + CN_2[\varrho_{\phi^{(1)}}]
      \mathcal{N}_1[\psi^{(2)}-\psi^{(1)}],
    \end{align*}
    using that 
    $
     N_1[\psi^{(2)}]+ N_2[\varrho_{\phi^{(1)}}]\le C \varepsilon,
    $
    we conclude that
    $$
    \mathcal{N}[\psi^{(2)}-\psi^{(1)}] \le C \varepsilon \mathcal{N}[\phi^{(2)}-\phi^{(1)}],
    $$
    proving the contraction property, for small $\varepsilon$.
\end{proof}

\bibliographystyle{plain}
\bibliography{biblio}

\begin{thebibliography}{10}

\bibitem{Arsenev}
A.~A. Arsenev.
\newblock Existence in the large of a weak solution of {V}lasov's system of
  equations.
\newblock {\em \v{Z}. Vy\v{c}isl. Mat i Mat. Fiz.}, 15:136--147, 276, 1975.

\bibitem{Bardos_Degond_1}
C.~Bardos and P.~Degond.
\newblock Existence globale des solutions des \'{e}quations de
  {V}lasov-{P}oisson.
\newblock In {\em Nonlinear partial differential equations and their
  applications. {C}oll\`ege de {F}rance seminar, {V}ol. {VII} ({P}aris,
  1983--1984)}, volume 122 of {\em Res. Notes in Math.}, pages 1--3, 35--58.
  Pitman, Boston, MA, 1985.

\bibitem{Bardos_Degond_2}
C.~Bardos and P.~Degond.
\newblock Global existence for the {V}lasov-{P}oisson equation in {$3$} space
  variables with small initial data.
\newblock {\em Ann. Inst. H. Poincar\'{e} Anal. Non Lin\'{e}aire},
  2(2):101--118, 1985.

\bibitem{Bardos_Degond_Golse}
C.~Bardos, P.~Degond, and F.~Golse.
\newblock A priori estimates and existence results for the {V}lasov and
  {B}oltzmann equations.
\newblock In {\em Nonlinear systems of partial differential equations in
  applied mathematics, {P}art 2 ({S}anta {F}e, {N}.{M}., 1984)}, volume~23 of
  {\em Lectures in Appl. Math.}, pages 189--207. Amer. Math. Soc., Providence,
  RI, 1986.

\bibitem{Bardos}
C.~Bardos, F.~Golse, T.~Nguyen, and R.~Sentis.
\newblock The {M}axwell-{B}oltzmann approximation for ion kinetic modeling.
\newblock {\em Phys. D}, 376/377:94--107, 2018.

\bibitem{Batt_Rein}
J.~Batt and G.~Rein.
\newblock Global classical solutions of the periodic {V}lasov-{P}oisson system
  in three dimensions.
\newblock {\em C. R. Acad. Sci. Paris S\'{e}r. I Math.}, 313(6):411--416, 1991.

\bibitem{JB}
J.~Bedrossian.
\newblock Nonlinear echoes and {L}andau damping with insufficient regularity.
\newblock {\em Tunis. J. Math.}, 3(1):121--205, 2021.

\bibitem{JB_notes}
J.~Bedrossian.
\newblock A brief introduction to the mathematics of landau damping.
\newblock 11 2022.
\newblock \url{https://arxiv.org/pdf/2211.13707}.

\bibitem{BMM_13}
J.~Bedrossian, N.~Masmoudi, and C.~Mouhot.
\newblock Landau damping: paraproducts and {G}evrey regularity.
\newblock {\em Ann. PDE}, 2(1):Art. 4, 71, 2016.

\bibitem{BMM_16}
J.~Bedrossian, N.~Masmoudi, and C.~Mouhot.
\newblock Landau damping in finite regularity for unconfined systems with
  screened interactions.
\newblock {\em Comm. Pure Appl. Math.}, 71(3):537--576, 2018.

\bibitem{BCR}
D.~Benedetto, E.~Caglioti, and S.~Rossi.
\newblock Comparison between the {C}auchy problem and the scattering problem
  for the {L}andau damping in the {V}lasov-{HMF} equation.
\newblock {\em Asymptot. Anal.}, 129(2):215--238, 2022.

\bibitem{Bouchut}
F.~Bouchut.
\newblock Global weak solution of the {V}lasov-{P}oisson system for small
  electrons mass.
\newblock {\em Comm. Partial Differential Equations}, 16(8-9):1337--1365, 1991.

\bibitem{plasma_2003}
T.~J.~M. Boyd and J.~J. Sanderson.
\newblock {\em The Physics of Plasmas}.
\newblock Cambridge University Press, 2003.

\bibitem{EC_CM}
E.~Caglioti and C.~Maffei.
\newblock Time asymptotics for solutions of {V}lasov-{P}oisson equation in a
  circle.
\newblock {\em J. Statist. Phys.}, 92(1-2):301--323, 1998.

\bibitem{Cesbron_Iacobelli}
L.~Cesbron and M.~Iacobelli.
\newblock Global well-posedness of {V}lasov-{P}oisson-type systems in bounded
  domains.
\newblock {\em Anal. PDE}, 16(10):2465--2494, 2023.

\bibitem{Chen}
Z.~Chen and J.~Chen.
\newblock Moments propagation for weak solutions of the {V}lasov-{P}oisson
  system in the three-dimensional torus.
\newblock {\em J. Math. Anal. Appl.}, 472(1):728--737, 2019.

\bibitem{FR_2016}
E.~Faou and F.~Rousset.
\newblock Landau damping in {S}obolev spaces for the {V}lasov-{HMF} model.
\newblock {\em Arch. Ration. Mech. Anal.}, 219(2):887--902, 2016.

\bibitem{FOPW_2023}
P.~Flynn, Z.~Ouyang, B.~Pausader, and K.~Widmayer.
\newblock Scattering map for the {V}lasov-{P}oisson system.
\newblock {\em Peking Math. J.}, 6(2):365--392, 2023.

\bibitem{G_23}
A.~Gagnebin.
\newblock Backward problem for the 1{D} ionic {V}lasov-{P}oisson equation.
\newblock {\em Kinet. Relat. Models}, 17(2):312--330, 2024.

\bibitem{GI_23}
A.~Gagnebin and M.~Iacobelli.
\newblock Landau damping on the torus for the {V}lasov-{P}oisson system with
  massless electrons.
\newblock {\em J. Differential Equations}, 376:154--203, 2023.

\bibitem{GNR}
E.~Grenier, T.~Nguyen, and I.~Rodnianski.
\newblock Landau damping for analytic and {G}evrey data.
\newblock {\em Math. Res. Lett.}, 28(6):1679--1702, 2021.

\bibitem{Griffin_Iacobelli_R}
M.~Griffin-Pickering and M.~Iacobelli.
\newblock Global strong solutions in {$\mathbb{R}^3$} for ionic
  {V}lasov-{P}oisson systems.
\newblock {\em Kinet. Relat. Models}, 14(4):571--597, 2021.

\bibitem{Griffin_Iacobelli_torus}
M.~Griffin-Pickering and M.~Iacobelli.
\newblock Global well-posedness for the {V}lasov-{P}oisson system with massless
  electrons in the 3-dimensional torus.
\newblock {\em Comm. Partial Differential Equations}, 46(10):1892--1939, 2021.

\bibitem{Griffin_Iacobelli_summary}
M.~Griffin-Pickering and M.~Iacobelli.
\newblock Recent developments on the well-posedness theory for {V}lasov-type
  equations.
\newblock In {\em From particle systems to partial differential equations},
  volume 352 of {\em Springer Proc. Math. Stat.}, pages 301--319. Springer,
  Cham, 2021.

\bibitem{HKD}
D.~Han-Kwan.
\newblock Quasineutral limit of the {V}lasov-{P}oisson system with massless
  electrons.
\newblock {\em Comm. Partial Differential Equations}, 36(8):1385--1425, 2011.

\bibitem{HKD_HDR}
D.~Han-Kwan.
\newblock Stabilit\'e, limites singuli\`eres et conditions de contr\^ole
  g\'eom\'etrique en th\'eorie cin\'etique, 09 2017.
\newblock M\'emoire pr\'esente\'e \`a Universit\'e Paris-Diderot pour
  l’obtention de l’habilitation \`a diriger des recherches.

\bibitem{Han-Kwan_Iacobelli}
D.~Han-Kwan and M.~Iacobelli.
\newblock The quasineutral limit of the {V}lasov-{P}oisson equation in
  {W}asserstein metric.
\newblock {\em Commun. Math. Sci.}, 15(2):481--509, 2017.

\bibitem{HKNR}
D.~Han-Kwan, T.~Nguyen, and F.~Rousset.
\newblock Asymptotic stability of equilibria for screened {V}lasov-{P}oisson
  systems via pointwise dispersive estimates.
\newblock {\em Ann. PDE}, 7(2):Paper No. 18, 37, 2021.

\bibitem{HM_18}
T.~Holding and E.~Miot.
\newblock Uniqueness and stability for the {V}lasov-{P}oisson system with
  spatial density in {O}rlicz spaces.
\newblock In {\em Mathematical analysis in fluid mechanics---selected recent
  results}, volume 710 of {\em Contemp. Math.}, pages 145--162. Amer. Math.
  Soc., [Providence], RI, [2018] \copyright 2018.

\bibitem{Horst_Hunze}
E.~Horst and R.~Hunze.
\newblock Weak solutions of the initial value problem for the unmodified
  nonlinear {V}lasov equation.
\newblock {\em Math. Methods Appl. Sci.}, 6(2):262--279, 1984.

\bibitem{huang_2d}
L.~Huang, Q.-H. Nguyen, and Y.~Xu.
\newblock Nonlinear {Landau} damping for the 2d {Vlasov}-{Poisson} system with
  massless electrons around {Penrose}-stable equilibria.
\newblock 06 2022.
\newblock arXiv:2206.11744.

\bibitem{huang_sharp_2022}
L.~Huang, Q.-H. Nguyen, and Y.~Xu.
\newblock Sharp estimates for screened {Vlasov}-{Poisson} system around
  {Penrose}-stable equilibria in $\mathbb{R}^d$, $d\geq3$, 05 2022.
\newblock arXiv:2205.10261.

\bibitem{Hw_Ve}
H.~J. Hwang and J.~L.~L. Vel\'{a}zquez.
\newblock On the existence of exponentially decreasing solutions of the
  nonlinear {L}andau damping problem.
\newblock {\em Indiana Univ. Math. J.}, 58(6):2623--2660, 2009.

\bibitem{MI_2022}
M~Iacobelli.
\newblock A new perspective on {W}asserstein distances for kinetic problems.
\newblock {\em Arch. Ration. Mech. Anal.}, 244(1):27--50, 2022.

\bibitem{IPWW_2022}
A.~D. Ionescu, B.~Pausader, X.~Wang, and K.~Widmayer.
\newblock On the asymptotic behavior of solutions to the {V}lasov-{P}oisson
  system.
\newblock {\em Int. Math. Res. Not. IMRN}, (12):8865--8889, 2022.

\bibitem{IPWW_2023}
A.~D. Ionescu, B.~Pausader, X.~Wang, and K.~Widmayer.
\newblock On the stability of homogeneous equilibria in the {V}lasov-{P}oisson
  system on {$\Bbb R^3$}.
\newblock {\em Classical Quantum Gravity}, 40(18):Paper No. 185007, 32, 2023.

\bibitem{IPWW_2024_arxvi}
A.~D. Ionescu, B.~Pausader, X.~Wang, and K.~Widmayer.
\newblock Nonlinear {L}andau damping and wave operators in sharp {G}evrey
  spaces.
\newblock 2024.
\newblock \url{https://arxiv.org/pdf/2405.04473}.

\bibitem{IPWW_2024}
A.~D. Ionescu, B~Pausader, X.~Wang, and K.~Widmayer.
\newblock Nonlinear {L}andau damping for the {V}lasov-{P}oisson system in
  {$\Bbb R^3$}: the {P}oisson equilibrium.
\newblock {\em Ann. PDE}, 10(1):Paper No. 2, 78, 2024.

\bibitem{Iordanskii}
S.~V. Iordanski\u{\i}.
\newblock The {C}auchy problem for the kinetic equation of plasma.
\newblock {\em Trudy Mat. Inst. Steklov.}, 60:181--194, 1961.

\bibitem{Landau}
L.~Landau.
\newblock On the vibrations of the electronic plasma.
\newblock {\em (Russian) Akad. Nauk SSSR. Zhurnal Eksper. Teoret. Fiz},
  16:574--586, 1946.

\bibitem{LZ_2011}
Z.~Lin and C.~Zeng.
\newblock Small {BGK} waves and nonlinear {L}andau damping.
\newblock {\em Comm. Math. Phys.}, 306(2):291--331, 2011.

\bibitem{LZ_2012}
Z.~Lin and C.~Zeng.
\newblock Small {BGK} waves and nonlinear {L}andau damping (higher dimensions).
\newblock {\em Indiana Univ. Math. J.}, 61(5):1711--1735, 2012.

\bibitem{Lions_Perthame}
P.-L. Lions and B.~Perthame.
\newblock Propagation of moments and regularity for the {$3$}-dimensional
  {V}lasov-{P}oisson system.
\newblock {\em Invent. Math.}, 105(2):415--430, 1991.

\bibitem{Loeper}
G.~Loeper.
\newblock Uniqueness of the solution to the {V}lasov-{P}oisson system with
  bounded density.
\newblock {\em J. Math. Pures Appl. (9)}, 86(1):68--79, 2006.

\bibitem{MW_64}
J.~H. Malmberg and C.~B. Wharton.
\newblock Collisionless damping of electrostatic plasma waves.
\newblock {\em Phys. Rev. Lett.}, 13:184--186, Aug 1964.

\bibitem{Miot}
E.~Miot.
\newblock A uniqueness criterion for unbounded solutions to the
  {V}lasov-{P}oisson system.
\newblock {\em Comm. Math. Phys.}, 346(2):469--482, 2016.

\bibitem{CM_CV}
C.~Mouhot and C.~Villani.
\newblock On {L}andau damping.
\newblock {\em Acta Math.}, 207(1):29--201, 2011.

\bibitem{Pallard}
C.~Pallard.
\newblock Moment propagation for weak solutions to the {V}lasov-{P}oisson
  system.
\newblock {\em Comm. Partial Differential Equations}, 37(7):1273--1285, 2012.

\bibitem{PW_2021}
B.~Pausader and K.~Widmayer.
\newblock Stability of a point charge for the {V}lasov-{P}oisson system: the
  radial case.
\newblock {\em Comm. Math. Phys.}, 385(3):1741--1769, 2021.

\bibitem{Penrose}
O.~Penrose.
\newblock Electrostatic instabilities of a uniform non-maxwellian plasma.
\newblock {\em Physics of Fluids (U.S.)}, Vol: 3, 3 1960.

\bibitem{Pfaffelmoser}
K.~Pfaffelmoser.
\newblock Global classical solutions of the {V}lasov-{P}oisson system in three
  dimensions for general initial data.
\newblock {\em J. Differential Equations}, 95(2):281--303, 1992.

\bibitem{Schaffer}
J.~Schaeffer.
\newblock Global existence of smooth solutions to the {V}lasov-{P}oisson system
  in three dimensions.
\newblock {\em Comm. Partial Differential Equations}, 16(8-9):1313--1335, 1991.

\bibitem{Ukai_Okabe}
S.~Ukai and T.~Okabe.
\newblock On classical solutions in the large in time of two-dimensional
  {V}lasov's equation.
\newblock {\em Osaka Math. J.}, 15(2):245--261, 1978.

\end{thebibliography}

\end{document}